\documentclass[12pt,a4paper]{amsart}
\usepackage{amsmath,amsthm, amscd, amssymb, amsfonts, enumerate}
\usepackage[spanish]{babel}
\usepackage{inputenc}
\usepackage[T1]{fontenc}
\usepackage{hyperref}

\numberwithin{equation}{section} 

\theoremstyle{plain}
\newtheorem{teorema}{Teorema}[section]
\newtheorem{lema}[teorema]{Lema}
\newtheorem{corolario}[teorema]{Corolario}

\newtheorem{proposicion}[teorema]{Proposici\'on}

\theoremstyle{definition}
\newtheorem{definicion}[teorema]{Definici\'on}

\theoremstyle{remark}

\newcommand{\R}{\mathbb R}

\newcommand{\Li}{Li}
\newcommand{\err}{err}
\newcommand{\Err}{Err}

\DeclareMathOperator{\mo}{\,mod}

\DeclareFontFamily{OT1}{rsfs}{}
\DeclareFontShape{OT1}{rsfs}{n}{it}{<-> rsfs10}{}
\DeclareMathAlphabet{\mathscr}{OT1}{rsfs}{n}{it}

\begin{document}

\setcounter{page}{1}

\title{Primos, paridad y an\'alisis}
\author{Harald Helfgott y Adri\'an Ubis}
\address{Georg-August Universit\"at G\"ottingen\\
Mathematisches Institut\\
Bunsenstrasse 3-5\\
D-37073 G\"ottingen\\
Alemania}
\address{Universit\'e Paris Diderot\\UFR de Math\'ematiques, case 7012\\75205 Paris CEDEX
13\\ Francia
}
\email{harald.helfgott@gmail.com}
\address{Departamento de Matem\'aticas, 
Universidad Aut\'onoma de Madrid, Madrid
28049 SPAIN}
\email{adrian.ubis@uam.es}
\thanks{Estas notas corresponden al curso dictado por los autores en la escuela AGRA III, Aritm\'etica, Grupos y An\'alisis, del 9 al 20 de Julio de 2018 en C\'ordoba, Argentina.}
\date{Mayo de 2019.}

\begin{abstract}
Sea $\lambda(n)=1$ cuando $n$ tiene un n\'umero par de divisores primos (contados con multiplicidad) y $\lambda(n)=-1$ de lo contrario. En general, distinguir entre n\'umeros con un n\'umero par o impar de divisores primos es una de las tareas m\'as dif\'iciles en la teor\'ia anal\'itica de n\'umeros. Un trabajo reciente de Matom\"aki y Radziwi{\l}{\l} muestra que, en promedio, ambos existen con
la misma frecuencia a\'un en intervalos muy cortos. Este avance ya ha tenido varias aplicaciones importantes en las manos de Matom\"aki, Radziwi{\l}{\l}, Tao y Ter\"av\"ainen. Explicaremos en detalle una prueba completa del resultado original de Matom\"aki y Radziwi{\l}{\l}, as\'i como de varias aplicaciones.
\end{abstract}

\maketitle

\tableofcontents

\section{Introducci\'on}

\subsection{Primalidad y paridad}

Los n\'umeros primos son uno de los objetos de estudio principales de
la teor\'ia de n\'umeros. Los problemas cl\'asicos acerca de los primos
son cl\'asicos en parte por su dificultad: nuestras
t\'ecnicas nos permiten acercarnos a la meta, pero, una vez que se trata
de encontrar todos los casos primos y s\'olo ellos, nos vemos a menudo
frustrados.

Consideremos un ejemplo entre muchos. La
{\em conjetura fuerte de Goldbach} dice que todo n\'umero par $n\ge 4$ puede
escribirse como la suma de dos primos. R\'enyi \cite{MR0021958}
mostr\'o que todo
n\'umero par suficientemente grande puede escribirse como la suma de un
primo y un n\'umero que es el producto de a lo m\'as $K$ primos, donde
$K$ es una constante.
Luego hubo una sucesi\'on de trabajos que mejoraron la constante. Finalmente,
en \cite{MR0434997},
Jing-Run Chen logr\'o
mostrar que todo par suficientemente grande puede escribirse como la suma
de un primo y un n\'umero que es ya sea un primo o el producto de dos primos.
Empero, la conjetura fuerte de Goldbach sigue sin resolver.

La misma situaci\'on ocurre con la {\em conjetura de los primos gemelos}, la
cual plantea que hay un n\'umero infinito de primos $p$ tales que $p+2$
tambi\'en es primo. Se puede dar una cota superior al n\'umero de tales $p$
con $p\leq N$ mediante m\'etodos de {\em criba} (c\'omo veremos en los ejercicios en la secci\'on \ref{subs:paridadcribas}), pero no tenemos una cota inferior.
Como en el caso de Goldbach, hubo una sucesi\'on de trabajos llevando a
un resultado similar al de Chen; tambi\'en hay otras aproximaciones a las dos
conjeturas ({\em Goldbach d\'ebil}, {\em distancias acotadas entre primos}).
Empero, las conjeturas en s\'i siguen abiertas.

En general, distinguir entre un primo y el producto de dos primos es muy
dif\'icil. Si se usa un procedimiento de {\em criba} de tipo tradicional, hacer
tal distinci\'on es no s\'olo dif\'{\i}cil sino imposible.
En general, las cribas, por s\'i solas,
no llegan a distinguir entre n\'umeros con un n\'umero par o impar de factores
primos; se trata del {\em problema de paridad} (\S \ref{subs:paridadcribas}).
Claro est\'a, hay otros procedimientos, generalmente anal\'iticos,
que nos permiten probar algunos enunciados sobre los n\'umeros primos.

Un resultado de base de este tipo es
el {\em teorema de los n\'umeros primos}. Recordamos que nos
dice que el n\'umero $\pi(x)$ de primos entre $1$ y $x$ es aproximadamente
$x/\log x$. Para ser m\'as precisos, en la versi\'on de
Hadamard y de la Vall\'ee Poussin (1896), nos dice que
\begin{equation}\label{eq:hadavall}
  \pi(x) = \Li(x) + O\left(x e^{- c\sqrt{\log x}}\right),\end{equation}
donde $\Li(x) = \int_2^x \frac{dt}{\log t}$, $c>0$ es una constante
y $O(f(x))$ quiere decir un t\'ermino cuyo valor absoluto est\'a
acotado por $f(x)$ por una constante. 

La prueba se basa sobre las propiedades de la funci\'on zeta de Riemann
$\zeta(s)$, y, en particular sobre el hecho que sabemos que no toma el valor
$0$ cuando $s$ est\'a dentro de una cierta regi\'on. Por los mismos m\'etodos,
o usando (\ref{eq:hadavall}), podemos mostrar que el n\'umero de enteros
$n\leq x$ con un n\'umero par o impar de factores primos es asint\'oticamente
el mismo:
\begin{equation}\label{eq:hadavall2}
  \sum_{n\leq x} \lambda(n) = O\left(e^{- c\sqrt{\log x}} x\right),\end{equation}
donde $\lambda(n)$ es la {\em funci\'on de Liouville}, la
cual es la funci\'on completamente multiplicativa tal que $f(p)=-1$ para
todo n\'umero primo. El enunciado (\ref{eq:hadavall2}) tambi\'en es cierto
si $\lambda(n)$ se reemplaza por la m\'as familiar
{\em funci\'on de M\"obius} $\mu(n)$, definida como $\mu(n) = \lambda(n)$
para $n$ sin divisores cuadrados (aparte de $1$), y como $\mu(n) = 0$
si $d^2|n$ para alg\'un $d>1$.

De hecho lo m\'as importante es que la cota en (\ref{eq:hadavall2})
dividida por la cota trivial $x$ va a cero:
\[\sum_{n\leq x} \lambda(n) = o(x),\]
donde $o(f(x))$ denota cualquier funci\'on $g(x)$ tal que
$\lim_{x\to \infty} g(x)/f(x) = 0$. (Aqu\'i, $f(x) = x$.)
En otras palabras,
el promedio de $\lambda$ en el intervalo $1\leq n\leq x$ es $o(1)$, es
decir, tiende a cero.

En general, es muy interesante saber que $\lambda$ tiene promedio $o(1)$
en un intervalo o conjunto de n\'umeros, pues esto quiere decir
precisamente que sabemos que en ese conjunto el n\'umero de enteros con un
n\'umero par o impar de factores primos (contados con multiplicidad)
es asint\'oticamente el mismo. Sab\'iamos, por ejemplo, que
\[\frac{1}{H} \sum_{n=N+1}^{N+H} \mu(n) = o(1)\]
para $H\geq N^\alpha$ y $\alpha$ dentro de un cierto rango. El primer resultado
con $\alpha<1$ fue probado por Hoheisel (1930); el mejor valor de $\alpha$
conocido en nuestros d\'ias es $7/12 + \epsilon$, $\epsilon>0$ arbitrario
(\cite{MR0424726} y \cite{MR0424723}, basados en parte en \cite{MR0292774}). Tambi\'en se ten\'ian resultados ``en promedio'':
por ejemplo, se sab\'ia que
\begin{equation}\label{eq:varcota}
  \int_{X}^{2 X} \left|\sum_{x<n\leq x+h} \lambda(n)\right|^2 dx = o(X h^2)
  \end{equation}
para $h\geq X^{1/6+\epsilon}$, $\epsilon>0$ arbitrario. La desigualdad
(\ref{eq:varcota}) es equivalente al enunciado siguiente: dado $h=h(x)$
tal que $h(x)\geq x^{1/6+\epsilon}$,
\[\sum_{N<n\leq N+h(x)} \lambda(n) = o(h(x))\]
para todo entero $N\in \lbrack x, 2 x\rbrack$ fuera de un conjunto de $o(x)$
excepciones. (La equivalencia es inmediata; el lado izquierdo de
(\ref{eq:varcota}) es una varianza.)

La {\em conjetura de Chowla} nos dice que, para $h_1,h_2,\dotsc,h_k$
enteros distintos,
\begin{equation}\label{eq:chowla}
  \lim_{N\to \infty} \frac{1}{N} \sum_{n\leq N} \lambda(n+h_1) \dotsb
\lambda(n+h_k)
= 0.\end{equation}
Cuando $k>1$, la conjetura est\'a abierta y se considera muy dif\'icil.
En general, se cree que, para todo polinomio $P\in \mathbb{Z}\lbrack x\rbrack$
no constante,
\[\lim_{N\to \infty} \frac{1}{N} \sum_{n\leq N} \mu(P(x)) = 0.\]
(Las conjeturas de {\em Hardy-Littlewood} y de  {\em Schinzel} son
enunciados relacionados para los n\'umeros primos.)
Nuevamente, no hay caso probado con $\deg P > 1$, si bien se tienen resultados
para polinomios en dos variables.


Una puerta fue abierta recientemente por Kaisa Matom\"aki y Maksym
Radziwi{\l}{\l}
\cite{MR3488742}. Ha conducido ya a numerosos resultados.

\subsection{Resultados de Matom\"aki, Radziwi{\l}{\l} y Tao}

El siguiente resultado fue un gran avance, puesto que, como hemos visto,
la conclusi\'on se sab\'ia s\'olo bajo la condici\'on $h(x)\geq x^{1/6+\epsilon}$,
mucho m\'as fuerte que $h(x)\to \infty$.
\begin{teorema}\label{te:mr1} 
  Sea $h=h(x)$ tal que $h(x)\to \infty$ cuando
  $x\to \infty$.
  Entonces, cuando $X\to \infty$,
  \begin{equation}\label{eq:cancellamb}
    \left|\sum_{N< n\leq N+h(X)} \lambda(n)\right| = o(h(X)) \end{equation}
  para todo entero $N\in \lbrack X, 2 X\rbrack$
  fuera de un conjunto de $o(X)$
  excepciones.
\end{teorema}

En verdad, con algo de trabajo adicional, Matom\"aki y
Radziwi{\l}{\l} demuestran un an\'alogo del Teorema \ref{te:mr1} 
para todo $f:\mathbb{Z}\to \mathbb{R}$
multiplicativo que satisfaga $|f(n)|\leq 1$ para todo
$n$ \cite[Thm.~1]{MR3488742}:
\[
\left|\frac{1}{h(X)}\sum_{N< n\leq N+h(X)} f(n) -
\frac{1}{X} \sum_{X< n\leq 2 X} f(n) \right| = o(1)
\]
para todo $N\in \lbrack X, 2 X\rbrack$
  fuera de $o(X)$
  excepciones. (Hay generalizaciones tambi\'en para
$f$ con valores complejos \cite[Thm.~A.1]{MR3435814}.)

(Por cierto, es plausible que (\ref{eq:cancellamb}) sea cierto
para {\em todo} $N\in \lbrack X, 2 X\rbrack$, con tal que, digamos,
$h(x)\geq C (\log x)^2$. Empero, es f\'acil construir contraejemplos a tal
enunciado con $f$ multiplicativa, $f\ne \lambda$.)

As\'i como el Teorema \ref{te:mr1} es un resultado sobre la cancelaci\'on en
intervalos cortos en promedio, se puede probar la conjetura de Chowla en
promedio.
\begin{teorema}\label{te:mrt1}$($\cite[Thm.~1.1]{MR3435814}$)$
  Sea $k\geq 1$ arbitrario. Sea $h=h(x)$ tal que $h(x)\to \infty$ cuando
  $x\to \infty$. Entonces, cuando $N\to \infty$,
  \[
  \sum_{n\leq N} \lambda(n+h_1) \dotsb \lambda(n+h_k) = o(N)
  \]
  para enteros $(h_1,\dotsc,h_k)$, $1\leq h_i\leq h(N)$, cualesquiera,
  fuera de un conjunto de $o\left(h(N)^k\right)$ excepciones.
\end{teorema}
La prueba se basa en la del Teorema \ref{te:mr1}, v\'ia el m\'etodo del c\'irculo.

Los Teoremas \ref{te:mr1} y \ref{te:mrt1} admiten variantes cuantitativas.
A\'un para $h(X)$ peque\~no, los m\'etodos de
Matom\"aki, Radziwi{\l}{\l} y Tao pueden dar una cota
de la forma
$$O((\log \log h(X))/\log h(X)).$$ Los t\'erminos de error de esta forma
son t\'ipicos de pruebas en las cuales el caso de los primos se ve como una
excepci\'on. Las pruebas que veremos no son una excepci\'on; de hecho, ponen
a los enteros que no tienen una factorizaci\'on t\'ipica en el t\'ermino
de error.

(En verdad, nosotros probaremos cotas del tipo $O(1/(\log h(x))^\alpha)$,
$\alpha>0$. Nuestro \'enfasis ser\'a en dar una exposici\'on clara
y concisa, y no en conseguir necesariamente las mejores cotas que los
m\'etodos permiten.)

Los m\'etodos y resultados de Matom\"aki y Radziwi{\l}{\l} han tenido varias
otras aplicaciones \cite{MR3513734}, \cite{MR3533300}, \cite{MR3569059},
\cite{zbMATH07035922}, \cite{mrt2017correlations}.
Discutiremos una de ellas: la prueba de una versi\'on logar\'itmica de la
conjetura de Chowla con $k=2$.
\begin{teorema}$($\cite[Thm.~1.1]{MR3569059}$)$\label{teo:taochowla}
  Sean $a_1, a_2, b_1,b_2$ enteros tales que (a) $a_1,a_2\geq 1$ y
 (b) $(a_1,b_1)$, $(a_2,b_2)$ no son m\'ultiplos el uno del otro.
  Sea $\omega=\omega(x)$ tal que $\omega(x)\to \infty$ cuando
  $x\to \infty$. Entonces, cuando $x\to \infty$,
  \[\sum_{x/\omega(x)<n\leq x} \frac{\lambda(a_1 n+ b_1) \lambda(a_2 n + b_2)}{n}
  = o(\log \omega(x)).\]
\end{teorema}
Tambi\'en este teorema se generaliza a una clase de funciones multiplicativas,
incluyendo algunas para las cuales la conjetura de Chowla ``est\'andar''
(\ref{eq:chowla}) no ser\'ia cierta.

\subsection{Notaci\'on}
Cuando escribimos $f(n)\ll g(n)$, $g(n)\gg f(n)$ o $f(n) = O(g(n))$, queremos
decir la misma cosa, esto es, que hay  $N>0$ y $C>0$ tales que
$|f(n)|\leq C\cdot g(n)$ para todo $n\geq N$. Escribimos $\ll_a$, $\gg_a$,
$O_a$ si $N$ y $C$ dependen de $a$ (digamos). 
Como de costumbre, $f(n) = o(g(n))$ significa que
$|f(n)|/g(n)$ tiende a $0$ cuando $n\to \infty$.

Dado un subconjunto $A\subset X$, 
 $1_A:X\to \mathbb{C}$ es la funci\'on caracter\'istica de $A$:
\[1_A(x) = \begin{cases} 1 &\text{si $x\in A$,}\\ 0 &\text{si no.}
\end{cases}
\]
Denotamos por $|A|$ el n\'umero de elementos de un conjunto finito $A$.

Escribiremos $(a,b)$ por el m\'aximo comun divisor de dos enteros $a,b\ne 0$, siempre y cuando no haya
posibilidad de confusi\'on con el par ordenado $(a,b)$.

La distancia $d(\beta_1,\beta_2)$
entre dos elementos $\beta_1,\beta_2\in \mathbb{R}/\mathbb{Z}$ se define como
$\min_{a\in \mathbb{Z}} |a+\beta_1-\beta_2|$. Por ejemplo, la distancia entre $0.01$ y $0.99$ es $0.02$.
Podemos tambi\'en definir, para $\beta\in \mathbb{R}$, la distancia $d(\beta,\mathbb{Z})$ como
la distancia entre $\beta$ y el entero m\'as cercano. Est\'a claro que $d(\beta,\mathbb{Z})$
depende s\'olo de $\beta \mo \mathbb{Z}$, y que $d(\beta_1,\beta_2) = d(\beta,\mathbb{Z})$
para cualquier $\beta$ tal que $\beta \equiv \beta_1-\beta_2 \mo \mathbb{Z}$.

\subsection{Agradecimientos}

El viaje y estad\'ia de los autores fueron financiados en parte por la fundaci\'on Humboldt. Se deben las gracias a Boris Bukh y Nikos Frantzikinakis por sus valiosos
comentarios.

\section{Breve repaso y resultados preliminares}

\subsection{Funci\'on zeta. F\'ormula de Perron.}\label{subs:zetperr}

Sea $f:\mathbb N\to \mathbb C$. Muchas veces
estamos interesados en sumas finitas de $f$: 
\[
 \sum_{n\leq x} f(n).
 \]

 En general, podemos tratar de obtener resultados sobre una funci\'on $f$
 estudiando una funcion generatriz de $f$, como, por ejemplo,
 $\sum_{n=1}^\infty f(n) z^n$. Si $f$ 
 est\'a asociada a un problema multiplicativo, tiene sentido
 estudiar la siguiente funci\'on generatriz,
llamada \emph{funci\'on zeta} de $f$:
\[
 Z_f(s)=\sum_{n=1}^{\infty} f(n) n^{-s}   \qquad   s\in \mathbb C.
\]
La idea de usar los factores $n^{-s}$ es que son multiplicativos:
$(ab)^{-s}=a^{-s}b^{-s}$. La idea es que, si obtenemos suficiente informaci\'on sobre $Z_f(s)$, podremos deducir informaci\'on sobre $f(n)$.

Si $f$ es multiplicativa, entonces, como lo not\'o ya Euler,
podemos factorizar $Z_f(s)$ como un producto sobre los primos:
\begin{equation}\label{eq:eulerprod}
 Z_f(s)=\prod_{p} (1+f(p)p^{-s}+f(p^2)p^{-2s}+\ldots)    \qquad  \Re s> 1
\end{equation}
Por ejemplo, en el caso $f=1$ tenemos que su funci\'on zeta es
\begin{equation}\label{eq:euler_zeta}
Z_1(s)= \zeta(s)=\sum_{n=1}^{\infty} n^{-s}=\prod_p \frac{1}{1-1/p^s},
\end{equation}
de donde, por ejemplo, obtenemos
que hay infinitos primos, puesto que $\lim_{s\to 1^+} \sum_n n^{-s}$ diverge,
y $\lim_{s\to 1^+} 1/(1-1/p^s)$ puede divergir s\'olo si es un producto infinito.
Escogiendo $s$ de manera m\'as precisa, podemos obtener cotas \'utiles sobre
los primos (ejercicio \ref{ej:euler_suma}).

Veamos c\'omo usar la funci\'on zeta de manera elemental para acotar la suma
de una funci\'on aritm\'etica, $\sum_{n<x} \tau(n^2) n^{-1}$, con $\tau$ la
funci\'on divisor. Para $s>1$, 
\[\begin{aligned}
&\sum_{n\le x}\frac{\tau(n^2)}{n}\le
\sum_n \frac{\tau(n^2)}{n}
\left(\frac xn\right)^{s-1} =x^{s-1} \sum_n \frac{\tau(n^2)}{n^s}=x^{s-1}
\prod_p \left(1+\frac{3}{p^s}+O\left(\frac{1}{p^{2s}}\right)\right)\\
\ll &x^{s-1}\prod_p \left(1+\frac{3}{p^s}\right)\le x^{s-1}\prod_p \left(1+\frac{1}{p^s}\right)^3\le x^{s-1}(\sum_{n} n^{-s})^3\ll \frac{x^{s-1}}{(s-1)^3},
\end{aligned}\]
donde hemos usado el hecho que $\sum_n n^{-s}-\int_1^{\infty} t^{-s} dt\ll 1$.
Tomando $s=1+\frac{1}{\log x}$, obtenemos
\begin{equation}\label{eq:ordendiv}
  \sum_{n\le x}\frac{\tau(n^2)}{n} \ll (\log x)^3.\end{equation}
\'Este es el orden de magnitud correcto:
en verdad, $\sum_{n\leq x} \tau(n^2) n^{-1}$ es asint\'otica a una constante
por $(\log x)^3$.

Si queremos asint\'oticas para sumas necesitamos algo m\'as preciso para relacionar las sumas con la funci\'on zeta. La herramienta que lo permite es la integral de Perron, que es capaz de expresar que un n\'umero $y$ sea mayor o menor que 1 en t\'erminos anal\'iticos, y en particular en t\'erminos de los <<arm\'onicos>> $y^s$, $s$ complejo: para $\sigma>0$,
\begin{equation}\label{eq:perron}
\frac{1}{2\pi i} \int_{\sigma-i\infty}^{\sigma+i\infty} y^{-s} \frac{ds}{s} =
\begin{cases} 1 & \text{si}\; 0<y<1,\\
  1/2 & \text{si}\; y=1,\\
  0 & \text{si}\; y>1.\end{cases}
\end{equation}
Esta f\'ormula
puede probarse usando el teorema de los residuos, o el teorema de inversi\'on de Fourier para la funci\'on  $1_{(0,\infty)}(x) e^{-\sigma x}$ en $x=\log y$, ya que $y^{-it}=e^{-it\log x}$. La integral debe comprenderse como el
l\'imite $\lim_{T\to \infty} \int_{\sigma-i T}^{\sigma+i T} y^{-s} ds/s$.

As\'i,  usando Perron, obtenemos, para $x>0$,
\begin{equation}\label{eq:perron_zeta}
  \sum_{n\leq x} f(n)= 
  \frac{1}{2\pi i} \int_{\sigma-i\infty}^{\sigma+i\infty} Z_f(s) x^s \frac{ds}{s}
  + \begin{cases} \frac 12 f(x) &\text{si}\; x\in \mathbb{Z},\\
    0 &\text{si}\; x\not\in \mathbb{Z}.\end{cases}
\end{equation}
Si bien todos los arm\'onicos $x^{it}$ contribuyen a la integral,  t\'ipicamente los importantes van a ser los que tienen frecuencia $t$ peque\~na. Para demostrarlo, la idea es que es mejor trabajar con
una funci\'on continua, antes que con la funci\'on en el lado derecho
de (\ref{eq:perron}). Podemos, por ejemplo, trabajar con la
siguiente funci\'on, continua en $(0,\infty)$:
\begin{equation}\label{eq:psidelta}
 \psi_{\delta}(x)=1_{(0,1-\delta)}+\frac{1-x}{\delta}1_{[1-\delta,1]}  \qquad   0<\delta<\frac{1}{2}.
 \end{equation}
 Por lo mismo que es igual a $1_{(0,1)}(x)$ cuando $0\leq x<1-\delta$ o $x\geq 1$, est\'a claro que, s\'i $|f(x)|\le 1$ para todo $x$, entonces
\begin{equation}\label{eq:apprsum}
\sum_{n\leq x} f(n)=O(\delta x)+ \sum_{n=1}^{\infty} f(n) \psi_{\delta}\left(\frac{n}{x}\right)
\end{equation}
Usando el teorema de inversi\'on de Fourier (o la integral de Perron), 
no es dif\'icil demostrar (ejercicio \ref{ej:mellin}) que  
\begin{equation}\label{eq:mellin}
  \psi_{\delta}(y)=\frac{1}{2\pi i} \int_{\sigma-i\infty}^{\sigma+i\infty} (M\psi_{\delta})(s) y^{-s} ds,
\end{equation}
donde
\begin{equation}\label{eq:tildepsi}
  (M\psi_{\delta})(s) =
  \frac{1}{s(s+1)} \frac{1-(1-\delta)^{s+1}}{\delta} .
\end{equation}
Utilizamos la notaci\'on $M \psi$ aqu\'i pues se trata de una {\em transformada
  de Mellin} (lo cual no es sino una transformada de Fourier con un
cambio complejo de variable). En general, 
\[M\psi(s) := \int_0^\infty \psi(x) x^{s-1} dx,\]
y, bajo ciertas condiciones de integrabilidad sobre $\psi$ y $M\psi$,
\[\psi(x) = \frac{1}{2\pi i} \int_{\sigma - i\infty}^{\sigma+ i \infty} M\psi(s) x^{-s} ds
\;\;\;\;\;\;\text{(teorema de inversi\'on de Mellin)},\]
lo cual se deduce del teorema de inversi\'on de Fourier.
\begin{lema}\label{le:perron}
Si $|f(n)|\le 1$ para todo $n$, entonces
\[
\sum_{n\leq x} f(n)=O(\delta x\log x)+ \frac{1}{2\pi i}
\int_{1+\frac{1}{\log x}-i\delta^{-2}}^{1+\frac{1}{\log x}+i\delta^{-2}}
x^{s} (M\psi_{\delta})(s) Z_f(s) ds.
\]
\end{lema}
\begin{proof}
  Por (\ref{eq:apprsum}) y (\ref{eq:mellin}), para $\sigma>0$,
  \[
  \sum_{n\leq x} f(n) = O(\delta x)+ \frac{1}{2 \pi i}
  \int_{\sigma-i\infty}^{\sigma+i\infty}
  x^{s} (M \psi_{\delta})(s) Z_f(s) ds.\]
  Vemos que (\ref{eq:tildepsi}) implica que
  $|M(\psi_{\delta})(s)|\leq 2/\delta |s (s+1)|$ cuando $\Re s >0$.
  Como $|f(n)|\leq 1$ para todo $n$, $|Z_f(\sigma + i t)|\leq \zeta(\sigma)
  \leq 1/(\sigma-1) + O(1)$ para $\sigma>1$.  Por lo tanto,
  \begin{equation}\label{eq:cotacola}\begin{aligned}
  \int_{\sigma + i T}^{\sigma +i\infty}
  x^{s} (M \psi_{\delta})(s) Z_f(s) ds &\ll
  \frac{|x^\sigma|}{\sigma-1} \int_{T}^{\infty} \frac{1}{\delta |t (t+1)|} dt
  \ll \frac{|x^\sigma|}{\delta (\sigma-1) T}
  .\end{aligned}\end{equation}
  Tomamos $\sigma = 1 + 1/\log x$ y $T = \delta^{-2}$, y el lado derecho de (\ref{eq:cotacola})  se vuelve $O(\delta x \log x)$. La cota para la integral de
  $\sigma - i\infty$ a $\sigma - i T$ es evidentemente la misma.
\end{proof}
  
Luego efectivamente s\'olo las frecuencias peque\~nas van a ser importantes para controlar el promedio. Veamos c\'omo usar este hecho para ver que hay cancelaci\'on en la suma
\[
 \sum_{n\leq x} \lambda(n)
\]
Para ello, miramos a su funci\'on zeta correspondiente a $\lambda(n)$:
\[
 Z_{\lambda}(s)=\sum_{n}\lambda(n)n^{-s}=\prod_p (1-p^{-s}+p^{-2s}-\ldots)=\prod_p\frac{1}{1+p^{-s}}=\prod_p \frac{1-p^{-s}}{1-p^{-2s}},
\]
luego 
\begin{equation}\label{eq:zetlambquot}
 Z_{\lambda}(s)=\frac{\zeta(2s)}{\zeta(s)}.
\end{equation}
Para acotar el promedio de $f$ queremos extender $Z_f(s)$ a $s$ con parte
real $<1$
con el prop\'osito de usar la integral en \eqref{eq:perron_zeta} para la $\sigma$ m\'as peque\~na que podamos, ya que $|x^s|=x^{\sigma}$.

Para $\Re s>1$,
\begin{equation}\label{eq:extension_zeta}
 \zeta(s)=\frac{1}{s-1}+\sum_{n=1}^{\infty} n^{-s}-\int_{1}^{\infty} u^{-s} \, du=\frac{1}{s-1}+\sum_{n=1}^{\infty} \int_{n}^{n+1} (n^{-s}-u^{-s}) \, du
\end{equation}
y es muy sencillo ver que la \'ultima sumatoria converge en $\Re s>0$.
Por lo tanto, podemos hablar de $\zeta(s)$
como $1/(s-1)$ m\'as una funci\'on anal\'itica en la regi\'on $\Re s> 0$.
As\'i, hemos extendido el
numerador y el denominador en (\ref{eq:zetlambquot})
a $\Re s > 0$; empero, tambi\'en debemos controlar cuando el denominador
se desvanece. Como el denominador es $\zeta(s)$, controlar cuando es cero
en toda la regi\'on $\Re s > 0$ equivaldr\'ia a la hip\'otesis de Riemann.
Vamos a citar el mejor resultado conocido en dicha direcci\'on, debido a Vinogradov y Korobov:
\begin{teorema}\label{te:vinogradov_korobov}$($\cite[Thm.~8.29]{MR2061214}$)$
  Hay una constante $c>0$ tal que 
$\zeta(s)\neq 0$ para $s=\sigma+it$ con $\sigma \ge 1- c(\log t)^{-2/3}(\log\log t)^{-1/3}$, $|t|\ge 3$, y en dicha zona tambi\'en se cumplen las cotas
\[
\frac{1}{\zeta(s)}\ll (\log t)^{2/3}(\log\log t)^{1/3},
\]
\[\frac{\zeta'(s)}{\zeta(s)}\ll (\log t)^{2/3}(\log\log t)^{1/3}.\]
\end{teorema}

Con dicho resultado podemos acotar el promedio de $\lambda$:
\begin{teorema}\label{te:TNP_liouville}
\[
\sum_{n\leq x} \lambda(n)\ll x \exp(-(\log x)^{3/5+o(1)}).
\]
\end{teorema}
\begin{proof}
  Aplicamos el Lema \ref{le:perron} para $f=\lambda$, y definimos $T = \delta^{-2}$.
  Por el teorema de Cauchy, podemos
  desplazar la l\'inea de integraci\'on a los segmentos rectos
  $L_-$, $L_1$, $L_+$,
  donde $L_1$ va de $\sigma- i T$ a $\sigma+i T$,
  $L_-$ va de $1 + 1/\log x - i T$ a
  $\sigma- i T$ y $L_+$ de $\sigma + i T$ a
  $1 + 1/\log x + i T$;
  aqu\'i $\sigma = 
  1-c (\log T)^{-2/3}(\log\log T)^{-1/3}$, y utilizamos
  el Teorema \ref{te:vinogradov_korobov} para asegurarnos
  que $\zeta(2s)/\zeta(s)$ es anal\'itica en la regi\'on entre la vieja
  y la nueva l\'inea de integraci\'on.
  As\'i,
\[\begin{aligned}
\sum_{n\leq x} \lambda(n)
&=
O(\delta x\log x)+
\frac{1}{2\pi i}
\int_{L_-+L_1+L_+} x^s M\psi_{\delta}(s) \frac{\zeta(2s)}{\zeta(s)} \, ds\\
&= O(\delta x\log x)+ O(x \delta^3) + 
\frac{1}{2\pi i}
\int_{\sigma- i T}^{\sigma + i T}
x^s M\psi_{\delta}(s) \frac{\zeta(2s)}{\zeta(s)} \, ds ,\end{aligned}\]
donde usamos
la cota del Teorema \ref{te:vinogradov_korobov}, as\'i como la cota
$M\psi_{\delta}(s) = O(1/\delta |s (s+1)|)$, y el hecho que $\zeta(2 s)$
est\'a acotada por $\zeta(2 \sigma)$; podemos asumir $2\sigma>3/2$ (digamos),
as\'i que $\zeta(2\sigma) = O(1)$.
Como $|x^s| = |x^\sigma| =
x \exp\left(- c (\log x) (\log T)^{-2/3} (\log \log T)^{-1/3}\right)$,
obtenemos, usando las mismas cotas,
\[\begin{aligned}
&\sum_{n\leq x} \lambda(n) = O(\delta x \log x) +
O\left(\delta^{-1} x e^{- c (\log x) (\log T)^{-2/3} (\log \log T)^{-1/3}} \log x\right)\\
&\ll e^{- C (\log x)^{\alpha} (\log \log x)^{\beta}} x \log x \\ &+
e^{C (\log x)^{\alpha} (\log \log x)^{\beta} - c (\log x) (2 C (\log x)^{\alpha} (\log \log x)^\beta)^{-2/3}
  (\log 2 C + \alpha \log \log x + \beta \log \log \log x)^{-1/3}}  x \log x\end{aligned}\]
para $\delta = e^{-C (\log x)^\alpha (\log \log x)^\beta}$ y $T = \delta^{-2}$.
Est\'a claro que lo \'optimo es escoger $\alpha$ y $\beta$ tales que
$\alpha = 1 - 2\alpha/3$ y $\beta = -2 \beta/3 -1/3$, i.e., $\alpha = 3/5$
y $\beta = -1/5$. Asimismo, escogemos $C$ suficientemente peque\~no
para que $C < c (2 C)^{-2/3} \cdot (2 \alpha)^{-1/3}$, digamos. As\'i obtenemos
\[\sum_{n\leq x} \lambda(n) \ll
e^{- C (\log x)^{3/5} (\log \log x)^{-1/5}} x \log x \ll
e^{- (C/2) (\log x)^{3/5} (\log \log x)^{-1/5}} x\]
para $C$ m\'as grande que una constante.
\end{proof}

\begin{corolario}\label{cor:sumlambinv}
 Sean $x\geq 1$, $t\leq \exp^{(\log x)^{3/5-\epsilon}}$, $\epsilon>0$. Entonces
  \[\sum_{x<n\leq 2x} \frac{\lambda(n)}{n^{1+it}} \ll \exp(-(\log x)^{3/5+o_\epsilon(1)}).\]
\end{corolario}
\begin{proof}
  Por el Teorema \ref{te:TNP_liouville} y sumaci\'on por partes (ejercicio \ref{ej:sumapart}), obtenemos
  en verdad que
  \[\sum_{x<n\leq 2x} \frac{\lambda(n)}{n^{1+it}} \ll (1+|t|) \exp(-(\log x)^{3/5+o(1)}) \log x.\]
\end{proof}

Es de suponer que la gran mayor\'ia de lectores ya han visto por lo menos
una prueba del teorema de los n\'umeros primos basada sobre la integraci\'on
compleja y las propiedades de $\zeta(s)$. De todas maneras, demos los
primeros pasos de una prueba siguiendo las l\'ineas generales que acabamos de
sentar.

Por \eqref{eq:euler_zeta}, la funci\'on
$Z(s)=\log\frac{1}{\zeta(s)}$ es de la forma $Z_{1_P} + O(1)$ para
$\Re s > 1$, donde $Z_{1_P} = \sum_p p^{-s}$ es la serie que corresponde a la
funci\'on indicatriz $1_P$ del conjunto de todos los
primos. Un problema con usar esta funci\'on $Z(s)$ es que $1/\zeta(1)=0$, lo
cual significa que no podemos definir $Z(s)$ como meromorfa alrededor de $s=1$. Una manera de solucionar el problema ser\'ia integrar por partes
en la integral que va de $c-i\infty$ a $c+i\infty$; luego aparecer\'ia,
en vez de $Z(s)$, su derivada
\begin{equation}\label{eq:Lambda_zeta}
 Z'(s)=-\frac{\zeta'(s)}{\zeta(s)}=Z_{\Lambda},
\end{equation}
donde $\Lambda(n)=\log p$ para $n=p^{\alpha}$, $p$ alg\'un primo,
y $\Lambda(n)=0$ para otros $n$. Debido a \eqref{eq:extension_zeta},
$Z(s)$ es meromorfa en $\Re s>0$ con un polo en $s=1$ de resto 1.

En vez de seguir el procedimiento descrito en el p\'arrafo anterior,
es m\'as f\'acil usar directamente $Z_{\Lambda}$ para demostrar
el teorema de los n\'umeros primos en la siguiente forma.

\begin{teorema}\label{te:TNP} Para $x>0$,
\[
\sum_{p\le x} \log p = x+O(x\exp(-(\log x)^{3/5+o(1)}))
\].
\end{teorema}
\begin{proof}
  Ejercicio \ref{ej:TNP}.
\end{proof}

\begin{corolario}\label{cor:corTNP}
  Para $x>0$,
  \begin{equation}\label{eq:vallmas}
    \pi(x) =  \Li(x) + O(x \exp(-(\log x)^{3/5+o(1)})),\end{equation}
  \begin{equation}\label{eq:sumloginvp}
    \sum_{p\leq x} \frac{\log p}{p} = \log x + \kappa_1 + O(\exp(-(\log x)^{3/5+o(1)})),
  \end{equation}
  \begin{equation}\label{eq:suminvp}
    \sum_{p\leq x} 1/p = \log \log x + \kappa_2 + O(\exp(-(\log x)^{3/5+o(1)})),
  \end{equation}
  donde $\kappa_1$ y $\kappa_2$ son constantes.
\end{corolario}
Por cierto, las f\'ormulas menos precisas $\sum_{p\leq x} (\log p)/p = \log x + O(1)$ y $\sum_{p\leq x} 1/p = \log \log x + \kappa_2 + O(1/\log x)$ ya eran conocidas
antes del teorema de los n\'umeros primos (Chebyshev-Mertens, a\~{n}os
1848--1874; ver \cite[\S 2.2]{MR2061214}).
\begin{proof}[Esbozo de prueba del corolario]
  Por sumaci\'on por partes. En el caso de (\ref{eq:sumloginvp}), para evitar
  problemas debidos al hecho que $\exp(-(\log x')^{3/5+o(1)})$ puede ser bastante
  m\'as grande que 
  $\exp(-(\log x)^{3/5+o(1)})$ para $x'$ mucho m\'as peque\~no que $x$,
  conviene estimar las sumas
  \begin{equation}\label{eq:harekrish}
    \sum_{n>x} \left(1_P(n)\cdot \frac{\log n}{n} - \frac{1}{n}\right),
  \;\;\;\;\;\;\;
    \sum_{n\leq x} \frac{1}{n},\end{equation}
    la primera de ellas por sumaci\'on por partes, y la segunda por
    la f\'ormula $\sum_{n\leq x} 1/n =\log x+ \gamma + O(1/x)$, donde
    $\gamma$ es una constante ({\em constante de Euler}). \'Esta \'ultima
    f\'ormula se establece f\'acilmente: las \'areas que quedan
    entre la hip\'erbola $y=1/x$ y las l\'{\i}neas horizontales $y=1/n$ para
    $n\leq x\leq n+1$ pueden ponerse en pila, y entonces est\'a claro
    que su suma para $n\geq 1$ es una constante, y su suma para $n>x$ es
    $O(1/x)$. N\'otese por \'ultimo que la primera suma
    en (\ref{eq:harekrish}), tomada sobre todos los $n\geq 1$, converge
    a una constante, gracias a la estimac\i\'on de esa misma suma por
    sumaci\'on por partes para $x$
    general. Estas observaciones son suficientes para construir una prueba
    (ejercicio).

    Procedemos de la misma manera para establecer \ref{eq:suminvp}: estimamos
    las sumas
  \[\sum_{n>x} \left(\frac{1_P(n)}{n} - \frac{1}{n \log n}\right),\;\;\;\;\;\;\;
  \sum_{n\leq x} \frac{1}{n \log n},\]
  la primera de ellas por sumaci\'on por partes, y la segunda por una
  comparaci\'on con $\int_2^x 1/n \log n$.
\end{proof}

De manera an\'aloga, podemos
calcular la probabilidad de que un n\'umero no tenga factores primos
en un rango dado.

\begin{lema}\label{le:casiprimos}
Sea $1\le Q^{\alpha}\le Q\le x^{\frac{1}{(\log\log x)^3}}$. Entonces
\[
\sum_{n<x, p\mid n \Rightarrow p\not\in [Q^{\alpha},Q]} 1= \alpha x+
x \cdot O\left(\exp\left(- \min\left((\alpha \log Q)^{3/5+o(1)},
\frac{\log x}{3 \log Q}\right)\right)\right).\]
\end{lema}
Para $Q$ cercano a $x$, el problema comenzar\'ia a cambiar de cariz
({\em funci\'on de Dickman}, {\em funci\'on de Buchstab}; ver, por ejemplo,
\cite[\S 7.1--7.2]{MR2378655}).
\begin{proof}
  Observemos que la suma del enunciado es $\sum_{n<x} g(n)$, donde $g$ es
  la funci\'on totalmente multiplicativa con
  $Z_g(s) = \prod_{p\not\in [Q^{\alpha},Q]} (1-p^{-s})^{-1}
  = \zeta(s) \prod_{Q^{\alpha}\leq p\leq Q}(1-p^{-s})$.
  La funci\'on
  $Z_g(s)$ tiene un polo en $s=1$ con residuo
  $\prod_{Q^{\alpha}\leq p\leq Q} (1-p^{-1})$.
  Gracias a (\ref{eq:suminvp}),
  vemos que \[\begin{aligned}\prod_{p\leq z} (1-p^{-1}) &= \exp\left(
  -\sum_{p\leq z} p^{-1} + c + O\left(z^{-1}\right)\right)
  = e^{-\log \log z + c' + O(\exp(-(\log z)^{3/5+o(1)}))} \\ &= 
  \frac{C}{\log z} (1 + O(\exp(-(\log z)^{3/5+o(1)})),\end{aligned}\]
  donde $c$, $c'$ y $C$ son constantes.
  En consecuencia, \begin{equation}\label{eq:astrico}\prod_{Q^{\alpha}\leq p\leq Q} (1-p^{-1}) =
  \alpha \cdot (1 + O(\exp(-(\alpha \log Q)^{3/5+o(1)}))) .\end{equation}

  Por otra parte,
\[
|Z_g(\sigma+it)|\leq \left|\zeta(\sigma+it)\right| \prod_{p\leq Q}(1+p^{-\sigma}).
\]
Para $1 - 1/\log Q\leq \sigma\leq 1$,
\[\begin{aligned}
\prod_{p\leq Q}(1+p^{-\sigma}) &\leq \exp\left(\sum_{p\leq Q} p^{-\sigma}\right)
\leq \exp\left(Q^{1-\sigma} \sum_{p\leq Q} p^{-1}\right) \\ &\leq
\exp(e (\log \log Q + \kappa + o(1))) \ll (\log Q)^3,\end{aligned}\]
por el corolario \ref{cor:corTNP}.
Usando (\ref{eq:extension_zeta}), vemos que, para $s = \sigma + i t$
con $\sigma \geq 1 - 1/\log Q$,
\[\begin{aligned}
\zeta(s) &= \frac{1}{s-1} + O\left(\int_1^Q \left(\lfloor u^{-s}\rfloor - u^{-s}
\right) du + \int_Q^\infty \left(\lfloor u^{-s}\rfloor - u^{-s}\right) du\right)\\
&= \frac{1}{s-1} + O\left(\int_1^Q \frac{du}{u} + 
|s| \int_Q^\infty u^{-\sigma-1} du\right) = \frac{1}{s-1} + O\left(\log Q +
|t|/Q^{\sigma}\right).\end{aligned}\]
As\'i, para $|t|\leq Q$, conclu\'imos que $|\zeta(s)| = 1/(s-1) + O(\log Q)$.
En particular, $|Z_g(\sigma+it)|\ll (\log Q)^4$ para $1\leq |t|\leq Q$.

Procedamos como en la demostraci\'on del Teorema \ref{te:TNP_liouville},
s\'olo que con
\[\sigma = 1 - \min\left(c (\log T)^{-2/3} (\log \log T)^{-1/3},1/\log Q\right).\]
Obtenemos que
\[
\sum_{n\leq x} g(n) =
\alpha x \cdot (1 + O(\exp(-(\alpha \log Q)^{3/5+o(1)}))) + S,\]
donde \[\begin{aligned} S &=
O(\delta x \log x) + O(x \delta^3 (\log Q)^4) \\
&+
O\left(\delta^{-1} x e^{- (\log x)
  \min(c (\log T)^{-2/3} (\log \log T)^{-1/3},1/\log Q)}
\log x\right).\end{aligned}\]
Si $\log Q \leq c (\log T)^{2/3} (\log \log T)^{1/3}$, procedemos exactamente
como lo hicimos anteriormente, y obtenemos $S = x O(\exp(-(\log x)^{3/5+o(1)}))$.
Si
$\log Q > c (\log T)^{2/3} (\log \log T)^{1/3}$, tomamos
$\delta = x^{-1/2 \log Q}$, $T = \delta^{-2} = x^{1/\log Q}$, y as\'i
\[\begin{aligned}
S &= O\left(\delta x \log x\right)
+O(\delta^3 x (\log x)^4) \ll \frac{x \log x}{x^{\frac{1}{2 \log Q}}}
\ll x^{1-\frac{1}{3 \log Q}}\end{aligned}\]
para $Q \leq x^{1/(\log \log x)^3}$.
\end{proof}

\subsubsection{Ejercicios}\label{sec:lobasico}
\begin{enumerate}
 \item Una herramienta \'util para nosotros va a ser sustituir sumas por integrales. Vamos a demostrar la \emph{Regla del rect\'angulo}.
\begin{enumerate}
 \item Demuestre que $f(n)=\int_{n}^{n+1} f(n)\,dt =\int_{n}^{n+1} f(t)\, dt  +O(\int_n^{n+1} |f'(t)|\, dt)$.
 \item Usando el apartado anterior, pruebe que, para enteros $a<b$ cualesquiera,
\begin{equation}\label{eq:regla_rectangulo}
  \sum_{a< n\leq b} f(n)=\int_a^b f(t)\, dt + O\left(\int_a^b |f'(t)|\, dt\right).\;\;\;\;\;\;
  \text{(regla del rect\'angulo)}
\end{equation}
Por cierto, existen
tambi\'en expresiones similares con t\'erminos de error que involucran
$f''$, $f'''$, etc. ({\em f\'ormula de Euler-Maclaurin}).
\item Observe que si $f$ es mon\'otona, entonces  la f\'ormula del apartado anterior da algo similar al \emph{criterio integral} para series:
\[
\sum_{a< n\leq b} f(n)=\int_a^b f(t)\, dt +O(|f(b)-f(a)|)
\]
\item Use
  las f\'ormulas anteriores para demostrar las siguientes aproximaciones:
\[
 \sum_{n\le x} n^3 = \frac{x^4}{4}+O(x^3),  \qquad  \sum_{n\le x} \sen\left(\frac{n^2}{x^{3/2}}\right)=cx^{3/4}+O(\sqrt x),
\]
con $c>0$ la constante $c= \int_0^{\infty} \frac{\sen u}{2\sqrt{u}}\, du$.

\end{enumerate}

\item \label{ej:sumapart}
  \begin{enumerate}
  \item
    Muestre que, para cualquier entero $N$ y $a_1,\dotsc,a_N\in \mathbb{C}$,
    $b_1,\dotsc,b_N\in \mathbb{C}$ arbitrarios,
    \begin{equation}\label{eq:sumpart}\begin{aligned}\sum_{1\leq n\leq N} a_n b_n &= \sum_{1\leq n\leq N} (A(n)-A(n-1)) \cdot b_n\\
    &= A(N) \cdot b_N - \sum_{1\leq n\leq N-1} A(n) \cdot (b_{n+1}-b_n),\end{aligned}
    \end{equation}
    donde  $A(x) = \sum_{1\leq n\leq x} a_n$. (Se trata simplemente de recomponer los
    t\'erminos de una suma.) A esta t\'ecnica se la denomina {\em sumaci\'on por partes}. Es utilizada cuando sabemos como estimar las sumas de tipo
    $A(x)$, y queremos estimar una suma $\sum_{n\leq N} a_n b_n$.
  \item Poner la igualdad de la forma siguiente
    hace m\'as claro el paralelismo con
    la integraci\'on por partes (la t\'ecnica b\'asica aprendida en un primer
    curso de c\'alculo integral). Reescriba (\ref{eq:sumpart}) como sigue:
    para $N$,
    $A_0, A_1,\dotsc,A_N\in \mathbb{C}$ y
    $b_1,\dotsc,b_N\in \mathbb{C}$ arbitrarios,
    \begin{equation}\label{eq:sumpart2}
      \sum_{1\leq n\leq N} (A_n - A_{n-1}) b_n
      = A_N b_N - A_0 b_1 - \sum_{1\leq n\leq N-1} A_n \cdot (b_{n+1}-b_n).\end{equation}
    Alternativamente, muestre que (\ref{eq:sumpart2}) es un caso especial de
    la integraci\'on por partes, formulada para integrales de Lebesgue.
  \end{enumerate}

\item
  \begin{enumerate}
    \item\label{ej:promtau} Sea $\tau(n)$ la funci\'on que cuenta el n\'umero de divisores de $n$, es decir $\tau(n)=\sum_{d\mid n } 1$. Demuestre, cambiando el orden de sumaci\'on,
  que
  \[
\sum_{n\le x}\tau(n)= \sum_{d\le x} \left\lfloor\frac xd\right\rfloor
=x\sum_{d\le x} \frac 1d +O(x)=x\log x+O(x),
\]
donde $\lfloor x\rfloor$ es la parte entera de $x$. (Evidentemente,
$\lfloor x \rfloor=x+O(1)$.)

\item {\em (Comentario)}
En otras palabras, la esperanza del n\'umero de divisores de
un entero aleatorio $n\leq x$ es $\log x + O(1)$. Empero, la mayor parte
de enteros tienen un n\'umero de divisores bastante menor. Como veremos
despu\'es (ejercicio \ref{ej:euler_suma}), la esperanza del n\'umero
de divisores primos de un $n\leq x$ aleatorio es $(1+o(1)) \log \log x$;
Verifique que un n\'umero con $(1+o(1)) \log \log x$ divisores primos
y sin divisores cuadrados aparte de $1$ tiene
$(\log x)^{(1+o(1)) \log 2}$ divisores.

Lo que sucede es que, si bien los n\'umeros con $\geq C \log \log x$ divisores
primos ($C>1$) resultan estar en la minor\'ia, tienen un n\'umero tan grande
de divisores (?`cu\'antos?) que hacen que el promedio (o esperanza)
sea bastante superior a la mediana (el valor tal que mitad de los casos
son superiores y mitad son inferiores a \'el).

\item Use sumaci\'on por partes e integraci\'on por partes para demostrar las siguientes aproximaciones:
\[
 \sum_{n\le x} n\tau(n)=\frac{x^2}{2}\log x +O(x^2),  \qquad  \sum_{n\le x} \frac{\tau(n)}{n}=\frac{(\log x)^2}{2}+O(\log x).
\]
\end{enumerate}
  
\item \label{ej:euler_suma} En este problema vamos a demostrar la asint\'otica $\sum_{p\le x}\frac{1}{p}= (1+o(1)) \log \log x$ a partir del producto de Euler (\ref{eq:euler_zeta}). Antes de ello, observa que de dicha asint\'otica es posible deducir que $\sum_{n\le x} w(n)=x(1+o(1))\log\log x$, con $w(n)=\sum_{p\mid n} 1$.
\begin{enumerate}
 \item Use el criterio integral para series para demostrar que $\zeta(s)=\sum_{n=1}^{\infty} n^{-s}=\frac{1}{s-1}+O(1)$ para $s>1$.
\item Tomando logaritmos en la identidad de Euler (\ref{eq:euler_zeta}), usando el apartado anterior y la aproximaci\'on de Taylor $\log\frac{1}{1-x}=x+O(x^2)$ para $|x|<1/2$, demuestre que 
\[
\sum_{p} p^{-s}=  (1+o(1)) \log\frac{1}{s-1}
\qquad \text{ para } s>1.
\]
En este contexto particular, $o(1)$ quiere decir ``una cantidad una cantidad que
tiende a $0$ cuando $s\mapsto 1$''.
\item Tome $s=1+\frac{\log\log x}{\log x}$. Usando $\sum_{p>x} p^{-s}\le \sum_{n>x} n^{-s}$ y el criterio integral, demuestre que $\sum_{p>x} p^{-s}\ll 1$. 

\item A partir de los dos apartados anteriores, demuestre que 
\[
 \sum_{p\le x} p^{-1}\ge \sum_{p\le x} p^{-1-\frac{\log\log x}{\log x}} = \sum_{p} p^{-1-\frac{\log\log x}{\log x}} +O(1) =  (1+o(1)) \log\log x.
\]
\item Para $p\le x$, demuestre que $p^{-1-\frac{1}{\log\log x \log x}}=p^{-1}(1+o(1))$. A partir de ah\'i, pruebe las desigualdades
\[
\sum_{p\le x} p^{-1} \le (1+o(1))\sum_p p^{-1-\frac{1}{\log\log x\log x}} =
 (1+o(1)) \log\log x.
\]

\end{enumerate}

\item \label{ej:mellin} 
\begin{enumerate}
\item Para $y>0$, la ecuaci\'on de Perron (\ref{eq:perron}) nos dice que $1_{(0,1)}(y)=\frac{1}{2\pi i} \int_{c-i\infty}^{c+i\infty} y^{-s} \frac{ds}{s}$, asumiendo que definimos $1_{(0,1)}(1)=1/2$. Usando dicha f\'ormula, y teniendo en cuenta que $1_{(0,b)}(y)=1_{(0,1)}(y/b)$, demuestre que $1_{(0,b)}(y)=\frac{1}{2\pi i} \int_{c-i\infty}^{c+i\infty} b^s y^{-s} \frac{ds}{s}$.
 \item Usando la f\'ormula del apartado anterior, demuestre la expresi\'on \eqref{eq:mellin}.

 \item Demuestre que si $\psi_{\delta}$ es la funci\'on definida en (\ref{eq:psidelta}) entonces se cumple que $\int_0^{\infty} \psi_{\delta}(u) u^{s-1}\, du$ es igual a la funci\'on $M\psi_{\delta}(s)$ definida en (\ref{eq:tildepsi}). As\'i, el Teorema de inversi\'on de Mellin tambi\'en demuestra la f\'ormula \eqref{eq:mellin}).

 \item Recuerde que $|M\psi_\delta(s)|\leq 2/\delta |s (s+1)|$. Pruebe
   que $|M\psi_\delta(s)|\leq 2/s$ para $\delta\leq 1/2 (s+1)$. Concluya
   que $|M\psi_\delta(s)|\leq 4/s$ para todo $\delta>0$.
\end{enumerate}

\item Decimos que un entero $n$ es {\em libre de factores cuadrados}
  si no es divisible por ningun cuadrado $d^2$ con $d$ un entero mayor que
  $1$.
  Vamos a demostrar que la proporci\'on de n\'umeros sin divisores cuadrados
  es $1/\zeta(2)=0.6079\ldots$ Para ser precisos: sea
  $Q(x)$ el n\'umero de enteros $n\leq x$ libres
  de factores cuadrados; 
mostraremos que
  $\lim_{x\to\infty} Q(x)/x =\frac{1}{\zeta(2)}$.

Como veremos en el ejercicio \ref{eq:librecuad}
de la secci\'on siguiente, es posible dar una prueba elemental de
este mismo enunciado (con un mejor t\'ermino de error). Ahora mostraremos
como probarlo usando la funci\'on $\zeta(s)$, simplemente para practicar las
t\'ecnicas que hemos expuesto en esta secci\'on.

\begin{enumerate}
\item Demuestre que $Z_{\mu^2}(s)=\frac{\zeta(s)}{\zeta(2s)}$.
  Aplique el Lema \ref{le:perron} para expresar $Q(x)$ en t\'erminos de una integral sobre el segmento $V_0$ descrito por $s=1+\frac{1}{\log x}+it$, con $-\delta^{-c}\le t \le \delta^{-c}$, donde
  $\delta\in (0,1)$ y $c>0$ son par\'ametros que luego elegiremos.
\item  Usando la f\'ormula (\ref{eq:extension_zeta}),
  demuestre que
\[
\zeta(s)-\frac{1}{s-1}\ll \sum_{n=1}^{\infty} n^{-\sigma}
\min \left(1,\frac{1+|t|}{n}\right)\ll (1+|t|)^{1-\sigma}
\]
para $s=\sigma+i t$, $\sigma\geq 1/2$. (Es posible probar cotas m\'as
fuertes, comenzando por la {\em cota de convexidad}
$\zeta(s) \ll_{\sigma,\epsilon} (1+|t|)^{(1-\sigma)/2+\epsilon}$; no las necesitaremos.)

\item Deduzca del apartado anterior y del producto de Euler (\ref{eq:euler_zeta}) para $\zeta(2s)$ que si estamos en la zona $\sigma_0\le \sigma<3$
  para alg\'un $\sigma_0>1/2$
  y a distancia $\gg 1$ de $s=1$, entonces
  $Z_{\mu^2}(s)\ll_{\sigma_0} (1+|t|)^{1-\sigma}$.

\item Usando el Teorema de los residuos y el apartado $a)$, demuestre que
\[
 Q(x) = \frac{x}{\zeta(2)}+ O(\delta x\log x)+ \frac{1}{2\pi i}
\int_{H_{+}\cup V \cup H_{-}}
x^{s} M\psi_{\delta}(s) Z_{\mu^2}(s) ds.
\]
con $V$ el segmento $\sigma=\sigma_0$
($\sigma_0>1/2$), $-\delta^{-c}\le t\le \delta^{-c}$ y
$H_+$, $H_{-}$ los segmentos horizontales que unen los extremos de $V$
con los de $V_0$.

\item Use las cotas que tenemos para $Z_{\mu^2}$ y $M\psi_{\delta}$ para demostrar que
\[
Q(x)=\frac{x}{\zeta(2)}+ O(\delta x\log x) + O\left(x \delta^{(1+\sigma) c - 1 }\right)
+ O\left(x^{\sigma}\delta^{\sigma - 1}\right).
\]

\item Eliga $\delta>0$, $\sigma>1/2$ y $c$ adecuadamente para demostrar que 
  $Q(x) =\frac{x}{\zeta(2)}+O_\epsilon(x^{2/3+\epsilon})$
  para $\epsilon>0$ arbitrario.
\end{enumerate}

El m\'etodo elemental que veremos en la siguiente seccion da
un mejor exponente: $1/2$, en vez de $2/3$. Es posible obtener algunas
mejoras ulteriores combinando el m\'etodo elemental y el m\'etodo anal\'itico.

\item\label{ej:TNP}
  Siga los pasos de la prueba del Teorema \ref{te:TNP_liouville}, usando la f\'ormula (\ref{eq:Lambda_zeta}) y el Teorema de Vinogradov-Korobov
  (Teo.~\ref{te:vinogradov_korobov}), para demostrar el teorema de los n\'umeros
  primos en la siguiente forma:
  \[\sum_{n\leq x} \Lambda(n)  = x+O(x\exp(-(\log x)^{3/5+o(1)})).\]
  Deduzca el Teorema \ref{te:TNP}.
\end{enumerate}

\begin{definicion}
  Mencionamos la transformada de Fourier, as\'i que debemos precisar las
  constantes en nuestra definici\'on. La {\em transformada de Fourier}
  $\widehat{f}:\mathbb{R}\to \mathbb{C}$ de
  una funci\'on $f:\mathbb{R}\to \mathbb{C}$ se define por
  \[\widehat{f}(t) = \int_{-\infty}^\infty f(x) e(- x t) dx\]
  donde $e(\alpha) = e^{2\pi i \alpha}$. Requerimos siempre que $f$ est\'e
  en $L^1$, es decir, $\int |f(x)| dx <\infty$ (``$f$ es integrable'').
\end{definicion}

\subsection{Las cribas y sus limitaciones: el problema de paridad}\label{subs:paridadcribas}

Hagamos una breve pausa, no para ver m\'as resultados que necesitamos,
sino para comprender mejor por qu\'e los problemas con los que lidiaremos
son dif\'iciles, y no pueden ser resueltos s\'olo mediante cribas.

  Digamos que tenemos un conjunto finito $A$ y un conjunto finito de propiedades $P$. Para cada subconjunto de
  $P$, se nos da una f\'ormula aproximada -- con alg\'un t\'ermino de error --
  para el n\'umero de
  elementos de $A$ que satisfacen todas las propiedades en el subconjunto.
  Se nos pide dar una cota, superior o inferior, para el
  n\'umero de elementos de $A$ que no satisfacen ninguna de las propiedades.

  La estrategia evidente ser\'ia aplicar el principio de inclusi\'on-exclusi\'on (ejercicio \ref{ej:incl_excl}). Empero, el n\'umero
  de t\'erminos en una inclusi\'on-exclusi\'on completa es exponencial en el n\'umero de propiedades,
  por lo cual el error total podr\'ia ser enorme. Nos podemos preguntar si existe alguna manera
  de usar un n\'umero mucho m\'as peque\~no de t\'erminos para obtener cotas superiores o inferiores, o
  a\'un expresiones asint\'oticas. Una {\em criba} da una respuesta afirmativa a esta pregunta en contextos
  comunes en la teor\'ia de n\'umeros. (En particular, $A$ ser\'a un conjunto de enteros, y las
  propiedades a considerar estar\'an dadas con una biyecci\'on natural con los primos en un conjunto finito.)

  Una amplia clase de cribas ({\em cribas peque\~nas}) se pueden poner en el formalismo siguiente.
  Para permitir multiplicidades, en vez de trabajar con un conjunto $A$, trabajaremos con 
  reales no negativos $a_n$ para $1\leq n\leq x$. Las propiedades que consideraremos ser\'an la
  divisibilidad por un conjunto $\mathscr{P}$ de primos $p\leq z$, donde $z$ es un par\'ametro.
  Para cada $d\leq D$ (donde $D\geq z$) libre de factores cuadrados y con factores primos s\'olo en $\mathscr{P}$, se nos da que
  \begin{equation}\label{eq:doria}
    \mathop{\sum_{1\leq n\leq x}}_{d|n} a_n = g(d) X + r_d,\end{equation}
    donde $g:\mathbb{Z}^+\to \mathbb{R}$ es una funci\'on multiplicativa,
  $r_d$ es el t\'ermino de error y $X$ depende s\'olo de $x$. Aqu\'i $0\leq g(p)<1$ para
    $p\in \mathscr{P}$. Los t\'erminos de error $r_d$ son
    tal que $\sum_{d\leq D} |r_d|$ sea peque\~no comparado con $X$
(digamos, $O(X/(\log X)^A)$). 

Hay diversas cribas para diversos problemas. Una criba muy simple nos permite
calcular $\sum_{n\leq x: \forall d>1, d^2\nmid n} a_n$. Asimismo, es posible usar
una criba para
dar cotas superiores e inferiores para $\sum_{n\leq x: \text{$n$ tiene $\leq 3$
    factores primos}} a_n$, digamos, o para dar una cota superior para
$\sum_{p\leq x} a_p$. ?`Podemos, empero, dar una
cota inferior no trivial para $\sum_{p\leq x} a_p$, o una estimaci\'on no
trivial de $\sum_{n\leq x} \lambda(n) a_n$, puramente a trav\'es de una criba?

La respuesta es no, como lo muestra el siguiente argumento de Selberg.
Por una parte, la secuencia $a_n=1/2$ satisface (\ref{eq:doria}) con
$g(d) = 1/d$, $|r_d|\leq 1$. Por otra parte, consideremos la secuencia
$a_n' = (\lambda(n)+1)/2$. Vemos que, para $d\leq D = x^{1-\epsilon}$,
\[\begin{aligned}
\mathop{\sum_{n\leq x}}_{d|n} a_n' = \frac{1}{2} \sum_{n\leq x/d} 1 +
\frac{\lambda(d)}{2} \sum_{n\leq x/d} \lambda(n)
= \frac{x}{2 d} + O\left(\frac{x/d}{(\log x)^{A+1}}\right)\end{aligned}\]
gracias a (\ref{eq:hadavall2}). Por lo tanto, (\ref{eq:doria}) se cumple
para $d\leq D$, nuevamente con $g(d)=1/d$ y $\sum_{d\leq D} |r_d| = O(x/(\log x)^A)$. En otras palabras, una criba con el dato (\ref{eq:doria}) para
$d\leq D = x^{1-\epsilon}$ no puede distinguir entre $a_n$ y $a_n'$. Ahora bien,
$\sum_{p\leq x} a_p' = 0$ y $\sum_{n\leq x} \lambda(n) a_n' = \sum a_n' \sim x/2$,
sumas que difieren dr\'asticamente de $\sum_p a_p = \sum_{p\leq x} 1/2 \sim x/2\log x$ y
$\sum_{n\leq x} \lambda(n) a_n = \sum_{n\leq x} \lambda(n) = o(x)$.

Por lo tanto, una criba, a\'un con datos v\'alidos para todo $d\leq x^{1-\epsilon}$, no puede, {\em por si s\'ola}, darnos una cota inferior para $\sum_p a_p$,
o una estimaci\'on no trivial para $\sum_n \lambda(n) a_n$. A este hecho se le
llama ``problema de la paridad''; fue formulado en la forma que acabamos de
ver por Selberg.

La raz\'on del nombre se vuelve m\'as clara si consideramos tamb\'en
la secuencia $a''_n = (1-\lambda(n))/2$. Exactamente por el mismo
argumento que acabamos de ver, esta secuencia satisface
(\ref{eq:doria}) con los mismos valores de $g(d)$ para $d\leq x^{1-\epsilon}$
que las secuencias $\{a_n\}$ y $\{a_n'\}$. Por lo tanto, una criba, por si
sola, no puede distinguir entre $a_n'$ y $a_n''$. Ahora bien, $a_n'=1$ cuando
$n$ tiene un n\'umero par de factores primos, y $a_n'=0$ de lo contrario,
mientras que $a_n''=1$ cuando $n$ tiene un n\'umero impar de factores primos,
y $a_n''=0$ de lo contrario.

No utilizaremos m\'etodos de cribas en el trabajo presente. Empero, se trata
de herramientas \'utiles y completamente usuales en la teor\'ia de n\'umeros,
por lo cual un cierto conocimiento de ellas es de suma importancia.

\subsubsection{Ejercicios}
\begin{enumerate}
\item\label{ej:incl_excl}
  \begin{enumerate}
  \item Sea $A$ un conjunto finito. Entonces el n\'umero de elementos de $A$
    que no son ni rojos ni grandes es igual a
    \[|A| - |\{a\in A: \text{$a$ es rojo}\}| -
    |\{a\in A: \text{$a$ es grande}\}| +
    |\{a\in A: \text{$a$ es rojo y grande}\}|.\]
    \'Este es un caso especial del {\em principio de inclusi\'on-exclusi\'on}.
  \item He aqu\'i un enunciado general del principio de inclusi\'on-exclusi\'on.
    Sea $A$ un conjunto finito y $\mathscr{P}$ un conjunto finito
    de propiedades que los elementos de $A$ pueden o no tener.
    Para $P\subset \mathscr{P}$, denotemos por $A_P$ el conjunto de
    elementos de $A$ que satisfacen todas las propiedades en $P$, y por
    $A_{\setminus \mathscr{P}}$ el conjunto de elementos de $A$ que no satisfacen ninguna
    de las propiedades en $\mathscr{P}$. Muestre que
    \begin{equation}\label{eq:incl_excl}
      \left|A_{\setminus \mathscr{P}}\right| = \sum_{i=0}^{|\mathscr{P}|} (-1)^i
      \sum_{P\subset \mathscr{P}: |P|=i} \left|A_{P}\right|.
    \end{equation}
  \end{enumerate}

\item\label{eq:librecuad}
  Veamos como a\'un el principio de inclusi\'on-exclusi\'on, sin m\'as
  elaboraci\'on, puede dar una soluci\'on a un problema de criba
  bastante particular.
  Recordemos que $Q(x)$ denota el n\'umero de enteros $n\leq x$ libres
  de factores cuadrados.

  \begin{enumerate}
  \item Usando (\ref{eq:incl_excl}), muestre que, para
    $\mathbf{P}$ el conjunto de primos $p\leq \sqrt{x}$,
    \[\begin{aligned}
    Q(x) &=
    \sum_{i=0}^{|\mathbf{P}|} (-1)^i
    \sum_{P\subset \mathbf{P}: |P|=i}
    \left|\{n\leq x: p^2 | n \;\; \forall p\in P\}\right|\\
    &= \sum_{d|\prod_{p\in \mathbf{P}} p} \mu(d)
  |\{n\leq x: d^2|n\}|
    = \sum_{d\leq \sqrt{x}} \mu(d) \left|\{n\leq x: d^2|n\}\right|.\end{aligned}\]
    Sugerencia: utilice dos veces
    el hecho que $d^2|n$ implica $d\leq \sqrt{x}$.
  \item Tomando en cuenta que, para todo $\ell \in \mathbb{Z}^+$,
    \[|\{n\leq x: \ell|n\}| = \frac{x}{\ell} + O(1),\]
    muestre que
    \[Q(x) = x\cdot \sum_{d\leq \sqrt{x}} \frac{\mu(d)}{d^2} +
    O(\sqrt{x}).\]
    Concluya que
    \[Q(x) = x\cdot \sum_{d=1}^\infty \frac{\mu(d)}{d^2} +
    O(\sqrt{x}) = \frac{x}{\zeta(2)} + O(\sqrt{x}).\]
  \end{enumerate}
  
\item
    Demuestre las dos siguientes desigualdades, las cuales podr\'iamos llamar
    versiones ``incompletas'' de (\ref{eq:incl_excl}):
    para $j\leq |\mathscr{P}|$ par,
    \begin{equation}\label{eq:bonf1}
      \left|A_{\setminus \mathscr{P}}\right| \leq \sum_{i=0}^{j} (-1)^i
      \sum_{P\subset \mathscr{P}: |P|=i} \left|A_{P}\right|,
    \end{equation}
   mientras que, para
   $j\leq |\mathscr{P}|$ impar, la desigualdad va en la otra direcci\'on:
   \begin{equation}\label{eq:bonf2}
     \left|A_{\setminus \mathscr{P}}\right| \geq \sum_{i=0}^{j} (-1)^i
     \sum_{P\subset \mathscr{P}: |P|=i} \left|A_{P}\right|
     .\end{equation}
    El primer resultado no trivial de cribas ({\em criba pura de Brun})
    se bas\'o en estas desigualdades. En la teor\'ia de probabilidades,
    se llaman {\em desigualdades de Bonferroni}, aunque el trabajo de
    Bonferroni \cite{zbMATH03026527}
    es posterior a aquel de Brun \cite{brun1915uber}.

    Para ver porque estas desigualdades pueden ser \'utiles para las cribas,
    acote el n\'umero de t\'erminos del las sumas en el lado derecho de
    (\ref{eq:bonf1}) o (\ref{eq:bonf2}), y comp\'arelo, para $j$ peque\~no,
    con el n\'umero $2^{|\mathscr{P}|}$ de t\'erminos de la suma para
    $j = |\mathscr{P}|$. T\'ipicamente, cada cantidad
    $\left|A_{P}\right|$ se estima de manera aproximada, con un error, y por
    lo tanto el t\'ermino de error total crece con el n\'umero de t\'erminos.
    As\'i, las desigualdades (\ref{eq:bonf1}) y (\ref{eq:bonf2}), tomadas juntas, pueden terminar dando  un resultado m\'as preciso
    que la igualdad (\ref{eq:incl_excl}).
    
\item  Discutamos algunos resultados generales, concretos y no triviales
  de cribas, utilizando la notaci\'on que acabamos de explicar.
  
  Suponemos siempre que hay una constante
  $\kappa>0$ (llamada la {\em dimensi\'on} de la criba) tal que
  \begin{equation}\label{eq:dimencrib}
    \frac{V(w)}{V(z)} \leq K \left(\frac{\log z}{\log w}\right)^\kappa
    \end{equation}
  para todo $1<w\leq z$ y alguna constante $K$, donde $V(y) = \prod_{p\in \mathscr{P}: p\leq y} (1-g(p))$, $g$ cumpliendo \eqref{eq:doria}. Entonces tenemos el siguiente resultado.

  \begin{lema}[Lema fundamental de las cribas]\label{le:lemfund}
    Sean dados $\{a_n\}_{n\leq x}$, $\mathscr{P}$, $z$, $D$, $g$, $X$, $r_d$, $\kappa$ con las propiedades
    mencionadas. Entonces, para $s = (\log D)/\log z$,
    \begin{equation}\label{eq:lemfund}\mathop{\sum_{n\leq x}}_{p|n\Rightarrow p\notin \mathscr{P}} a_n =
    (1+ O_{\kappa}(K^{O(1)} e^{-s})) X V(z) + R,\end{equation}
    donde \[|R|\leq R(D) = \mathop{\sum_{d\leq D}}_{p|d\Rightarrow p\in \mathscr{P}} |r_d|.\]
  \end{lema}
  \begin{proof}
    Por la criba (no tan pura) de Brun (o varias otras). Ver, por ejemplo, \cite[Cor.~6.10]{MR2647984}
    o \cite[\S 2.8]{MR0424730}.
  \end{proof}

  Para el ejercicio siguiente, es suficiente saber que
  
  \begin{equation}\label{eq:cribcotsup}
    \mathop{\sum_{n\leq x}}_{p|n\Rightarrow p\notin \mathscr{P}} a_n \ll_{\kappa} K^{O(1)} X V(z) + R(D),
  \end{equation}
  lo cual se deduce inmediatamente del Lema \ref{le:lemfund}.
  A tal resultado se le puede llamar {\em criba de cota superior}.
  Se puede obtener (con definiciones del t\'ermino de resto $R(D)$ ligeramente distintas) de muchos
  procedimientos de criba distintos, e.g., la criba de Selberg (\cite[\S 6.4]{MR2061214} o
  cualquier otra fuente), la cual da (\ref{eq:cribcotsup}) sin la dependencia en $K$, bajo 
  condiciones muy generales.
\item 
  \begin{enumerate}
    \item
    Sea $k\geq 2$.
    Veamos como utilizar una criba de cota superior para acotar el n\'umero
    de primos $p$ tal que $p+k$ tambi\'en es primo. 
  \item Como pr\'actica, utilizando (\ref{eq:cribcotsup}) con $a_n = 1$ para $n\leq x$, muestre que
    \[\pi(x) \ll \frac{x}{\log x}.\]
    Use la {\em f\'ormula de Mertens}
    \[\prod_{p\leq y} \left(1 - \frac{1}{p}\right) = \frac{c+o(1)}{\log y},\]
    la cual se deduce del Corolario \ref{cor:corTNP}, pero es hist\'oricamente anterior al teorema
    de los n\'umeros primos. (S\'olo requiere el m\'etodo de Chebyshev en su prueba; ver
    \cite[\S 2.2]{MR2061214}. Por cierto, la constante $c$ es igual a $e^{-\gamma}$, donde $\gamma$
    es la constante de Euler.)
  \item Sean $k,X\in \mathbb{Z}^+$, $x = X (X+k)$. Para $n\leq x$, sea $a_n$ la sucesi\'on indicatriz de los n\'umeros
   $n=m (m+k)$ para alg\'un $1\leq m\leq X$.    Muestre que la condici\'on
    (\ref{eq:doria}) se cumple para todo $d\geq 1$ libre de factores cuadrados con
    \[g(d) = \mathop{\prod_{p\mid d}}_{p\mid k} \frac{1}{p}  \cdot  \mathop{\prod_{p|d}}_{p\nmid k} \frac{2}{p}\]
    y $|r_d|\leq \tau(d)$. Verifique tambi\'en que (\ref{eq:dimencrib}) se cumple con
    $\kappa=2$, $K = K(k) = \prod_{p|k} (1-1/p)/(1-2/p)$ y $\mathscr{P} = \{p\leq z: \text{$p$ primo}\}$,
    para $z$ arbitrario. Aqu\'i
    \[V(y) = \prod_{p\leq y} (1-g(p)).\]
  \item Usando el Lema \ref{le:lemfund} (o (\ref{eq:cribcotsup})), deduzca que,
    para $Y>X>0$ y $k\in \mathbb{Z}^+$ arbitrarios, el n\'umero
    de enteros $Y < m\leq Y+X$ tales que $m (m+k)$ no tiene factores primos de tama\~no $\leq \sqrt{X}$ es
    \[\mathop{\sum_{n\leq x}}_{p|n\Rightarrow p>\sqrt{X}} a_n 
    \ll_{\kappa} K(k)^{O(1)} X V(z) + \sqrt{X} \log(X) \ll K(k)^{O(1)} \frac{X}{(\log X)^2}.\]
    (Claro est\'a, $\sum_{d\leq D} \tau(d) = \sum_{d\leq D} \sum_{l|d} 1 = \sum_{l\leq D} \lfloor D/l\rfloor \ll
    D \log D$.)
    Concluya que el n\'umero de primos $Y < p\leq Y+X$ tal que $p+k$ tambi\'en es primo es
    \begin{equation}\label{eq:primgemcot}\ll K(k)^{C} \frac{X}{(\log X)^2}\end{equation}
    para alguna constante $C$.
    N\'otese que $K(k)\ll \prod_{p|k} (1+1/p)$.

    Por cierto, se cree que la as\'intotica correcta es $K(k) X/(\log X)^2$.
    (Se trata de un caso especial de las {\em conjeturas de Hardy y Littlewood}.)
    No sabemos, empero, siquiera
    si el n\'umero de primos $p$ con $p+2$ primo es infinito o no ({\em problema de los primos
      gemelos}). Lo mismo vale para $p$ y $p+k$, $k>2$.
  \end{enumerate}
\end{enumerate}  
\subsection{Estimaciones de valor medio}

\begin{lema}\label{lem:valmed} $($\cite[Thm.~6.1]{zbMATH03342902}$)$; see also
  \cite[Thm.~9.1]{MR2061214}$)$
  Sean $a_1, a_2,\dotsc,a_N \in \mathbb{C}$. Entonces, para todo $T\geq 0$,
  \begin{equation}\label{eq:valmed}\int_0^T \left|\sum_{n\leq N} a_n n^{i t}\right|^2 dt =
  (T+O(N)) \sum_{n\leq N} |a_n|^2.\end{equation}
\end{lema}
Esta es una cota que no deja de recordar a la {\em gran criba} (de la cual
no precisaremos). Como en ciertas pruebas de la gran criba, es m\'as
sencillo probar una cota con un factor de $\log$ de m\'as. Expliquemos
la prueba as\'i, comenzando con una demostraci\'on de esa cota m\'as d\'ebil.
\begin{proof}
  Hagamos un intento directo e ingenuo. Expandiendo el cuadrado, vemos que
  \begin{equation}\label{eq:apemos}
    \int_0^T \left|\sum_{n\leq N} a_n n^{i t}\right|^2 dt =
    \sum_{n_1,n_2\leq N} a_{n_1} \overline{a_{n_2}} \int_0^T (n_1/n_2)^{i t} dt.
  \end{equation}
  Ahora bien, para todo $r>0$,
  \begin{equation}\label{eq:idio}\int_0^T r^{i t} dt = \frac{r^{i T} - 1}{i \log r}
  = O\left(\frac{1}{\log r}\right).\end{equation}
  Puesto que $\log r \geq (r-1)/(r+1)$ para $r\geq 1$
  (como podemos verificar comparando derivadas),
  $\log(n_1/n_2) \geq (n_1 - n_2)/(n_1+n_2)\geq (n_1-n_2)/2 N$ para
  $1\leq n_2\leq n_1 \leq N$, y as\'i
$\log(n_1/n_2)\geq |n_1-n_2|/2 N$ para $1\leq n_1,n_2\leq N$.
  Por lo tanto,   la expresi\'on en el lado derecho de (\ref{eq:apemos})
  es
  \[\begin{aligned}
  &T \mathop{\sum_{n_1,n_2\leq N}}_{n_1=n_2} a_{n_1} \overline{a_{n_2}}
  + O\left(\mathop{\sum_{n_1,n_2\leq N}}_{n_1\ne n_2} \left|a_{n_1} \overline{a_{n_2}}
  \right|
  \frac{N}{|n_1 - n_2|}\right) 
  \\
  = \; &T \sum_{n\leq N} |a_{n}|^2 + 
  O\left(N \mathop{\sum_{n_1,n_2\leq N}}_{n_1\ne n_2}
  \frac{|a_{n_1}|^2 + |a_{n_2}|^2}{2} \frac{1}{|n_1 - n_2|}\right) 
  = \; (T + O(N \log N)) \sum_{n\leq N} |a_{n}|^2.\end{aligned}\]

  Hemos obtenido casi lo que quer\'iamos. Veamos ahora porqu\'e hay un factor
  de $\log N$ espurio. La integral en (\ref{eq:idio}) no es sino
  $\widehat{1_{\lbrack 0,T\rbrack}}(-(\log r)/2\pi)$, donde $1_I$ es la funci\'on
  caracter\'istica del intervalo $I$.
Estamos en verdad en una situaci\'on muy similar a la de la secci\'on
\S \ref{subs:zetperr}, cuando examinamos la f\'ormula de Perron
(\ref{eq:perron}).

  Como (\ref{eq:idio}) nos muestra,
  la transformada de Fourier de una
  funci\'on $1_I(t)$ decae como $1/|t|$ cuando $t\to \pm \infty$,
  mientras que, como es bien conocido y f\'acil de probar, la transformada
  de Fourier de una funci\'on continua $f$ decae por lo menos tan
  r\'apidamente como $1/t^2$ (esto es, es $O(1/t^2)$ cuando $t\to \pm \infty$).
  El factor de $\log N$ viene de una suma $\sum_{n\leq N} 1/n$; nos lo
  ahorraremos si tenemos una suma de tipo $\sum_{n\leq N} 1/n^2$, que converge.

  Por lo tanto, lo que hacemos es elegir una funci\'on $f$ continua tal que
  $f(t)\geq 1_{\lbrack 0,T\rbrack}(t)$ para todo $t$, y, al mismo tiempo,
  $\int f(t) dt$ no es mucho mayor que $\int 1_{\lbrack 0,T\rbrack}(t) dt = T$,
  y la derivada $f'$, aparte de estar definida en casi todas partes, no tiene
  grandes fluctuaciones (pues esto afectar\'ia la constante en la cota $\widehat{f}(t) = O(1/t^2)$).
    Algo est\'andar es la funci\'on
  ``trapecio'' $f(t)$ igual a $0$ cuando $t\leq -N$ o $t\geq T+N$, igual
  a $1$ para $0\leq t\leq T$, y lineal af\'in en $\lbrack -N,0\rbrack$ y
  $\lbrack T, T+N\rbrack$ (igual a $1+t/N$ en el primer intervalo, y a
  $1-(t-T)/N$ en el segundo). (Es en verdad el mismo tipo de elecci\'on que
  en (\ref{eq:psidelta}). Otras funciones continuas tambi\'en servir\'ian; \'esta
  es simplemente conveniente.)

  Como $f(t)\geq 1_{\lbrack 0,T\rbrack}(t)$,
  \begin{equation}\label{eq:apemos2}\begin{aligned}
    \int_0^T \left|\sum_{n\leq N} a_n n^{i t}\right|^2 dt &\leq
    \int_0^T f(t) \left|\sum_{n\leq N} a_n n^{i t}\right|^2 dt \leq
    \sum_{n_1,n_2\leq N} a_{n_1} \overline{a_{n_2}} 
    \widehat{f}\left(-\frac{\log n_1/n_2}{2\pi}\right) .
  \end{aligned}\end{equation}
  Ahora bien, para $f$ la funci\'on ``trapecio'',
  $\widehat{f}(0) = \int f(t) dt = T + 2 N$, y (ejercicio \ref{ej:ffourquad})
  \[\widehat{f}(x) = O\left(\frac{1 }{N x^2}\right).\]
  Por lo tanto, el valor
  absoluto del lado derecho de (\ref{eq:apemos2}) es 
  \[\begin{aligned} &\sum_{n\leq N} |a_n|^2 (T + 2 N) + 
  O\left(\mathop{\sum_{n_1,n_2\leq N}}_{n_1\ne n_2}
  \frac{|a_{n_1}|^2 + |a_{n_2}|^2}{2} \frac{N}{|n_1 - n_2|^2}\right)
  = (T + O(N)) \sum_{n\leq N} |a_n|^2.\end{aligned}\]  
\end{proof}





Si en la parte izquierda de la igualdad (\ref{eq:valmed}) integr\'asemos en un subconjunto $\mathscr{T}$ de $[0,T]$ en vez de en el total $\mathscr{T}=[0,T]$, entonces parece razonable que se siguiera cumpliendo la desigualdad ($\ll$) sustituyendo en la parte derecha $T$ por algo menor. Esto es lo que demostraremos en la Proposici\'on \ref{prop:halaszmonty}. El problema es que para conseguirlo no sirve la t\'ecnica de la prueba del Lema \ref{lem:valmed}, ya que \'esta depend\'ia de la cancelaci\'on que se obtiene al integrar $(n_1/n_2)^{it}$ en $[0,T]$, y no est\'a claro que se pueda obtener al integrar sobre subconjuntos $\mathscr{T}$ generales de $[0,T]$. Sin embargo, veremos que es posible sustituir dicha cancelaci\'on por la de las sumas $\sum_{n} n^{it}$ para $t$ fijo, que es lo que establecemos a continuaci\'on.

\begin{lema}\label{lem:cotadirichsuma}
  Sea $f:\mathbb{R}^+\to \mathbb{R}$ definido por $f(x) = 1$ para
  $0<x\leq 1$,
 $f(x) = 2-x$ para $1<x\leq 2$, y $f(x)=0$ de lo contrario.
 Entonces, para $x\geq 1$ y $t\in \mathbb{R}$ arbitrario,
  \begin{equation}\label{eq:nitmejor}
    \sum_{n} f(n/x) n^{i t} \ll \frac{x}{(1 + |t|)^2} +
  \sqrt{|t|} \log(1+|t|).\end{equation}
\end{lema}
Como de costumbre, la elecci\'on de $f$ no tiene nada de particular; cualquier
$f$ doblemente diferenciable con $f''$ en $L^1$ nos
dar\'ia una cota de la misma forma que (\ref{eq:nitmejor}).
\begin{proof}[Demonstraci\'on, o m\'as bien dicho un comentario;
  la demonstraci\'on est\'a en los ejercicios]
  En los ejercicios \ref{ej:sumapoisson}--\ref{ej:cotadirichsuma}, para
  ser precisos.
  Se trata de una aplicaci\'on de la f\'ormula
  de sumaci\'on de Poisson, seguidas de cotas generales
  para integrales del tipo $\int_a^b \eta(x) e(\theta(x)) dx$, $\eta$
  cont\'inua y $\theta$ en $C^1$. (Se trata de cotas superiores, no
  de asint\'oticas, para los dos casos principales -- fase no
  estacionaria y fase estacionaria con $\theta''\ne 0$.) Algunos notar\'an
  que las mismas cotas aparecen en el m\'etodo de van der Corput.
  De hecho, los ejercicios est\'an cercanos al tratamiento cl\'asico de tal
  m\'etodo (\cite[\S 3.3]{MR1297543} o \cite[\S 8.3]{MR2061214}), con las mejoras
  que son de esperarse dado que usamos un peso liso $f$.
  Si no se usa tal $f$, se obtiene un resultado ligeramente peor:
 \begin{equation}\label{eq:suma_abrupta}
  \sum_{n\leq x} n^{i t} \ll \frac{x}{1 + |t|} + \sqrt{|t|} \log(1+|t|).
 \end{equation}

\end{proof}
\begin{proof}[Esbozo de otra demonstraci\'on]
  Por el teorema de inversi\'on de Mellin,
  \[\sum_{n} f(n/x) n^{i t} = \frac{1}{2\pi i}
\int_{\sigma-i \infty}^{\sigma+ i \infty} x^s Mf(s) \zeta(s- i t) ds\]
para $\sigma>1$. Desplazamos la integral hacia la izquierda, pasando
por polos en $s=1$ y en $s=0$. Las contribuciones de los polos
dan los t\'erminos del lado derecho de (\ref{eq:nitmejor}).

(Aqu\'i estamos utilizando el hecho que la funci\'on zeta tiene continuaci\'on
anal\'itica a todo el plano complejo, as\'i como la cota est\'andar
$|\zeta(i t)|\ll \sqrt{t} \log t$, y cotas similares para $|\zeta(\sigma+ i t)|$,
$\Re \sigma <0$. Para obtener tales cotas, es necesario utilizar la ecuaci\'on
funcional de $\zeta(s)$ (la cual tambi\'en da la continuaci\'on anal\'itica).
Como la prueba usual de la ecuaci\'on funcional se basa sobre la formulaci\'on
de Poisson, esta demostraci\'on y la anterior no son tan distintas como
parecen a primera vista.)
\end{proof}
\begin{proposicion}[Hal\'asz--Montgomery]\label{prop:halaszmonty} 
  Sean $T\geq 2$ y $\mathscr{T}\subset \lbrack -T,T\rbrack$ medible.
  Sean $a_1, a_2,\dotsc,a_N \in \mathbb{C}$. Entonces
  \begin{equation}\label{eq:halaszmonty}\int_{\mathscr{T}}
    \left|\sum_{n\leq N} a_n n^{i t}\right|^2 dt \ll (N + |\mathscr{T}|
    \sqrt{T} \log T) \sum_{n\leq N} |a_n|^2.\end{equation}  
\end{proposicion}
El lector avisado notar\'a que la prueba tiene gran similitud con una de las
pruebas de la {\em gran criba} (distinta de las {\em peque\~nas cribas}
que consideramos en \S \ref{subs:paridadcribas}). La diferencia consistir\'a
en que aplicamos el Lema \ref{lem:cotadirichsuma} en vez de cotas sobre
sumas exponenciales. Parte del procedimiento (utilizaci\'on de un $f$ continuo)
es como en la prueba del Lema \ref{lem:valmed} que acabamos de ver.
\begin{proof}
  Por el principio de dualidad (ejercicio \ref{it:duallebesgue}),
  basta mostrar que, para $a\in L^2(\mathscr{T})$,
  \[\sum_{n\leq N} \left|\int_{\mathscr{T}} a(t) n^{i t}\right|^2 \ll
  (N + |\mathscr{T}|
  \sqrt{T} \log T) \int_\mathscr{T} |a(t)|^2 dt.\]
  Claro est\'a, para cualquier $f:\lbrack 0,\infty) \to \lbrack 0,\infty)$ con
  $f(x)\geq 1$ para $0\leq x\leq 1$,
  \[\sum_{n\leq N} \left|\int_{\mathscr{T}} a(t) n^{i t}\right|^2 \leq
  \sum_{n} f\left(\frac{n}{N}\right) 
  \left|\int_{\mathscr{T}} a(t) n^{i t}\right|^2 .\]
  Expandiendo, vemos que
  \begin{equation}\label{eq:dobtnit}
    \sum_{n} f\left(\frac{n}{N}\right) 
  \left|\int_{\mathscr{T}} a(t) n^{i t}\right|^2 =
  \int_{\mathscr{T}} \int_{\mathscr{T}} \overline{a(t_1)} a(t_2)
  \sum_n f\left(\frac{n}{N}\right) n^{i (t_2 - t_1)} dt_1 dt_2.\end{equation}
  Usando 
  la desigualdad
  $|a_1 a_2|\leq (|a_1|^2 + |a_2|^2)/2$
  y el Lema \ref{lem:cotadirichsuma}, conclu\'imos que el lado derecho
  de (\ref{eq:dobtnit}) es a lo m\'as
  \[\begin{aligned}\int_{\mathscr{T}} |a(t_1)|^2 
  &\int_{\mathscr{T}}
  \sum_n f\left(\frac{n}{N}\right) n^{i (t_2 - t_1)} dt_1 dt_2\\
  &\ll \int_{\mathscr{T}} |a(t_2)|^2 
\int_{\mathscr{T}}
\left(  \frac{N}{(1 + |t_1 - t_2|)^2} +
  \sqrt{T} \log T \right) dt_1  dt_2 \\&\ll
  \left(N + |\mathscr{T}| \sqrt{T} \log T\right)
  \int_{\mathscr{T}} |\overline{a(t_2)}|^2 dt_2
  ,\end{aligned} \]
  que era lo que quer\'iamos demostrar.
\end{proof}
\subsubsection{Ejercicios}
\begin{enumerate}

 \item En este problema vamos a ver que es sencillo  demostrar el Lema  \ref{lem:cotadirichsuma} en cierto rango para $t$. Demostrarlo para todo $t$ ser\'a m\'as complicado, como veremos en los siguientes ejercicios.
 \begin{enumerate}
  \item Demuestre que podemos extender la regla del rect\'angulo (\ref{eq:regla_rectangulo}) a funciones complejas $f:\mathbb R\to \mathbb C$ usando la desigualdad $|a|+|b|\le \sqrt{2}\sqrt{|a|^2+|b|^2}$.
  \item Aplicando el apartado anterior, pruebe la desigualdad (\ref{eq:suma_abrupta}) en el rango $|t|\le \sqrt{x/\log x}$.
  \item Demuestre el Lema \ref{lem:cotadirichsuma} en el rango $|t|\le (x/\log x)^{1/3}.$
 \end{enumerate}

\item \label{ej:ffourquad}
  Sean $T,D>0$. Sea
  \begin{equation}\label{eq:ftrapecio} f(t) = \begin{cases}
    0 & \text{si $t\leq -D$ o $t\geq T+D$,}\\
    1+t/D &\text{si $-D\leq t\leq 0$,}\\
    1 &\text{si $0\leq t\leq T$,}\\
    1 - \frac{t-T}{D} &\text{si $T\leq t\leq T+D$.}
  \end{cases}\end{equation}
  \begin{enumerate}
  \item Muestre, por integraci\'on por partes, que, para $t\ne 0$,
    \[\begin{aligned}\widehat{f}(t) 
    &= \frac{1}{2\pi i t} \left(\int_{-D}^0 + \int_0^T + \int_T^{T+D} f'(x) e(- x t) dx\right) \\
    &= \frac{1}{2\pi i D t} \left(\int_{-D}^0 e(- x t) - \int_T^{T+D} e(- x t) dx\right). 
    \end{aligned}\]
  \item Concluya que, para $t\ne 0$,
    \[|\widehat{f}(t)|\leq \frac{1/D}{(\pi t)^2}.\]
  \item En general, sea $f:\mathbb{R}\to \mathbb{C}$ tal que (a) $f(x), f'(x)\to 0$ cuando $x\to \pm \infty$,
    (b) $f''$ es continua e integrable (i.e., $|f''|_1 = \int_{-\infty}^\infty |f''(x)| dx$ es finita). Usando
    integraci\'on por partes dos veces, muestre que, para $t\ne 0$,
    \[|\widehat{f}(t)|\leq \frac{|f''|_1}{(2 \pi t)^2}.\]

 {\bf Nota.} La condici\'on que $f''$ sea continua puede ser reemplazada por condiciones bastante m\'as d\'ebiles;
    basta con usar versiones m\'as generales de la integraci\'on por partes (mediante integrales
    de Riemann-Stieltjes, o distribuciones, medidas, etc.). Basta con que $f$ sea continua,
    $f'$ sea definida fuera de un conjunto finito de puntos, $f'$ sea de {\em variaci\'on total} finita
    y $f(x), f'(x)\to 0$ cuando $x\to \pm \infty$. Un enunciado de esta generalidad cubre la funci\'on
    $f$ en (\ref{eq:ftrapecio}).
    \end{enumerate}
\item \label{ej:sumapoisson}
  Sea $f:\mathbb{R}\to \mathbb{C}$ continua e integrable, as\'i como
  diferenciable fuera de un conjunto finito de puntos. Asumamos tambi\'en
  que $f'$ es integrable (en todo intervalo donde est\'a definida)
  y que $\sum_{n\in \mathbb{Z}} |\widehat{f}(n)| < \infty$.
  Entonces
  \begin{equation}\label{eq:poissonsuma}
    \sum_{n\in \mathbb{Z}} f(n) = \sum_{n\in \mathbb{Z}} \widehat{f}(n).\;\;\;\;\;\;\;\;
    \text{(f\'ormula de sumaci\'on de Poisson)}
  \end{equation}
  Veamos como probar esta f\'ormula.
  \begin{enumerate}
  \item Muestre que la funci\'on $F:\mathbb{R}\to \mathbb{C}$ dada por
    $F(t) = \sum_{n\in \mathbb{Z}} f(t+n)$ est\'a bien definida (es decir, la suma
    converge). Verifique tambi\'en que $F$ es de per\'iodo $1$ y que
    $\int_0^1 |F(t)| dt < \infty$.
  \item El teorema
    de inversi\'on de Fourier (para funciones de per\'iodo $1$) nos dice que
    \[F(t) = \sum_{n\in \mathbb{Z}} a_n e(n t),\]
    donde $a_n = \int_0^1 F(t) e(-n t) dt$, bajo la condici\'on que $\sum_n |a_n|< \infty$.
    Muestre que
    $a_n = \widehat{f}(n)$.
  \item Concluya que
    \[\sum_{n\in \mathbb{Z}} f(n) = F(0) = \sum_{n\in \mathbb{Z}} \widehat{f}(n).\]
  \end{enumerate}
\item \label{ej:fasenoestacionaria}
  {\em Fase no estacionaria.} Sean $\theta, \eta: \lbrack a,b\rbrack \to
  \mathbb{R}$ tales que $\eta$ y $\theta'$ son continuas.
  Asumamos que
  $\theta'(x)\ne 0$ para todo $x\in \lbrack a,b\rbrack$. (\'Esta es la suposici\'on
  crucial, llamada ``fase no estacionaria'', pues $\theta(x)$ es la ``fase''.)
  Asumamos tambi\'en que $g(x) = \eta(x)/\theta'(x)$
  es mon\'otona en $\lbrack a,b\rbrack$.

  Muestre, por intregraci\'on por partes, que
  \[\int_a^b \eta(x) e(\theta(x)) dx = \frac{1}{2\pi i} \left(
  g(x) e(\theta(x)) |_a^b + \int_a^b e(\theta(x)) dg(x)\right).\]
  Concluya que
  \begin{equation}\label{eq:fasenoestacionariagen}\left|\int_a^b \eta(x) e(\theta(x)) dx\right|\leq \frac{\max(|g(a)|,|g(b)|)}{\pi}.\end{equation}
  En particular, para $\theta'$ mon\'otona con $\theta'(x)\ne 0$ para todo
  $x\in \lbrack a,b\rbrack$,
  \begin{equation}\label{eq:fasenoestacionaria}
    \left|\int_a^b e(\theta(x)) dx\right|\leq \frac{1/\pi}{\min(|\theta'(a)|,|\theta'(b)|)}
  .\end{equation}
  
  Deduzca de (\ref{eq:fasenoestacionariagen}) la siguiente variante:
  sea $\theta:\lbrack a,b\rbrack\to \mathbb{R}$ diferenciable,
  con $\theta'$ continua y $\theta'(x)\ne 0$ para todo
  $x\in \lbrack a,b\rbrack$, y
  $\eta_1, \eta_2: \lbrack a,b\rbrack\to \mathbb{R}$ tales que
  $\eta_1(x)$ y $\eta_2(x)/\theta'(x)$ son mon\'otonas.
  Entonces, para
  $\eta = \eta_1\cdot \eta_2$,
  \begin{equation}\label{eq:fasenoestacionaria2}\left|\int_a^b \eta(x) e(\theta(x)) dx\right|\leq
    \frac{c}{\pi} \cdot
    \max\left(\frac{|\eta_2(a)|}{|\theta'(a)|},
    \frac{|\eta_2(b)|}{|\theta'(b)|}\right)
    \end{equation}
  con $c=\max(|\eta_1(a)|,|\eta_1(b)|,|\eta_1(a)-\eta_1(b)|)$. Es f\'acil
  obtenir otras variantes: por ejemplo, si, en vez de $\eta_1$ mon\'otona
  en $\lbrack a,b\rbrack$,
  tenemos $\eta_1$ mon\'otona en $\lbrack a,x_0\rbrack$ y en
  $\lbrack x_0,b\rbrack$ para alg\'un $x_0\in \lbrack a,b\rbrack$,
  y $\eta_1(a)=\eta_1(b)=0$, entonces
  (\ref{eq:fasenoestacionaria2}) vale con $c = |\eta_1(x_0)|$.

\item \label{ej:fasequizasestac}
  Sea $\theta:\lbrack a,b\rbrack \to \mathbb{R}$ doblemente diferenciable.
  Asumamos que $\theta''(x) \geq \rho >0$ para todo $x\in \lbrack a,b\rbrack$.

  Para $\delta>0$ arbitrario, podemos tener $|\theta'(x)|< \delta$
  s\'olo dentro de un intervalo $I$ de longitud $\leq 2 \delta/\rho$. (?`Por qu\'e?)
  Por lo tanto,
  \[\int_{I} e(\theta(x)) dx \leq 2 \delta/\rho.\]
  (\'Esta es la contribuci\'on de la regi\'on de {\em fase estacionaria}, es decir, de
  una vecindad del punto en el cual $\theta'(x)=0$, si tal punto existe.)
  Aplique (\ref{eq:fasenoestacionaria})
  para mostrar que
  \[\int_{I\setminus\lbrack a,b\rbrack} e(\theta(x)) dx \leq 2/\pi \delta.\]
  Escoja $\delta$ de manera \'optima para as\'i concluir que
  \begin{equation}\label{eq:fasequizaestac}
    \int_a^b e(\theta(x)) dx \leq \frac{2}{\sqrt{\pi \rho}}.
  \end{equation}
  
  Aqu\'i tambi\'en podemos introducir un peso $\eta$. Por ejemplo, sea
  $\eta:\lbrack a,b\rbrack\to \lbrack 0,\infty)$ continua tal que
  $\eta$ es mon\'otona en $\lbrack a,x_0\rbrack$ y en
  $\lbrack x_0,b\rbrack$ para alg\'un $x_0\in \lbrack a,b\rbrack$, y adem\'as
  $\eta(a)=\eta(b)=0$. Deduzca de (\ref{eq:fasequizaestac}) que
  \begin{equation}\label{eq:fasequizaestac2}
    \int_a^b \eta(x) e(\theta(x)) dx \leq \frac{2 \max_{x\in \lbrack a,b\rbrack}
    |\eta(x)|}{\sqrt{\pi \rho}}.
  \end{equation}

\item \label{ej:fasenoestacliso}
  Sean $\theta: \lbrack a,b\rbrack \to \mathbb{R}$ diferenciable
  tal que $\theta'$ es mon\'otona y $\theta'(x)\ne 0$
  para todo $x\in \lbrack a,b\rbrack$.  Sea
  $\eta:\lbrack a,b\rbrack \to \lbrack 0,\infty)$ continua,
  as\'i como diferenciable
  fuera de a lo m\'as un conjunto finito de puntos; asumamos tambi\'en que $\eta'$
  es decreciente y acotada, y que $\eta(a)=\eta(b)=0$.

  Como el signo de $\theta'(x)$ es constante, y como podemos hacer un
  cambio de variables $x\mapsto b+a-x$ sin cambiar lo que suponemos sobre
  $\eta$, podemos asumir sin p\'erdida
  de generalidad que $\theta'$ es creciente y positiva en $\lbrack a,b\rbrack$.
  Sea $\alpha = \theta'(a)$.

  Muestre que, por integraci\'on por partes,
  \begin{equation}\label{eq:pripaslis}
    \int_a^b \eta(x) e(\theta(x)) dx = -\frac{1}{\alpha}
  \int_a^b \left(\frac{\eta'(x)}{2\pi i} + \eta(x) (\theta'(x) - \alpha)\right)
  e(\theta(x)) dx.\end{equation}
  Aplique (\ref{eq:fasenoestacionaria2}) y obtenga que 
  \[\left|\int_a^b \left(\frac{\eta'(x)}{2\pi i} + \eta(x) (\theta'(x) - \alpha)
  \right)
e(\theta(x)) dx\right|\leq
\frac{\eta'(a)-\eta'(b)}{2 \pi^2 \alpha} +
\frac{\max_{x\in \lbrack a,b\rbrack} |\eta(x)|}{\pi \alpha}.\]
Por lo tanto \begin{equation}\label{eq:fasenoestacliso}
  \left|  \int_a^b \eta(x) e(\theta(x)) dx \right|\leq
  \frac{c}{\min(|\theta'(a)|,|\theta'(b)|)^2}\end{equation}
donde $c = (\eta'(a)-\eta'(b))/2 \pi^2 +
(\max_{x\in \lbrack a,b\rbrack} |\eta(x)|)/\pi$.

\item \label{ej:cotadirichsuma}
  Sea $\eta:\mathbb{R}\to \mathbb{R}$ dada por
  $\eta(x) = 2 x - 1$ para $1/2\leq x\leq 1$,
  $\eta(x) = 2-x$ para $1\leq x\leq 2$ y $\eta(x)=0$ cuando $x<1/2$ o $x>2$.
  Quisi\'eramos estimar
  \[\sum_n \eta(n/X) x^{i t}\]
  para $t\ne 0$ arbitrario y $X>1/2$. (Para $0<X\leq 1/2$,
  la suma es trivialmente igual a $0$.)

  Por la formula de sumaci\'on de Poisson (\ref{eq:poissonsuma}),
  \[\sum_n \eta(n/X) x^{i t} = \sum_{n\in \mathbb{Z}}
  \int_0^\infty \eta(x/X) x^{i t} e(- n x) dx,\]
  asumiendo que la suma en el lado derecho converge absolutamente.
  
  \begin{enumerate}
  \item Estimemos primero la contribuci\'on de $n=0$. Por integraci\'on por partes,
  \[\int_0^\infty \eta(x/X) x^{i t} dx
  = -
  \int_0^\infty \frac{\eta'(x/X)}{X} \frac{x^{i t +1}}{i t + 1} d x.\]
  Utilice integraci\'on por partes una vez m\'as para
  concluir que
  \[\int_0^\infty \eta(x/X) x^{i t} dx \ll \frac{X}{( |t| + 1)^2}.\]
\item Sea $n\in \mathbb{Z}$ tal que ya sea $|n|\geq  2 t /\pi X$ o $n$
  es de signo contrario a $t$.
  Use (\ref{eq:fasenoestacliso}) para mostrar que
  \begin{equation}\label{eq:cotcola}
    \int_0^\infty \eta(x/X) x^{i t} e(- n x) dx \ll \frac{1}{n^2}.\end{equation}
\item 
  Concluya que, para $X\geq 2 |t|/\pi$,
  \begin{equation}\label{eq:contint1}
    \sum_n \eta(n/X) n^{i t} \ll \frac{X}{( |t| + 1)^2} + 1.\end{equation}
\item Sea ahora $X< 2 |t|/\pi$.
  Muestre que (\ref{eq:fasequizaestac2}) nos da que
  \[\sum_n \eta(n/X) x^{i t} e(- n x) \ll \frac{X}{\sqrt{|t|}}\]
  para $n$ arbitrario, y en particular para
  $|n|\geq  2 t /\pi X$. Concluya, usando tambi\'en (\ref{eq:cotcola}),
  que
  \begin{equation}\label{eq:contint2}
    \sum_n \eta(n/X) n^{i t} \ll \sqrt{|t|}.\end{equation}
\item Sea $x\geq 1$. Sea $f_x$ la funci\'on
  $t\mapsto f(t/x)$, donde $f$ es como en el enunciado del Lema
  \ref{lem:cotadirichsuma}. Exprese $f_x$ 
  como una suma de funciones del tipo $\eta(n/X)$ para diversos valores de $X$,
  y deduzca de (\ref{eq:contint1}) y (\ref{eq:contint2}) que
  \[\sum_n f(n/x) n^{i t} \ll \frac{x}{(|t|+1)^2} + \sqrt{|t|} (\log |t| +1) +
  \log x\]
  para $t\geq 1$. En otras palabras, el Lema \ref{lem:cotadirichsuma} es cierto.
  \end{enumerate}
\item \label{ej:dualidad} {\em Dualidad.}
  \begin{enumerate}
  \item Sean $V$, $W$ espacios de Hilbert.\footnote{Recordamos que un
    {\em espacio de
      Hilbert} es un espacio lineal $V$
    sobre $\mathbb{R}$ o $\mathbb{C}$ con un producto escalar
    $\langle \cdot,\cdot\rangle$ tal que $V$
    es completo con respecto a la distancia $d(x,y) := |x-y|_2 :=
    \sqrt{\langle x-y,x-y\rangle}$. El \'unico ejemplo que necesitaremos
    es el siguiente: sea $X$ un espacio de Lebesgue de medida finita;
    definamos el producto escalar $\langle v_1,v_2\rangle = \int_X \overline{v_1(x)}
    v_2(x) dx$ para $v_1,v_2:X\to \mathbb{C}$ medibles; entonces el espacio
    lineal $L^2(X)$ de funciones $v:X\to \mathbb{C}$ con $|v|_2<\infty$
    es un espacio de Hilbert.}
    Definimos la norma $|A|$
    de un operador (es decir, una funci\'on lineal) $A:V\to W$ por
    \[|A| = \mathop{\sup_{v\in V}}_{v\ne 0} \frac{|A v|_2}{|v|_2}.\]
    Se dice que $A$ es acotado si su norma es finita. Todo $A$ acotado
    tiene un (\'unico) {\em operador dual} $A^*:W\to V$, definido como la
    funci\'on lineal tal que
    \[\langle A^*v, w\rangle = \langle v,A w\rangle.\]
  \item Sean $X$, $Y$ son dos espacios de Lebesgue de medida
    finita. Sea $V = L^2(X)$ y $W = L^2(Y)$. Definamos
    $A:V\to W$ por
    \[(A v)(y) = \int_X K(x,y) v(x) dx,\]
    donde $K:X\times Y\to \mathbb{C}$ es una funci\'on de imagen acotada.
    (Se comprende que $v\in V$ es una funci\'on $v:X\to \mathbb{C}$, y
    $Av\in W$ es una funci\'on $w:Y\to \mathbb{C}$.)
    Muestre que la funci\'on $A^*:W\to V$ definida por
    \[(A^* w)(x) = \int_Y \overline{K(x,y)} w(y) dy\]
    es el operador dual de $A$.
  \item  Sean $V$, $W$ espacios de Hilbert. Sea $A:V\to W$ un operador
    acotado. La desigualdad de
    Cauchy-Schwarz\footnote{Ya sabemos que el nombre hist\'oricamente correcto es {\em Cauchy} o {\em Cauchy-Bunyakovsky-Schwarz}, pero todo el
      mundo sabe lo que {\em Cauchy-Schwarz} quiere decir.}
    nos dice que $\langle w_1, w_2\rangle \leq |w_1|_2 |w_2|_2$ para
    $w_1,w_2\in W$, con igualdad si $w_1 = \lambda w_2$ para alg\'un $\lambda\in \mathbb{C}$. Deduzca que, para $v\in V$ arbitrario,
    $|A v|_2 = \sup_{w\in W: w\ne 0} \langle w, A v\rangle/|w|_2$. Usando
    este hecho, muestre que
    \[\mathop{\sup_{v\in V}}_{v\ne 0} \frac{|A v|_2}{|v|_2} =
     \mathop{\sup_{v\in V}}_{v\ne 0}
   \mathop{\sup_{w\in W}}_{w\ne 0}
     \frac{\langle w, A v\rangle}{|v|_2 |w|_2} =
        \mathop{\sup_{w\in W}}_{w\ne 0}  \mathop{\sup_{v\in V}}_{v\ne 0} 
        \frac{\langle A^* w, v\rangle}{|v|_2 |w|_2} =
           \mathop{\sup_{w\in W}}_{w\ne 0} 
           \frac{|A^* w|_2}{|w|_2}.\]
           Concluya que
           \begin{equation}\label{eq:dualidad}
\;\;\;\;\;\;\;\;  |A| = |A^*| \;\;\;\;\;\;\;\;\;\;\;\;\;\;\;\;\;\;\;\;
             \text{(principio de dualidad).}
           \end{equation}
   \item\label{it:duallebesgue}
     Sean $X$ e $Y$ dos espacios de Lebesgue de medida finita. Sea
     $K:X\times Y\to \mathbb{C}$ una funci\'on acotada y $C\geq 0$ una constante
     tal que, para todo $v\in L^2(X)$,
     \[\int_Y \left|\int_X K(x,y) v(x) dx\right|^2 dy \leq C \int_X
     |v(x)|^2 dx.\] Concluya que, para todo $w\in L^2(Y)$,
     \[\int_X \left|\int_Y K(x,y) w(y) dy\right|^2 dx \leq
     C \int_Y |w(y)|^2 dy.\]
  \end{enumerate}
\end{enumerate}
\section{Cancelaci\'on de $\lambda$ en intervalos cortos, en promedio}

\subsection{Un primer tratamiento}

Queremos demostrar el Teorema \ref{te:mr1}. La demostraci\'on pasar\'a por
%
la funci\'on $Z_\lambda(s) = \zeta(2 s)/\zeta(s)$, asociada con $\lambda$.
En este contexto, es m\'as f\'acil usar sumas donde los intervalos est\'en expresados en t\'erminos de $n/x$.

\begin{teorema}\label{te:intervalos_cortos_m}
    Sea $h=h(X)$ tal que $h(X)\to \infty$ cuando
  $X\to \infty$. Entonces, cuando $X\to \infty$,
\[
 \mathbb E_{(1-\frac hX)N<n\leq N} \lambda(n)=o_h(1)
 \]
  para todo entero $N\in \lbrack X, 2 X\rbrack$ fuera de un conjunto de $o(X)$
  excepciones. 
\end{teorema}
Aqu\'i hemos usado la notaci\'on $\mathbb E_{n\in A} f(n)=\frac{1}{|A|} \sum_{n\in A} f(n)$. En verdad podemos leer $\mathbb{E}_{X<n\leq 2X}$,
$\mathbb E_{N<n\leq (1+\frac hX)N}$, etc., como
$(1/X) \sum_{X<n\leq 2X}$, $(1/(X/h)) \sum_{N<n\leq (1+\frac hX)N}$, etc.,
ya que $|\{X<n\leq 2X\}| = X + O(1)$, y estamos trabajando asint\'oticamente.

Es sencillo ver que el Teorema \ref{te:mr1} es una consecuencia del Teorema \ref{te:intervalos_cortos_m} (ejercicio \ref{ej:multi_implica_aditivo}), as\'i que vamos a concentrarnos en demostrar este \'ultimo. Es una consecuencia directa de la siguiente cota de varianza.
\begin{teorema}\label{te:promedio_cortos_m}
    Sea $h=h(X)$ tal que $h(X)\to \infty$ cuando
  $X\to \infty$. Entonces, cuando $X\to \infty$,
\[
 \mathbb E_{X<x\leq 2X} |\mathbb E_{(1-\frac{h}{X})x<n\le x} \lambda(n)|^2 =o_h(1).
\]
\end{teorema}

Nuestra tarea en esta secci\'on ser\'a probar este teorema, con, por cierto,
una cota precisa en vez de $o_h(1)$ (Teorema \ref{te:result_principal}).
Para empezar, veamos
 la relaci\'on entre las sumas $\mathbb E_{(1-\frac{h}{X})x<n\leq x} f(n)$  y $Z_f(s)$.

\begin{proposicion}\label{pr:corto_zeta}
  Sea $f:\mathbb{N}\to \mathbb{C}$ una funci\'on con soporte en $(X,2X]$
  tal que $|f(x)|\leq 1$ para todo $x$.
  Entonces si $x\in (X,2X]$
\[
\mathbb E_{(1-\frac{h}{X})x < n\leq x}
f(n)= \frac{1}{2\pi i}\int_{1-i\frac{X}{h\delta^{2}}}^{1+i\frac{X}{h\delta^{2}}} x^{s-1} M \psi_{\delta,h}(s) Z_f(s) \, ds + O(\delta)\]
para cierta funci\'on $\psi_{\delta,h}:(0,+\infty)\to \R$ con  $M \psi_{\delta,h}(s)\ll \min(1,\frac{X/h}{|s|},\frac{X/h}{\delta |s|^2})$ para $\Re s > 0$.
\end{proposicion}
\begin{proof}
  La suma inicial puede escribirse como $\frac{1}{x}\sum_{n} f(n) \frac{1}{h/X}1_{(1-\frac hX,1\rbrack}(\frac nx)$. Lo que vamos a hacer es, como en el Lema \ref{le:perron}, aproximar la funci\'on $\frac{1}{h/X}1_{(1-\frac{h}{X},1)}$ por una funci\'on $\psi_{\delta,h}$ que (a) tenga soporte en $[1-\frac hX,1]$, (b) valga $\frac{1}{h/X}$ en $[1-(1-\delta)\frac hX,1-\delta\frac{h}{X}]$, y entre $0$ y $\frac{1}{h/X}$ en el resto del
  soporte. Podemos tomar
  una funci\'on ``trapecio'' similar a (\ref{eq:psidelta}). Entonces
  \[
 \psi_{\delta,h}(y)= \frac{1}{2\pi i} \int_{1-i\infty}^{1+i\infty} M\psi_{\delta,h}(s) y^{-s} ds
\]
con $M\psi_{\delta,h}$ cumpliendo las condiciones del enunciado. Por lo
tanto,
\[ \mathbb E_{(1-\frac{h}{X})x < n\leq x} f(n)= \frac{1}{2\pi i}\int_{1-i\infty }^{1+i\infty} x^{s-1} M \psi_{\delta,h}(s) Z_f(s) \, ds + O(\delta).
\]

Como $f$ es acotada y tiene soporte en $[X,2X]$, tenemos que $|Z_f(1+it)|\ll 1$.
Por el decaimiento de $M\psi_{h,\delta}$ podemos cortar la integral a altura $\delta^{-2}X/h$ perdiendo s\'olo $O(\delta)$.
\end{proof}

Ahora vamos a aplicar esta expresi\'on para evaluar el promedio de sumas cortas de $f$.

\begin{proposicion}\label{pr:parseval}
  Sea $f:\mathbb{N}\to \mathbb{C}$ una funci\'on con soporte en $(X,2X]$
  tal que $|f(x)|\leq 1$ para todo $x$.  Entonces
\[
 \mathbb E_{X<x\leq 2X} |\mathbb E_{(1-\frac{h}{X})x<n\leq x} f(n)|^2
\ll \int_{1-i\frac{X}{h\delta^2}}^{1+i\frac{X}{h\delta^2}}  |Z_f(s)|^2 ds +O(\delta)
\]
\end{proposicion}
\begin{proof}
Queremos usar la Proposici\'on  \ref{pr:corto_zeta}, expandir el cuadrado e intentar aprovechar la sumaci\'on en $x$. Para que esto funcione mejor, de nuevo es conveniente introducir una funci\'on suave que reemplaze a $1_{[X,2X]}$. En este caso
\begin{equation}\label{eq:eqvar0}
 \mathbb E_{X<x\leq 2X} |\mathbb E_{(1-\frac{h}{X})x<n<x} f(n)|^2\ll \int \eta\left(\frac{x}{X}\right) |\mathbb E_{(1-\frac{h}{X})x<n<x} f(n)|^2 \frac{dx}{x}
\end{equation}
para cierta funci\'on $\eta$ con $0\le \eta(t)\le 1$,  soporte en $[1/2,3]$, $C^{2}$
y tal que $\eta''$ es integrable. Ahora, usando la Proposici\'on \ref{pr:corto_zeta} y expandiendo el cuadrado tenemos que la parte derecha de (\ref{eq:eqvar0}) es
\[
O(\delta)+\frac{1}{(2\pi)^2} \int_{1-i\frac{X}{h\delta^2}}^{1+i\frac{X}{h\delta^2}}\int_{1-i\frac{X}{h\delta^2}}^{1+i\frac{X}{h\delta^2}} (M\psi_{\delta,h})(s_1)\overline{M\psi_{\delta,h}(s_2)} Z_f(s_1) \overline{Z_f(s_2)} I(t_1,t_2) ds_1 ds_2,
\]
donde $t_i = \Re s_i$, y con 
\[
I(t_1,t_2) =\int \eta\left(\frac{x}{X}\right) x^{i(t_1-t_2)} \frac{dx}{x}=X^{i(t_1-t_2)} \int \eta(e^y) e^{i(t_1-t_2)y} \, dy.
\]
Por integraci\'on por partes,
$I(t_1,t_2)\ll (1+|t_1-t_2|)^{-2}$. (En general, para  $f$
doblemente diferenciable tal que $f$ y $f''$ son integrables, tenemos
$\widehat{f}(t) \ll (1+|t|)^{-2}$, por integraci\'on por partes.)
As\'i, usando la cota $M \psi_{\delta,h}\ll 1$ y
$|Z_f(s_1)Z_f(s_2)|\le |Z_f(s_1)|^2+|Z_f(s_2)|^2$ obtenemos el resultado.
\end{proof}

Una estimaci\'on de valor medio (Lema \ref{lem:valmed}) permite mostrar que  tenemos una desigualdad en la otra direcci\'on. Esta desigualdad nos dice que no hemos perdido nada con respecto a la cota trivial.

\begin{proposicion}
\label{pr:zeta_trivial}
Si $f$ tiene soporte en $(Y,2Y\rbrack$, entonces
\[
 \int_{1-iY}^{1+iY} |Z_f(s)|^2 \, |ds|\ll \mathbb E_{Y<n\leq 2Y} |f(n)|^2.
\]
\end{proposicion}
\begin{proof}
Por el Lema \ref{lem:valmed}.
\end{proof}

La ganancia sobre esta cota trivial va a venir de factorizar $Z_f(s)=Z_{f_1}(s)Z_{f_2}(s)$.

\begin{proposicion}\label{pr:maximo_fuera}
  Si $f_1$ es una funci\'on con soporte en $(Q,2Q]$ y $f_2$ es una funci\'on con soporte en $\left(\frac{X}{Q},\frac{2X}{Q}\right\rbrack$, entonces,
    para todo intervalo $\mathcal T\subset [1-i\frac{X}{Q},1+i\frac{X}{Q}]$,
\[
 \int_{\mathcal T} |Z_{f_1}(s)Z_{f_2}(s)|^2 |ds|\ll M_1^2 \, \mathbb E_{\frac{X}{Q}<n\leq 2\frac XQ} |f_2(n)|^2 
\]
con $M_1=\max_{s\in \mathcal T}|Z_{f_1}(s)|\le \max |f_1|$.
\end{proposicion}
\begin{proof}
Sacamos el m\'aximo fuera de la integral y aplicamos la proposici\'on \ref{pr:zeta_trivial}.
\end{proof}

Para usar el resultado anterior, queremos factorizar $Z_f(s)$. La idea es usar la factorizaci\'on de un n\'umero $n$ en $n=pm$ con $p$ su primo m\'as peque\~no en cierto intervalo $P_0<p<Q_0$. Para poder hacer esto primero necesitamos ver que la mayor\'ia de n\'umeros $n$ tienen alg\'un primo en dicho intervalo, lo cual est\'a asegurado por el Lema \ref{le:casiprimos} si tomamos $P_0=Q_0^{\alpha}$ grande con $\alpha$ peque\~no. Ahora, si consideramos los n\'umeros con primos en un intervalo  $(P_0,Q_0\rbrack$, $P_0=Q_0^{\alpha}$, y si sacamos el primo m\'as peque\~no en dicho intervalo tenemos la factorizaci\'on
\begin{equation}\label{eq:cuasifactorizacion}
  \sum_{\substack{X<n\leq 2X \\ \exists p\in (P_0,Q_0\rbrack, p\mid n }} f(n)n^{-s} =
  \sum_{P_0< p\leq Q_0} f(p)p^{-s}
  \sum_{\substack{m\\X <m p \leq 2 X \\ p'\mid m\Rightarrow p'\not\in \lbrack P_0,p)}}  f(m) m^{-s}
\end{equation}
para $f$ totalmente multiplicativa (con $p'$ primo). El problema es que esta no es una factorizaci\'on en dos funciones zeta, ya que en el sumatorio de dentro aparece la variable $p$ en dos condiciones.

Eliminar la dependencia en $p$ de la condici\'on $X/p<m\leq 2 X/p$ es
algo completamente de rutina. Eliminar la dependencia en $p$ de la
condici\'on $p'\not\in (P_0,p)$ tambi\'en es factible. Lo haremos de manera
ligeramente distinta a \cite{MR3488742} (o al art\'iculo expositorio
\cite{MR3666035}, el cual tambi\'en da su propia versi\'on).


Primero, un poco de notaci\'on. Para $\eta_1,\eta_2:\mathbb{R}^+\to \mathbb{C}$,
denotaremos por $\eta_1\ast_M \eta_2$ la convoluci\'on
    multiplicativa en $\mathbb{R}^+$ (``convoluci\'on de Mellin''):
  \[(\eta_1\ast_M \eta_2)(x) = \int_0^\infty
  \eta_1(x/t) \eta_2(t) \frac{dt}{t}.\]
  Para $\upsilon_1,\upsilon_2:\mathbb{Z}^+\to \mathbb{C}$, denotaremos
  simplemente por $\upsilon_1\ast \upsilon_2$ la convoluci\'on multiplicativa
  en $\mathbb{Z}^+$ (``convoluci\'on de Dirichlet''):
  \[(\upsilon_1\ast \upsilon_2)(n) = \sum_{d|n} \upsilon_1(n/d) \upsilon_2(d)
  = \sum_{d|n} \upsilon_1(d) \upsilon_2(n/d).\]
  
\begin{lema}\label{le:factoriza}
  Sean $X\geq 1$, $0<\alpha,\beta,\delta<1$,
  $Q_0\leq \exp((\log X)^{1-\beta})$ y $P_0 = Q_0^\alpha$. Definamos
  $u_X:\mathbb{Z}^+\to \mathbb{R}$ por
  \begin{equation}\label{eq:uXdef}
    u_X(n) = \frac{1}{\log (1+\delta)} \int_{P_0}^{Q_0}
  \left(\upsilon_{1,Q,\delta} \ast \upsilon_{2,Q(1+\delta),X/Q}\right)(n)\; \frac{d Q}{Q},\end{equation}
  donde
  \begin{equation}\label{eq:bedrich}
\upsilon_{1,Q,\delta}(n) = \begin{cases} 1 & \text{si $n$ es primo y $Q<n\leq (1+\delta) Q$},\\ 0 &\text{de otra manera,}\end{cases}\end{equation}
\begin{equation}\label{eq:smetana}
  \upsilon_{2,r,Y}(m) = \begin{cases} 1 & \text{si $m\in (Y,2Y\rbrack$ y
    $p'|m\Rightarrow p'\notin \lbrack P_0,r),$}\\
  0 & \text{de otra manera.}\end{cases}\end{equation}
Entonces existe un conjunto $\Err \subset (X, 2(1+\delta) X\rbrack$ con
\[\frac{|\Err|}{X} \ll \alpha + \delta + 
\exp(-(\log P_0)^{\min(3/5,\beta)+o(1)})\]
tal que $u_X(n) = 1_{(X,2 X\rbrack}(n)$ para $n\notin \Err$. M\'as a\'un,
$0\leq u_X(n)\leq 1$ para todo $n\in \mathbb{Z}^+$.
\end{lema}
\begin{proof}
  Por el Lema \ref{le:casiprimos},
  existen a lo m\'as
  \[\begin{aligned}
  &\alpha X + X\cdot O\left(\max\left(\exp(-(\log P_0)^{3/5+o(1)}),
  \exp\left(- \frac{\log X}{3 \log Q_0}\right)\right)\right) + \frac{X}{P_0}\\
  &\ll (\alpha + O(\exp(-(\log P_0)^{\min(3/5,\beta)+o(1)}))) X
  \end{aligned}\]
  elementos de $(X,2 X\rbrack$ tales que no hay ning\'un primo
  $p\in (P_0, Q_0]$ tal que $p|n$. Inclu\'imos todos estos elementos en el
    conjunto $\Err$.
    
Es f\'acil ver que la diferencia
\[ \eta_\Delta(x) = \left(1_{(X,2X\rbrack}-\frac{1_{(1,1+\delta\rbrack} \ast_M 1_{(X,2X\rbrack}}{\log(1+\delta)}\right)(x)\]
se desvanece cuando $x\leq X$, $(1+\delta) X< x\leq 2 X$ o
$x> 2 (1+\delta) X$.
Inclu\'imos en $\Err$, entonces, todos los elementos de $(X,(1+\delta) X]$
  y $(2 X, 2 (1+\delta) X]$.

Por definici\'on y un cambio de variables $t = m Q$, tenemos
\[\begin{aligned}\left(1_{(1,1+\delta\rbrack} \ast_M 1_{(X,2X\rbrack}\right)(p m)
  &= \int_0^\infty 1_{(1,1+\delta\rbrack}\left(\frac{p}{Q}\right)
  1_{(X,2X\rbrack}(m Q) \frac{dQ}{Q}\\
  &= \int_{(1-\delta) P_0}^{Q_0} 1_{(Q,(1+\delta) Q\rbrack}(p)
  1_{\left(\frac{X}{Q},\frac{2X}{Q}\right\rbrack}(m)
  \frac{dQ}{Q}\end{aligned}\]
  para $P_0<p\leq Q_0$ y $m$ arbitrario.
  As\'i, para todo $n$ que no hayamos ya inclu\'ido en $\Err$,
  \begin{equation}\label{eq:dvorak}
    1_{(X,2X\rbrack}(n) = 
    \frac{1}{\log (1+\delta)} \int_{(1-\delta) P_0}^{Q_0}
    \sum_{p|n} 1_{(Q,(1+\delta) Q\rbrack}(p)
    \upsilon_{2,p,X/Q}(n/p) \frac{dQ}{Q},
  \end{equation}
  donde $\upsilon_{2,r,Y}$ es como en (\ref{eq:smetana}).
  (Est\'a claro que la suma $\sum_{p|n}$ puede tener a lo m\'as un t\'ermino
  no nulo, ya que $\upsilon_{2,p,X/Q}(n/p)$ se anular\'a a menos que $p$
  sea el factor primo $\geq P_0$ m\'as peque\~no de $n$.)
  
  Como el conjunto $S_2$ de enteros $X<n\leq 2 (1+\delta)X$ con por lo menos
  un factor primo en el rango $((1-\delta) P_0, (1+\delta) P_0\rbrack$
  tiene a lo m\'as
  \[\begin{aligned}
  \sum_{p\in ((1-\delta) P_0, (1+\delta) P_0\rbrack} \left(\frac{(1 + 2\delta)X}{p}+
  O(1)\right) &= (1+2\delta) X
  \log \frac{\log (1+\delta) P_0}{\log (1-\delta) P_0}
  +   O\left((1+\delta) P_0\right)\\
  &\ll \frac{\delta X}{\log P_0} + P_0\end{aligned}\]
  elementos, podemos permitirnos cambiar el rango de la integral
  en (\ref{eq:dvorak}) de $\lbrack (1-\delta) P_0,Q_0\rbrack$ a
  $\lbrack P_0,Q_0\rbrack$, a\~nadiendo $O(\delta X/\log P_0 + P_0)$
  elementos a $\Err$. 
  
  Finalmente, debemos examinar lo que sucede cuando cambiamos 
  $\upsilon_{2,p,X/Q}$ en (\ref{eq:dvorak}) (con la integral de $P_0$
  a $Q_0$) por $\upsilon_{2,Q (1+\delta),X/Q}$.
  Los \'unicos enteros afectados est\'an en el conjunto
  $S_3$ de
  enteros $X<n\leq 2(1+\delta) X$ divisibles por alg\'un producto $p p'$ de
  dos primos $p$, $p'$ con $P_0\leq p\leq Q_0$  y $p\leq p'<(1+\delta) p$,
  y tales que ning\'un primo $P_0\le p''\leq p$ divide $n$.
  Ahora bien, 
  \[\begin{aligned}|S_3| &\leq
   \sum_{P_0\leq p\leq Q_0} \sum_{p_1 < p_1'<(1+\delta) p_1}
       \left|\{X<n\leq 2 (1+\delta) X: p_1,p_1'|n\}\right|\\
    &\ll \sum_{P_0\leq p\leq Q_0} \sum_{p < p'<(1+\delta) p}
    \left(\frac{X}{p p'} + O(1)\right)\\
    &\ll X \sum_{P_0\leq p\leq Q_0} \frac{1}{p}
    \left(\log \frac{\log (1+\delta) p}{\log p} +
    O(\exp(-(\log p)^{3/5+o(1)}))\right)
    + O(Q_0^2)\\
    &\ll X \sum_{P_0\leq p\leq Q_0} \frac{\delta}{p \log p} +
    X \sum_{P_0\leq p\leq Q_0} \frac{1}{p e^{(\log p)^{3/5+o(1)}}}
    + Q_0^2\\
&\ll \frac{\delta X}{\log P_0} +
X \exp(-(\log P_0)^{3/5+o(1)}),\end{aligned}\]
  donde acotamos las dos sumas en la pen\'ultima l\'inea por el teorema de los
  n\'umeros primos y sumaci\'on por partes. Inclu\'imos $S_3$ en $\Err$.
  Tenemos ahora
\[
1_{(X,2X\rbrack}(n) = u_X(n)\]
para todo $n\notin \Err$, donde
\[u_X(n) = 
    \frac{1}{\log (1+\delta)} \int_{(1-\delta) P_0}^{Q_0}
    \sum_{p|n} 1_{(Q,(1+\delta) Q\rbrack}(p)
    \upsilon_{2,Q (1+\delta),X/Q}(n/p) \frac{dQ}{Q}.\]

    Por \'ultimo, como $\upsilon_{2,Q (1+\delta),X/Q}(n/p)$ puede ser no nulo
    s\'olo cuando $p$ es el divisor primo $\geq P_0$ m\'as peque\~no de
    $n$, y  $1_{(Q,(1+\delta) Q\rbrack}(p)\ne 0$ s\'olo cuando $Q$
    est\'a entre $p/(1+\delta)$ y $p$, vemos que, para cualquier $n$,
    $0\leq u_X(n)\leq 1$.
\end{proof}

\begin{corolario}\label{cor:factorizacion}
  Sea $f:\mathbb{N}\to \mathbb{C}$ una funci\'on completamente
  multiplicativa
  tal que $|f(n)|\leq 1$ para todo $n\in \mathbb{N}$.  
  Entonces, para $0<\alpha,\beta,\delta<1$,
  $Q_0\leq \exp((\log X)^{1-\beta})$ y $P_0 = Q_0^\alpha$,
\[
Z_{f\cdot 1_{(X,2X]}}=Z_{\err}+ \frac{1}{\log(1+\delta)}
  \int_{P_0}^{Q_0} Z_{f_{1,Q,\delta}} Z_{f_{2,Q (1+\delta),X/Q}} \frac{dQ}{Q} 
\]
donde \[
f_{1,Q,\delta}(n) = \begin{cases} f(n) & \text{si $n$ es primo y $Q<n\leq (1+\delta) Q$},\\ 0 &\text{de otra manera,}\end{cases}\]
\[f_{2,Q,Y}(m) = \begin{cases} f(m) & \text{si $m\in (Y,2Y\rbrack$ y
    $p'|m\Rightarrow p'\notin \lbrack P_0,Q),$}\\
  0 & \text{de otra manera,}\end{cases}\]
y $\err:\mathbb{N}\to \mathbb{C}$ es una funci\'on con soporte en
$(X, 2 (1+\delta) X\rbrack$ tal que
\begin{equation}\label{eq:tamanerr}
  \frac{1}{X} \sum_{X<n\leq 2 (1+\delta) X} |\err(n)|^2\ll \alpha + \delta + 
  \exp(-(\log P_0)^{\min(3/5,\beta)+o(1)})
  .\end{equation}
\end{corolario}
\begin{proof}
  Se sigue inmediatamente del Lema \ref{le:factoriza}.
\end{proof}

Ahora que tenemos esencialmente una  factorizaci\'on de $Z_{\lambda 1_{[X,2X]}}$, la idea es mostrar que el m\'aximo del factor
\[
 Z_{\lambda_{1,Q,\delta}}(1+it)=\sum_{Q<p\leq Q+\delta Q} \lambda(p)p^{-1-it}=-\sum_{Q<p<Q+\delta Q} p^{-1-it}.
\]
es peque\~no. En realidad esto no es siempre cierto, ya que para $t$ peque\~na el signo de $\Re p^{-it}$ va a ser positivo, por lo cual no va a haber cancelaci\'on.
Afortunadamente, para $t$ peque\~na, ya sabemos que
$Z_{\lambda 1_{[X,2X]}}(1+it)$ es peque\~no por el Corolario \ref{cor:sumlambinv}.
Cuando $t$ no es peque\~na, la cancelaci\'on en la suma
$\sum_{Q<p<Q+\delta Q} p^{-1-it}$ nos permite acotarla de la manera siguiente.

\begin{proposicion}\label{pr:ptgrande}
  Para $\exp((\log x)^a)\leq t\leq \exp((\log x)^{(3/2) (1-a)})$,
  $a>0$, $0<\delta\leq 1$,
\[
\sum_{x<p\leq (1+\delta) x} p^{-1-i t} \ll \exp(- (\log x)^{a+o(1)}).
\]
\end{proposicion}

Es, por cierto, necesario poner no s\'olo una cota inferior sobre $t$
como condici\'on,
sino tambi\'en alguna cota superior, tal y como lo hemos hecho aqu\'i.
\footnote{Para ver (de manera informal) la
necesidad de una cota superior para obtener un resultado no trivial,
notamos que tendremos cancelaci\'on en la suma $\sum_{x<p\leq (1+\delta) x}
p^{-1-i t}
$ s\'olo si el n\'umero complejo
\[
 p^{-it}=\exp(it\log p)
 \]
 cambia de argumento lo suficiente.
 No es plausible que esto pase para todo $t$ grande, por el siguiente argumento
 heur\'istico. Deber\'iamos tener que
 $\{t\log p/2\pi\}<1/8$ con probabilidad $1/8$, para $t$ tomado al azar
 en un intervalo grande. Si tales eventos probabil\'isticos (uno por primo)
 son aproximadamente independientes -- lo cual es generalmente visto
 como plausible -- entonces,
 con probabilidad $\gtrsim (1/8)^{\delta Q/\log Q}$,
tendremos que
 $\{t\log p/2\pi\}<1/8$ para todo $Q < p\leq (1+\delta) Q$.
 Por tanto, deber\'ia haber alg\'un $t\leq 8^Q$ (digamos) para el cual
 este es el caso, lo cual implicar\'ia que $\Re p^{-i t}\geq 1/\sqrt{2}$
 para todo $x<p\leq (1+\delta) x$, y por lo tanto
 no habr\'ia suficiente cancelaci\'on:
 $\Re \sum_{x<p\leq (1+\delta) x}
p^{-1-i t}$ ser\'ia $\geq (1/\sqrt{2}) |\{x<p\leq (1+\delta) x\}|$.}

\begin{proof}[Prueba de la Proposici\'on \ref{pr:ptgrande}]
  Por sumaci\'on por partes es suficiente demostrar la desigualdad
  \[\sum_{n\leq x} \Lambda(n) n^{-it}\ll x \exp(-(1/2+o(1)) (\log x)^a).\]
  Como la funci\'on zeta correspondiente es $-\zeta'(s)/\zeta(s)$, esta desigualdad puede
  demostrarse como en la prueba del Teorema \ref{te:TNP},
  pero evitando el polo mediante la elecci\'on $T = |t|/2$
  (digamos), $\delta = 1/\sqrt{T}$. Usando las cotas en el
  Teorema \ref{te:vinogradov_korobov}, as\'i como la cota
  $|M\psi_\delta(s)|\ll 1/s$, obtenemos
  \[\begin{aligned}
  \sum_{n\leq x} \Lambda(n) n^{-i t} &\ll 
  \frac{x (\log T)^{2/3+o(1)}}{\sqrt{T}} +
  x \exp\left(- (\log x) (\log T)^{-2/3+ o(1)}\right)
  (\log T)^{2/3+o(1)}\\
  &\ll \frac{x}{e^{(1+o(1)) (\log x)^a/2}} +
  \frac{x}{e^{(\log x)^{a + o(1)}}} = \frac{x}{e^{(\log x)^{a + o(1)}}}.
  \end{aligned}\]
\end{proof}

\begin{proposicion}\label{pr:cota_maximo}
  Sean $Q_0\leq \exp((\log X)^{1-\beta})$, $P_0 = Q_0^\alpha$,
  $0<\alpha,\beta<1$.
  Sean 
  $h\geq \delta^{-2}Q_0$ y $0<\delta<1$. Entonces \[\begin{aligned}
 \mathbb E_{X<x\leq 2X} |\mathbb E_{(1-\frac{h}{X})x<n\leq x} \lambda(n)|^2
 &\ll \alpha 
 + \frac{\eta^2}{\delta^2 \alpha^2} + \delta +
 \exp(-(\log P_0)^{\min(3/5,\beta)+o(1)}),\end{aligned}\]
  donde
  \begin{equation}\label{eq:defeps}
    \eta = \mathop{\max_{t\in \lbrack \exp(\sqrt{\log X}),X/h\delta^2\rbrack}}_{P_0\leq Q\leq Q_0} |Z_{\lambda_{1,Q}}(1+it)| .\end{equation}
\end{proposicion}
\begin{proof}
  Aplicamos la Proposici\'on \ref{pr:parseval} con $f=\lambda 1_{[X,2X]}$.
  Luego aplicamos el Corolario \ref{cor:sumlambinv}
  para concluir que $Z_{\lambda 1_{[X,2X]}}(1+it)\ll \exp(-(\log X)^{3/5})$ para $t\leq \exp(\sqrt{\log X})$ (digamos).
  As\'i, podemos quitar esa parte de la integral, y en el resto usamos el Corolario \ref{cor:factorizacion} para factorizar $Z_{\lambda 1_{[X,2X]}}$. Aplicando
  la Proposici\'on \ref{pr:zeta_trivial}
  para controlar $Z_{\err}$, llegamos a
\[\begin{aligned}
\mathbb E_{X<x\leq 2X} |\mathbb E_{(1-\frac{h}{X})x<n<x} \lambda(n)|^2
&\ll \int_{1+i\exp(\sqrt{\log X})}^{1+i\frac{X}{h\delta^2}}  \left|
\frac{1}{\delta} \int_{P_0}^{Q_0}  Z_{\lambda_{1,Q}} Z_{\lambda_{2,Q,X/Q}} \frac{d Q}{Q}\right|^2 ds\\
&+ \alpha + \delta + \exp(-(\log P_0)^{\min(3/5,\beta)+o(1)}).
\end{aligned}\]
Por Cauchy-Schwarz y la Proposici\'on \ref{pr:maximo_fuera},
\begin{equation}\label{eq:frlesp}\begin{aligned}
\int_{1+i\exp(\sqrt{\log X})}^{1+i\frac{X}{h\delta^2}}  
&\left| \frac{1}{\delta} \int_{P_0}^{Q_0}  Z_{\lambda_{1,Q}} Z_{\lambda_{2,Q,X/Q}} \frac{d Q}{Q}\right|^2 ds \\
&\leq \frac{1}{\delta^2} \left(\log \frac{Q_0}{P_0}\right) \int_{P_0}^{Q_0}
\int_{1+i\exp(\sqrt{\log X})}^{1+i\frac{X}{h\delta^2}}  
\left|Z_{\lambda_{1,Q}} Z_{\lambda_{2,Q,X/Q}}\right|^2 ds
\frac{d Q}{Q}
\\
&\leq \frac{\eta^2}{\delta^2} \left(\log \frac{Q_0}{P_0}\right)
\int_{P_0}^{Q_0}
\int_{1}^{1+i\frac{X}{h\delta^2}} \left|Z_{\lambda_{2,Q,X/Q}}\right|^2 ds
\frac{d Q}{Q}\\
&\leq \frac{\eta^2}{\delta^2} \left(\log \frac{Q_0}{P_0}\right) \int_{P_0}^{Q_0}
\mathbb{E}_{\frac{X}{Q}<n\leq 2 \frac{X}{Q}} |\lambda_{2,Q,X/Q}|^2
\frac{d Q}{Q}
,\end{aligned}\end{equation}
donde $\eta$ es como en (\ref{eq:defeps}). (En la \'ultima l\'inea,
estamos usando la suposici\'on $h \delta^2 \geq Q$.)
Usando la cota trivial
$\mathbb{E}_{\frac{X}{Q}<n\leq 2 \frac{X}{Q}} |\lambda_{2,Q,X/Q}|^2 \leq 1$,
obtenemos que la expresi\'on en la \'ultima l\'inea de (\ref{eq:frlesp}) es
\[\begin{aligned}
\ll \frac{\eta^2}{\delta^2} \left(\log \frac{Q_0}{P_0}\right)^2 =
\frac{\eta^2}{\delta^2 \alpha^2}.
\end{aligned}\]
\end{proof}

Usando este resultado y la Proposici\'on \ref{pr:cota_maximo} con $Q_0=\delta^3 h$ y $\delta=(\log h)^{-1/2}$, podemos completar nuestro objetivo en cierto
rango.

\begin{teorema}\label{te:caso_h_grande}
  Para $\exp((\log X)^{2/3+\epsilon})\leq h \leq \exp((\log X)^{1-\epsilon})$,
  $0<\epsilon\leq 1/6$, tenemos
\begin{equation}\label{eq:bouh_grande}
  \mathbb E_{X<x\leq 2X} |\mathbb E_{(1-\frac{h}{X})x<n<x} \lambda(n)|^2
  \ll_\epsilon \frac{(\log X)^{\frac{2}{3} + \frac{\epsilon}{2}}}{\log h} \leq \; \frac{1}{
    (\log X)^{\epsilon/2}}.
\end{equation}
\end{teorema}
\begin{proof}
  Aplicaremos la Proposici\'on \ref{pr:ptgrande} para acotar la
  cantidad $\eta$ definida en (\ref{eq:defeps}).
  Escogemos $a = 1-1/(1+3\epsilon/4)$, $\delta\geq h^{-1/3}$,
  $Q_0 = \delta^2 h$,  $P_0 = \exp((\log X)^{2/3+\epsilon/2})$ y
  $\alpha = (\log P_0)/\log Q_0 = (\log X)^{2/3+\epsilon/2}/\log Q_0$.
  De esta manera,
  \[(\log P_0)^{\frac{3}{2} (1-a)} =
  (\log X)^{\left(\frac{2}{3}+\frac{\epsilon}{2}\right) \cdot \frac{3}{2} (1-a)}
  =  \log X.\]
  As\'i, 
  $t\leq X$ implica $t\leq \exp((\log P_0)^{(3/2) (1-a)})$.
  Como $\epsilon\leq 1/6$, vemos que $a\leq 1/9 < 1/2$, y por lo tanto
  $t\geq \exp(\sqrt{\log X})$ implica
  $t\geq \exp((\log Q_0)^a)$. Conclu\'imos que s\'i podemos aplicar
  la Proposici\'on \ref{pr:ptgrande}, la cual nos da que
  \[\eta\ll \exp(-(\log P_0)^{a+o(1)})
  = \exp(-(\log X)^{(2/3+\epsilon/2) (a+o(1))}).\]
  Resulta sensato escoger $\delta = \exp(-(\log X)^{2 a/3})$.
  (La condici\'on $\delta\geq h^{-1/3}$ se cumple para $X$ m\'as grande
  que una constante, ya que $h\geq \exp((\log X)^{2/3+\epsilon})$
  y $a\leq 1/9$.)

  Ahora aplicamos la Proposici\'on \ref{pr:cota_maximo}
  con $\beta = \epsilon$, y obtenemos que
  \[\begin{aligned}
   \mathbb E_{X<x\leq 2X} &|\mathbb E_{(1-\frac{h}{X})x<n<x} \lambda(n)|^2
  \ll\; \frac{(\log X)^{\frac{2}{3} + \frac{\epsilon}{2}}}{\log h}
  + \frac{\exp(-(\log X)^{(2/3+\epsilon/2) (a+o(1))})}{\exp(-(\log X)^{2 a/3}) (\log X)^{-2/3}
  } \\ &+ \exp(-(\log X)^{2 a/3}) +
  \exp\left(-(\log X)^{2 \epsilon/3+ o(1)}\right)
  \ll_\epsilon \frac{(\log X)^{\frac{2}{3} + \frac{\epsilon}{2}}}{\log h}  
  .\end{aligned}\]
\end{proof}
Es f\'acil deducir el resultado para $h>\exp((\log X)^{1-\epsilon})$
del Teorema \ref{te:caso_h_grande} (ejercicio \ref{ej:caso_h_grande}).

\subsubsection{Ejercicios}
\begin{enumerate}
\item\label{ej:multi_implica_aditivo}
Queremos demostrar que el Teorema   
\ref{te:intervalos_cortos_m} implica el Teorema \ref{te:mr1}.
\begin{enumerate}
\item Sea $|f|\le 1$. Demuestre que para $h,Y>1$, $0<\delta<1$
  cualesquiera se cumple
\[
 \int_Y^{(1+\delta)Y} \left|\sum_{ x<n\leq x+h} f(n) \right| d x =O(\delta)\delta Y h  +\int_{Y+h}^{Y+h+\delta Y} \left|\sum_{(1-\frac{h}{Y+h})x<n\leq x} f(n)\right| d x.
\]
\item Aplique el apartado anterior para demostrar que
\[
 \int_X^{2X} \left|\sum_{ x<n\leq x+h} f(n) \right| d x \ll \delta Xh+ \frac{1}{\delta} \max_{X\le X'\le 3X} \int_{X'}^{2X'} \left|\sum_{(1-\frac{h}{X'})x<n\leq x} f(n)\right| d x.
\]
\item Tomando el $\delta$ adecuado, concluya que el Teorema   
\ref{te:intervalos_cortos_m} implica el Teorema \ref{te:mr1}.

\end{enumerate}


\item Para cierto $B>1$ y todo $h\geq X^{1/6}(\log X)^B$, vamos a ver una demostraci\'on del resultado 
\[
 \mathbb E_{X<x\le 2X} |\mathbb E_{(1-\frac{h}{X})x<n\le x} \Lambda(n)- 1|^2 =o(1),
 \]
 usando diferentes resultados conocidos sobre la funci\'on $\zeta(s)$.
 Se trata de un resultado cl\'asico, conocido ya mucho antes de Matom\"aki
 y Radziwi{\l}{\l}.
 
 N\'otese que este resultado implica que hay primos en casi todos los
 intervalos de longitud $h$. Usando t\'ecnicas similares (pero un poco m\'as complicadas) es posible demostrar lo mismo para $\lambda(n)$ en vez de $\Lambda(n)-1$.
 Por otra parte: ?`por qu\'e es que
 la demostraci\'on que hemos hecho en esta secci\'on para $\lambda(n)$ no vale para $\Lambda(n)-1$? (Hay, por otra parte, resultados
 sobre $\Lambda$ que pueden ser probados de manera indirecta usando los trabajos
 de Matom\"aki y Radziwi{\l}{\l} sobre $\lambda(n)$: ver
 \cite[Thm.~1.3]{zbMATH07035922}.)
\begin{enumerate}
\item Como en la prueba de la Proposici\'on \ref{pr:corto_zeta}, demuestre que,
  si  $x\in (X,2X]$, 
\[
\mathbb E_{(1-\frac{h}{X})x < n\leq x}
\Lambda(n)= \frac{1}{2\pi i}\int_{1+\frac{1}{\log X}-i\frac{X}{h\delta^{2}}}^{1+\frac{1}{\log X}+i\frac{X}{h\delta^{2}}} x^{s-1} M \psi_{\delta,h}(s) \frac{-\zeta'(s)}{\zeta(s)} \, ds + O(\delta \log X)\]
para cierta funci\'on anal\'itica $M \psi_{\delta,h}(s)\ll \min(1,\frac{X/h}{|s|},\frac{X/h}{\delta |s|^2})$ para $\Re s > -1$, con $M\psi_{\delta,h}(1)=1+O(\delta)$.

 \item En la zona $-1<\Re s<2$, los \'unicos ceros de la funci\'on $\zeta(s)$ est\'an en $0\le \Re s \le 1$, y si $N(T)$ es el n\'umero de ceros con $|\Im s|\le T$, entonces $N(T)\sim \frac{T}{\pi}\log T$ y $N(T+1)-N(T)\ll \log T$. Adem\'as, si $|s-1|\gg 1$ y $s$ est\'a a distancia $\epsilon$ del cero de $\zeta(s)$ m\'as cercano, entonces $\frac{\zeta'(s)}{\zeta(s)}\ll  \frac{\log (2+|s|)}{\min(1,\epsilon)}$. Usando esa informaci\'on, el apartado anterior  y el teorema de los residuos demuestre que, para alg\'un $|w|\le 1$,
 \[
 \mathbb E_{(1-\frac{h}{X})x < n\leq x}
\Lambda(n)=1- \sum_{\rho} x^{\rho-1} M\psi_{\delta,h}(\rho) +O(\delta (\log X)^2),
 \]
 donde $\rho$ son los ceros de $\zeta(s)$ que satisfacen $|\Im \rho|<\frac{X}{h\delta^2}+w.$
 
\item Usando la f\'ormula del apartado anterior y procediendo como en la demostraci\'on de la Proposici\'on \ref{pr:parseval}, demuestre que 
\[
 \mathbb E_{X<x\le 2X} |\mathbb E_{(1-\frac{h}{X})x<n\le x} \Lambda(n)- 1|^2 \ll \log X\sum_{\rho} X^{-2(1-\Re \rho)}  + O(\delta (\log X)^2).
\]
\item Se sabe que el n\'umero $N(\sigma,T)$ de ceros de $\zeta(s)$ con $\sigma\le \Re \rho\le 1$ y $|\Im \rho|\le T$ satisface $N(\sigma,T)\ll T^{\frac{12}{5}(1-\sigma)}(\log T)^A$, para cierto $A>1$ ({\em teorema de densidad}). Use ese hecho para acotar  $\sum_{\sigma\le \Re \rho\le \sigma+\frac{1}{\log X}} X^{-2(1-\Re \rho)}$ por cierta funci\'on $S(\sigma)$ creciente  si $h\delta^2>X^{1/6}.$
\item Usando los dos apartados anteriores y la regi\'on libre de ceros del Teorema \ref{te:vinogradov_korobov}, demuestre el resultado que se menciona al principio del problema. Observe tambi\'en que si se cumpliera la hip\'otesis de Riemann, podr\'iamos demostrarlo para $h>(\log X)^6$.
 
\end{enumerate}

\item\label{ej:caso_h_grande}
  Usando el Teorema \ref{te:caso_h_grande}, vamos a mostrar que,
  para $\exp((\log X)^{1-\epsilon})\leq h\leq X/2$, $0<\epsilon\leq 1/6$, 
\[    \mathbb E_{X<x\leq 2X} |\mathbb E_{(1-\frac{h}{X})x<n\leq x} \lambda(n)|^2
\ll_\epsilon \frac{1}{(\log X)^{\frac{1}{3} -\epsilon}}.\]
\begin{enumerate}
 \item Sea $j\in \mathbb N$. Si definimos $h_j$ por $1-\frac{h_j}{X}=(1-\frac{h}{X})^{1/j}$, dividiendo la suma interior en $j$ trozos, demuestre que para $|f|\le 1$ se cumple
 \[
 \mathbb E_{X<x\le 2X} |\mathbb E_{(1-\frac hX)x<n\le x} f(n)|^2\le
 j^2
 \max_{X'\in [X-h, X]}\mathbb E_{X'<x\le 2X'} |\mathbb E_{(1-\frac {h_j}{X})x<n\le x} f(n)|^2.
 \]
 
\item Usando el apartado anterior y el Teorema \ref{te:caso_h_grande}, demuestre la desigualdad del enunciado del problema.
\end{enumerate}

\end{enumerate}

\subsection{El caso general: valores excepcionales}

Cuando $h$ es m\'as peque\~no que un cierto nivel -- digamos,
$\exp((\log X)^{2/3})$ -- 
nos encontramos con un obst\'aculo. No podemos hacer que $Q$ sea m\'as
peque\~no que $\exp((\log x)^{2/3+\epsilon})$, pues estar\'iamos fuera del rango
en el cual podemos aplicar la Proposici\'on \ref{pr:ptgrande}. Por otra parte,
tampoco parecer\'ia que podemos tomar $h$ m\'as peque\~no que $Q$: el
rango de integraci\'on en
\[
  \int_{0}^{X/h} |Z_{\lambda_{2,Q,X/Q}}(1+it)|^2 dt
  \]
  es demasiado grande como para que podamos aplicar la Proposici\'on
  \ref{pr:maximo_fuera} (es decir, la estimaci\'on de valor medio que viene
  del Lema \ref{lem:valmed}).

  Existe la alternativa siguiente. Nuestra ganancia viene
  del factor $|Z_{\lambda_{1,Q}}|^2\leq M_{1,Q}^2$ en (\ref{eq:frlesp}).
  Podemos estudiar por separado los valores de
  $t$ para los cuales $|Z_{\lambda_{1,Q}}|$ es excepcionalmente grande --
  digamos $|Z_{\lambda_{1,Q}}|>Q^{-1/9}$. Si demostramos que tal cosa sucede 
  s\'olo para muy pocos valores de $s = 1 + it$, podremos usar el teorema del valor medio de Hal\'asz-Montgomery (Proposici\'on \ref{prop:halaszmonty}) en
  vez del Lema \ref{lem:valmed} para acotar 
\[
 M_{1,Q}^2 \int\limits_{\substack{0<t<X/h \\ |Z_{\lambda_{1,Q}}|>Q^{-1/9}}}  |Z_{\lambda_{2,Q,X/Q}}(1+it)|^2 dt.
\]
Mostremos, entonces, que el conjunto de valores de $t$
para los cuales $|Z_{\lambda_{1,Q}}(1+i t)|>Q^{-1/9}$ es suficientemente peque\~no. 

\begin{lema}\label{le:excepcionales}
Sea $|f|\le 1$ con soporte en los primos de $(Q,2Q\rbrack$  y $(\log T)^9<Q<T^{1/3}$. En ese caso, el conjunto de $t$ en $\lbrack 0,T\rbrack$ tales que $|Z_f(1+it)|>Q^{-1/9}$ tiene medida $O(T^{4/9})$.  
\end{lema}
El exponente $4/9$ no es \'optimo, pero nos ser\'a suficiente.
\begin{proof}
  El teorema del valor medio (Lema \ref{lem:valmed}) aplicado a $|Z_f(1+it)|^2$ no controla bien los valores grandes de $|Z_f(1+it)|$. Para conseguir
  tal control, es mejor usarlo sobre potencias superiores de $Z_f$. Lo 
 aplicaremos a $Z_f(1+it)^{\ell}$ con $\ell\in \mathbb{Z}^+$
  tal que $Q^{\ell-1}<T\leq Q^{\ell}$. As\'i, como
\[
 Z_f(1+it)^{\ell}=\sum_{n} a_{n} n^{-1-it}
\]
donde $a_n$ tiene soporte en $(Q^{\ell},(2Q)^{\ell}\rbrack$
y $|a_n|\le \ell!$, por el Lema \ref{lem:valmed} tenemos
\[\begin{aligned}
\int_{0}^T |Z_f(1+it)^{\ell}|^2
&=
\int_0^T |\sum_{Q^{\ell}< n\leq (2Q)^{\ell}} a_{n} n^{-1-it}|^2 dt
\ll (2Q)^{\ell}  \sum_{n} \frac{|a_{n}|^2}{n^2}\\
&\ll \ell! 2^{\ell}  Q^{-\ell} (\sum_{Q< p< 2Q} 1)^{\ell}\ll \ell! 2^{\ell} \ll \ell^{\ell}.
\end{aligned}\]
As\'i, si $A$ es el conjunto del enunciado, tenemos que
$(Q^{-1/9})^{2\ell}|A| \ll \ell^{\ell}$. Como $Q > (\log T)^9$, sabemos que
$\ell<Q^{1/9}$, y por lo tanto
\[
 |A|\ll  (Q^{\ell})^{3/9}\ll Q^{3/9} T^{3/9}\ll T^{4/9}.
\]
\end{proof}

En el siguiente resultado mostramos que podemos controlar la integral sobre los valores excepcionales de $Z_{\lambda 1_{(X,2X]}}$. 

\begin{teorema}\label{te:excepcionales}
  Sea $h\ge 1$ y $Q_0=\exp((\log X)^\beta)$, donde $2/3<\beta< 1$.
  Entonces, para $0<\rho<1-\frac{2}{3\beta}$,
  $ \alpha=1/(\log Q_0)^{\rho}$ y
 $1/(\log X) \leq \delta\leq 1/(\log Q_0)^\rho$,
\[
\int_{[1,1+i\frac Xh]\setminus \mathscr T_{\lambda,Q_0,\frac 19,\alpha,\delta}}   |Z_{\lambda\cdot 1_{(X,2X]}}(s)|^2 ds \ll_{\beta,\rho} \frac{1}{(\log Q_0)^\rho}
\]
con $\mathscr T_{f,A,\gamma,\alpha,\delta}$ el conjunto de todos los valores de
$s = 1 + i t$, $0\leq t\leq X/h$, tales que
$|Z_{f_{1,Q,\delta}}(s)|\le Q^{-\gamma}$ para todo $Q\in [A^\alpha,A]\cap \mathbb Z$,
donde $f_{1,Q,\delta}(n)=f(n)\cdot 1_{\left(Q,(1 + \delta) Q\right\rbrack}(n)$ si $n$ es primo y 0 en otro caso. 
\end{teorema}
\begin{proof}
  Comenzamos como en la demostraci\'on de la Proposici\'on \ref{pr:cota_maximo};
  es decir, aplicamos el Corolario \ref{cor:sumlambinv}
  para concluir que $Z_{\lambda 1_{(X,2X]}}(1+it)\ll \exp(-(\log X)^{3/5+o(1)})$ para $t\leq \exp(\sqrt{\log X})$ (digamos).
    As\'i, podemos quitar esa parte de la integral, y en el resto usamos el
    Corolario \ref{cor:factorizacion} para factorizar $Z_{\lambda 1_{(X,2X]}}$.
      Aplicando
  (\ref{eq:tamanerr}) y la Proposici\'on \ref{pr:zeta_trivial}
  para controlar $Z_{\err}$, obtenemos
\[\begin{aligned}
&\int_{[1,1+i X/h]\setminus
  \mathscr T_{\lambda,Q_0,1/9, (\log Q_0)^{-\rho},(\log Q_0)^{-\rho'}}}
|Z_{\lambda 1_{(X,2X]}}(s)|^2 ds \\&\ll
  O_\rho\left(\frac{1}{(\log Q_0)^\rho}\right) +
\frac{1}{\delta^2}
\int_{B}
\left|\int_{Q_0^\alpha}^{Q_0}
Z_{f_{1,Q,\delta}}(s)Z_{f_{2,Q (1+\delta),X/Q}}(s) \frac{dQ}{Q}\right|^2 ds
\end{aligned}\]
con $B = [1+i\exp(\sqrt{\log X}),1+iX/h]\setminus
\mathscr{T}_{\lambda,Q_0,1/9,\alpha,\delta}$.
Como $\rho<1$, el t\'ermino de error
$\exp\left(-(\log P_0)^{3/5+o(1)}\right)$ en (\ref{eq:tamanerr}) es
absorbido por $O_\rho\left(1/(\log Q_0)^\rho\right)$.
Por Cauchy-Schwarz, de manera similar a (\ref{eq:frlesp}),
\begin{equation}\label{eq:dolesp}\int_{B}
\left|\int_{Q_0^\alpha}^{Q_0}
Z_{f_{1,Q,\delta}}(s)Z_{f_{2,Q (1+\delta),X/Q}}(s) \frac{dQ}{Q}\right|^2 ds
\leq (\log Q_0^{1-\alpha})^2 \max_Q
\int_B \left|Z_{f_{1,Q,\delta}}(s)Z_{f_{2,Q (1+\delta),X/Q}}(s)\right|^2 ds,\end{equation}
donde el m\'aximo de $Q$ se toma en el intervalo $\lbrack Q_0^\alpha,Q_0\rbrack$.
(Despu\'es de aplicar Cauchy-Schwarz, invertimos el orden de las integrales,
y luego reemplazamos la integral ahora exterior por un m\'aximo.)

Ahora sacamos el m\'aximo de $Z_{\lambda_{1,Q}}$ en $B$,
aplicando la Proposici\'on \ref{pr:ptgrande} con $a=
1-1/((3/2) \beta (1-\rho))$
(N\'otese que $\beta (1-\rho) \cdot (3/2) \cdot (1-a)  = 1$, por lo cual
$\exp((\log Q)^{(3/2) (1-a)}) \geq X$). Obtenemos que
\[ 
 \int_B |Z_{\lambda_{1, Q, \delta}} Z_{\lambda_{2,(1+\delta) Q,X/Q}}|^2 ds\ll \exp(-2 (\log X)^{a}) \int_B | Z_{\lambda_{2,Q (1+\delta),X/Q}}|^2 ds.
 \]
 
 Finalmente, tenemos que $B=\cup_{Q'} B_{Q'}$ con $Q'\in [Q_0^{\alpha},Q_0]\cap \mathbb Z$ y $B_{Q'}$ contenido en el conjunto de $t\in [0,X]$ tales que $|Z_{\lambda_{1,Q',\delta}}(1+it)|>Q'^{-1/9}$. 
 Por el Lema \ref{le:excepcionales} 
 tenemos que $|B_{Q'}|\ll X^{4/9}$, lo cual implica que $|B|\ll Q_0 X^{4/9}= X^{4/9+o(1)}$. Luego,
por Hal\'asz-Montgomery (Proposici\'on \ref{prop:halaszmonty}), conclu\'imos
que
\[
\int_B | Z_{\lambda_{2,Q (1+\delta),X/Q}}|^2 ds\ll \left(\frac{X}{Q}+X^{4/9+o(1)}
\sqrt{X} \log X\right)\cdot \sum_{n\leq X/Q} (Q/X)^2 \ll 1,
\]
lo que demuestra el teorema.

\end{proof}

\subsubsection{Ejercicios}
\begin{enumerate}

\item Sea $|f|\le 1$, $X\ge 1$ y $\epsilon>0$.
  \begin{enumerate}
    \item
  Use el Teorema del valor medio para demostrar que la medida del conjunto de
  $t\in [0,X]$ para los que $|\sum_{X\le p\le 2X} f(p)p^{-it}|>X^{1-\epsilon}$ es
  $O(X^{2\epsilon})$.
  \item
    Use el Teorema del valor medio, junto con una elevaci\'on a una
    potencia (como en la prueba del Lema \ref{le:excepcionales})
    para demostrar que, para todo $k\geq 1$, la medida del conjunto de
  $t\in [0,X^k]$ para los que $|\sum_{X\le p\le 2X} f(p)p^{-it}|>X^{1-\epsilon}$ es
  $O_k(X^{2 k \epsilon})$.
\end{enumerate}
\end{enumerate}

\subsection{El caso general: valores t\'ipicos}

Es imposible controlar una integral 
\[
 \int_{\mathscr T_{\lambda,Q_0,\gamma,\alpha,\delta}}   |Z_{\lambda 1_{[X,2X]}}(s)|^2 ds 
\]
factoriz\'andola con un factor del tama\~no $Q_0$, ya que $Q_0$ es mucho mayor que $h$. (Recu\'erdese que $\mathscr T_{\lambda,Q_0,\gamma,\alpha,\delta}$ es el
conjunto de los valores ``t\'ipicos'' dentro del intervalo $\lbrack 0,X/h\rbrack$.)
La idea crucial ahora es usar que es posible sustituir dicha integral por otra del tipo
\[
 \int_{\mathscr T_{\lambda,q_0,\gamma-\epsilon,\alpha,\delta}}   |Z_{\lambda 1_{[X,2X]}}(s)|^2 ds, 
\]
con $q_0$ sustancialmente m\'as peque\~no que $Q_0$, si pagamos aumentado el tama\~no que permitimos para el factor zeta correspondiente. Dicho resultado est\'a basado en el siguiente lema, que de nuevo usa el teorema del valor medio.

\begin{lema}\label{le:momentos}
  Sea $0<\varepsilon<\gamma<\frac 12$,
  $(\log Q)^{4/\varepsilon}<q<Q^{\varepsilon/4}<X^{\frac{1}{(\log\log X)^3}}$.
  Sean $f_q, f_Q, g_{X/Q} : \mathbb{Z}\to \mathbb{C}$ funciones con
  $|f_q(n)|, |f_Q(n)|, |g_{X/Q}(n)|\le 1$ para todo $n$. Asumamos que
  $g_{X/Q}$, $f_Q$ y $f_q$ tienen
  soporte en $[X/Q,2X/Q]$, en $[Q,2 Q]$ y en
  los primos de $[q,2q]$, respectivamente. Entonces
\[
 \int\limits_{\substack{s\in [1,1+iX] \\ |Z_{f_Q}(s)|\ll Q^{-\gamma}  \\|Z_{f_q}(s)|\geq q^{-\gamma+\varepsilon} }}  |Z_{f_Q}(s)Z_{g_{X/Q}}(s)|^2 \, ds \ll Q^{-\varepsilon}.
\]
\end{lema}
\begin{proof}
Podemos acotar la integral por
\[
\ll Q^{-2\gamma} \int\limits_{\substack{s\in [1,1+iX] \\|Z_{f_q}(s)|\geq q^{-\gamma+\varepsilon} }}  |Z_{g_{X/Q}}(s)|^2 \, ds \ll Q^{-2\gamma} \int_{1}^{1+iX}
\left|\left(\frac{Z_{f_q}(s)}{q^{-\gamma+\varepsilon}}\right)^{\ell}\right|^2
|Z_{g_{X/Q}}(s)|^2 \, ds.
\]
Tomando $\ell$ natural tal que $q^{\ell}<Q\le q^{\ell+1}$, como $(q^{\gamma-\varepsilon})^{2\ell}Q^{-2\gamma}\le Q^{2\gamma-2\varepsilon}Q^{-2\gamma}=Q^{-2\varepsilon}$,
vemos que la integral del enunciado es 
\[
\ll  Q^{-2\varepsilon} \int_1^{1+iX} |Z_{f_q}^{\ell} (s) Z_{g_{X/Q}}(s)|^2 \, ds. 
\]
Ahora se trata de acotar la integral por el teorema del valor medio (Lema \ref{lem:valmed}). Tenemos
\[
 \int_1^{1+iX} |Z_{f_q}^{\ell} (s) Z_{g_{X/Q}}(s)|^2 \, ds\ll 2^{\ell}X\sum_{X/q\le n\le 2^{\ell+1}X} \frac{a_n^2}{n^2}
\]
con 
\[
 a_n=\sum_{\substack{p_1p_2\ldots p_{\ell} m=n \\ q\le p_j\le 2q \\  X/Q\le m\le 2X/Q}} 1 \le \ell! \sum_{\substack{rm=n \\ p\mid r \Rightarrow q\le p\le 2q}} 1 =\ell! \, w(n)
\]
con $w(n)$ la funci\'on totalmente multiplicativa con $w(p^k)=1$ si $p\not\in [q,2q]$ y $w(p^k)=k+1$ si $p\in [q,2q]$. Como $Z_{w^2}(s)=\zeta(s) \prod_{q\le p\le 2q}\frac{1+4p^{-s}+9p^{-2s}+\ldots}{1+p^{-s}+p^{-2s}+\ldots}$, luego con residuo en $s=1$ de la forma $\prod_{q\le p\le 2q}(1+O(1/p))\ll 1,$ podemos proceder como en la demostraci\'on del Lema \ref{le:casiprimos} para demostrar que $\sum_{n<x}w^2(n)\ll x$ para $x\ge X/q$. Por lo tanto
\[
\int_1^{1+iX} |Z_{f_q}^{\ell} (s) Z_{g_{X/Q}}(s)|^2 \, ds\ll
q^2 2^{2\ell} (\ell!)^2 \ll q^2 \ell^{2\ell}.
\]
Como $q<Q^{\varepsilon /4}$ y $\ell^{2\ell}\ll (\log Q)^{2\ell}\ll (q^{\varepsilon/4})^{2\ell}\ll Q^{\varepsilon/2}$, hemos acabado.
\end{proof}

Ahora veamos que efectivamente podemos sustituir la integral de $Z_f$ sobre
$\mathscr T_{f,Q_0,c,\alpha,\delta}$ (``valores t\'ipicos con respecto a $Q_0$'') por la integral
de $Z_f$ sobre $\mathscr T_{f,q_0,c,\alpha,\delta}$ (``valores t\'ipicos con respecto a
$q_0$''), donde $q_0$ es mucho m\'as peque\~no que $Q_0$.

\begin{proposicion}\label{pr:iteracion}
  Sea $f$ totalmente multiplicativa con $|f(n)|\le 1$ para todo $n$. Sea
  $\mathscr T_{f,A,\gamma,\alpha,\delta}$ definido como en el Teorema
  \ref{te:excepcionales}. Sea $Q_0\leq \exp((\log X)^{1-\epsilon})$,
  $\epsilon\in (0,1)$. Entonces,
  para $\gamma\in (1/\log \log Q_0,1/2)$, $0<\epsilon<1$, $0<\rho<1$,
  $\rho<\kappa<1-\rho$, as\'i como
  $\alpha \leq 1/(\log Q_0)^\rho$, $\alpha' = 1/(\log q_0)^{1-\epsilon}$,
  $1/(\log Q_0) \leq \delta \leq 1/(\log Q_0)^\rho$, $0<\delta'\leq 1$ y $\log q_0=(\log Q_0)^{\kappa}$, tenemos que
\[
\int_{\mathscr T_{f,Q_0,\gamma,\alpha,\delta}}   |Z_{f_X}(s)|^2 ds \le
\int_{\mathscr T_{f,q_0,\gamma',\alpha',\delta'}}   |Z_{f_X}(s)|^2 ds + O_{\rho,\kappa,\epsilon}\left(\frac{1}{(\log Q_0)^\rho}\right).
\]
con $f_X=f\cdot 1_{(X,2X]}$ y $\gamma'=\gamma- 1/\log \log Q_0$.
\end{proposicion}
Cualquier funci\'on del tipo $(\log Q_0)^{o(1)}$ podr\'ia usarse en vez
de $\log \log Q_0$.
\begin{proof}
  Podemos suponer que $\alpha = 1/(\log Q_0)^\rho$, ya que
  $\alpha \leq \alpha_0 = 1/(\log Q_0)^\rho$ implica
  $\mathscr T_{f,Q_0,\gamma,\alpha,\delta} \subset
  \mathscr T_{f,Q_0,\gamma,\alpha_0,\delta}$.
  
Para comenzar,
\[
\int_{\mathscr T_{f,Q_0,\gamma,\alpha,\delta}}   |Z_{f_X}(s)|^2 ds
\le \int\limits_{\mathscr T_{f,q_0,\gamma',\alpha',\delta'}}   |Z_{f_X}(s)|^2 ds+
 \int\limits_{\mathscr T_{f,Q_0,\gamma,\alpha,\delta}\setminus \mathscr T_{f,q_0,\gamma',\alpha',\delta'}}   |Z_{f_X}(s)|^2 ds.
\]
Apliquemos el Corolario \ref{cor:factorizacion} para factorizar $Z_{f_X}$,
usando (\ref{eq:tamanerr}) y la Proposici\'on \ref{pr:zeta_trivial}
para controlar $Z_{\err}$. Obtenemos, procediendo como en
(\ref{eq:frlesp}) o (\ref{eq:dolesp}),
\[\begin{aligned}
\int\limits_{\mathscr T_{f,Q_0,\gamma,\alpha,\delta}\setminus
  \mathscr T_{f,q_0,\gamma',\alpha',\delta'}}   |Z_{f_X}(s)|^2 ds
&\ll O_\rho\left(\frac{1}{(\log Q_0)^\rho}\right) \\ +
\frac{\left(\log Q_0^{1-\alpha}\right)^2}{\delta^2}
&\max_{Q\in \lbrack Q_0^\alpha,Q_0\rbrack}
\int\limits_{\mathscr T_{f,Q_0,\gamma,\alpha,\delta}\setminus \mathscr T_{f,q_0,\gamma',\alpha',\delta'}}
\left|Z_{f_{1,Q,\delta}}(s)Z_{f_{2,Q (1+\delta),X/Q}}(s)\right|^2 ds.\end{aligned}\]

Ahora bien, para todo $s\in \mathscr T_{f,Q_0,\gamma,\alpha,\delta}$,
tenemos, por definici\'on, que $|Z_{f_{1,Q,\delta}}(s)|<Q^{-\gamma}$ para todo
$Q\in \lbrack Q_0^\alpha,Q_0\rbrack\cap \mathbb Z$, luego $|Z_{f_1,Q,\delta}|\le \frac 2Q+Q^{-\gamma}\ll Q^{-\gamma}$ para todo $Q\in [Q_0^{\alpha},Q_0]$; m\'as a\'un, si
  $s\notin \mathscr T_{f,q_0,\gamma',\alpha,\delta}$, entonces
  $|Z_{f_{1,q,\delta}}(s)|\geq q^{-\gamma'}$ para alg\'un 
$q\in [q_0^{\alpha},q_0]\cap \mathbb Z$.
As\'i, vemos que
\[\begin{aligned}
\int\limits_{\mathscr T_{f,Q_0,\gamma,\alpha,\delta}\setminus \mathscr T_{f,q_0,\gamma',\alpha',\delta'}}
&\left|Z_{f_{1,Q,\delta}}(s)Z_{f_{2,Q (1+\delta),X/Q}}(s)\right|^2 ds\\ &\leq
2 q_0 \max_{q\in [q_0^\alpha,q_0]}
\int\limits_{\substack{ s\in [1,1+iX]\\ \mathscr |Z_{f_{1,Q,\delta}}|\ll Q^{-\gamma} \\ |Z_{f_{1,q,\delta'}}|\geq q^{-\gamma'} }}
|Z_{f_{1,Q,\delta}}(s)Z_{f_{2,Q (1+\delta),X/Q}}(s)|^2 ds,
\end{aligned}\]
donde recordamos que $\mathscr T_{f,Q_0,\gamma,\alpha,\delta} \subset
[1,1 + i X/h] \subset [1,1+i X]$.
Aplicando el Lema \ref{le:momentos}, tenemos la cota
\[
q_0 \int\limits_{\substack{ s\in [1,1+iX]\\ \mathscr |Z_{f_{1,Q,\delta}}|\ll Q^{-\gamma} \\ |Z_{f_{1,q,\delta'}}|\geq q^{-\gamma'} }}   |Z_{f_{1,Q}}(s)Z_{f_{2,Q,X/Q}}(s)|^2 ds \ll q_0 Q^{-\varepsilon} \leq 
q_0 (Q_0^{\alpha})^{-\varepsilon} < Q_0^{- 3 \alpha \varepsilon/4}
\]
para $\gamma' = \gamma-\varepsilon$ y cualquier
$\varepsilon\in (0,\gamma)$ tal que
$(\log Q_0)^{4/\varepsilon} < q_0^{\alpha}$ y $q_0<Q_0^{\varepsilon \alpha/4}$.
Para $\varepsilon = 1/\log \log Q_0$, estas desigualdades se cumplen
(debido a $\kappa +\rho < 1$), si asumimos que $Q_0$ es m\'as grande
que una constante $c$ que depende s\'olo de $\rho$ y $\kappa$.
Bajo las mismas condiciones, $Q_0^{-3\alpha \varepsilon/4}<1/(\log Q_0)^\rho$
(por un amplio margen).
\end{proof}

\begin{teorema}\label{te:tipicos}
  Sean $0< \epsilon<1/3$, $Q_0 = \exp((\log X)^{1-\epsilon})$ y
  $1\le h \le \exp((\log X)^{2/3-\epsilon})$.
  Entonces, para $\alpha=\delta=  1/(\log Q_0)^{\rho}$, $1/6<\rho<1/3$,
\[
\int_{\mathscr T_{\lambda,Q_0,\frac 19,\alpha,\delta}}   |Z_{\lambda \cdot 1_{(X,2X]}}(s)|^2 ds \ll_\epsilon
  \max\left(\frac{1}{(\log Q_0)^{\rho}},\frac{1}{(\log h)^{1-\epsilon}}\right).
\]
\end{teorema}
\begin{proof}
  Supongamos primero que $h\geq \exp(\sqrt{\log Q_0})$.
  Aplicamos la Proposici\'on \ref{pr:iteracion} una vez con
  $\rho$,
  $\gamma = 1/9$, $\delta' = 1/\log h$, $\alpha'= 1/(\log h)^{1-\epsilon}$ y
  $\kappa = (\log \log h)/(\log \log Q_0)$, de tal manera que
  $q_0 = h$.
  Es f\'acil ver que $1/2\le \kappa \leq (2/3-\epsilon)/(1-\epsilon) <2/3$,
  as\'i que $\rho<\kappa<1-\rho$. Obtenemos
  \[\int_{\mathscr T_{\lambda,Q_0,\frac 19,\alpha,\delta}}   |Z_{\lambda 1_{[X,2X]}}(s)|^2 ds\leq
  O_{\epsilon}\left(\frac{1}{(\log Q_0)^{\rho}}\right)+
  \int_{\mathscr T_{\lambda,h,\frac 1{18},\alpha',\delta'}} |Z_{\lambda 1_{[X,2X]}}(s)|^2 ds,\]
  puesto que podemos suponer que $\gamma - 1/\log \log Q_0 \geq 1/18$.

  Ahora utilizamos el Corolario \ref{cor:factorizacion}
  para factorizar $Z_{\lambda 1_{\lbrack X, 2 X\rbrack}}$ con $h$ en vez de $Q_0$, $\alpha=(\log h)^{-(1-\epsilon)}$. Por Cauchy-Schwarz, la Proposici\'on \ref{pr:maximo_fuera} y la definici\'on de
  $\mathscr T_{\lambda,h,\frac 1{18},\alpha',\delta'}$, conclu\'imos que
\[
\int_{\mathscr T_{\lambda,h,\frac 1{18},\alpha',\delta'}
}   |Z_{\lambda 1_{[X,2X]}}(s)|^2 ds\ll
O_\epsilon\left(\frac{1}{(\log h)^{1-\epsilon}}\right)
+ (\log h)^{4}
\left(h^{\alpha}\right)^{-\frac{2}{18}} =
O_\epsilon\left(\frac{1}{(\log h)^{1-\epsilon}}\right).
\]

  Supongamos ahora que $h< \exp(\sqrt{\log Q_0})$. Procederemos exactamente
  como antes, excepto que iteraremos el uso de
  la Proposici\'on \ref{pr:iteracion}. Definimos
  $Q_{i+1} = \exp(\sqrt{\log Q_i})$ para $i=0,1,\dotsc$,
  hasta que llegamos a un $i$
  tal que $(\log h)^2<\log Q_i\le  (\log h)^4$. Entonces definimos $m = i+1$,
  $Q_m = h$.

  Comenzamos aplicando la Proposici\'on \ref{pr:iteracion}
 con  $\rho$,
  $\gamma = 1/9$, $\delta' = \delta_1 = 1/\log Q_1$,
  $\alpha'= \alpha_1 = 1/(\log Q_1)^{1-\epsilon}$,
  $\kappa = (\log \log Q_1)/(\log \log Q_0) = 1/2$ y los valores
  iniciales de $\alpha$ y $\delta$. Obtenemos
\[\int_{\mathscr T_{\lambda,Q_0,\frac 19,\alpha,\delta}}   |Z_{\lambda 1_{[X,2X]}}(s)|^2 ds\leq
  O_{\epsilon}\left(\frac{1}{(\log Q_0)^{\rho}}\right)+
  \int_{\mathscr T_{\lambda,Q_1,\gamma_1,\alpha_1,\delta_1}} |Z_{\lambda 1_{[X,2X]}}(s)|^2 ds,\]
  donde $\gamma_1 = \gamma - 1/\log \log Q_0$. 
  Luego aplicamos la 
  Proposici\'on \ref{pr:iteracion} para $i=1,2,\dotsc,m-1$,
  con $\rho = 1/2-\epsilon$, $\gamma = \gamma_i$,
  $\delta = \delta_i$, $\alpha = \alpha_i$, 
  $\delta' = \delta_{i+1} = 1/\log Q_{i+1}$, $\alpha' = \alpha_{i+1} =
  1/(\log Q_{i+1})^{1-\epsilon}$ y $\kappa = (\log \log Q_{i+1})/\log \log Q_i
  = 1/2$. As\'i,
\[\int_{\mathscr T_{\lambda,Q_i,\gamma_i,\alpha_i,\delta_i}}   |Z_{\lambda 1_{[X,2X]}}(s)|^2 ds\leq
  O_{\epsilon}\left(\frac{1}{(\log Q_i)^{1/2}}\right)+
  \int_{\mathscr T_{\lambda,Q_{i+1},\gamma_{i+1},\alpha_{i+1},\delta_{i+1}}} |Z_{\lambda 1_{[X,2X]}}(s)|^2 ds,\]
  donde $\gamma_{i+1} = \gamma_i - 1/\log \log Q_i$. Est\'a claro que
  $1/(\log Q_i)^{1/2} \leq 1/(\log Q_{i+1})$. Verificamos que
\[
\sum_{i<m} \frac{1}{\log \log Q_i} \le
\frac{1}{\log \log Q_m} \sum_{k\geq 0} \frac{1}{2^k} =
\frac{2}{\log \log h} <\frac 1{18},
\]
ya que podemos asumir que $h$ es m\'as grande que una constante.
Por lo tanto, $\gamma_m\geq 1/18$.

Terminamos usando el Corolario \ref{cor:factorizacion}
exactamente como antes ($h$ en vez de $Q_0$, $\alpha=(\log h)^{-(1-\epsilon)}$). Por Cauchy-Schwarz y la Proposici\'on \ref{pr:maximo_fuera}, conclu\'imos que
\[\begin{aligned}
\int_{\mathscr T_{\lambda,Q_m,\gamma_m,\alpha_m,\delta_m}} |Z_{\lambda 1_{[X,2X]}}(s)|^2 ds&\ll
O_\epsilon\left(\frac{1}{(\log h)^{1-\epsilon}}\right) + (\log h)^{4}
\left(\exp\left((\log h)^\epsilon\right)\right)^{-\frac{2}{18}}\\ &=
O_\epsilon\left(\frac{1}{(\log h)^{1-\epsilon}}\right).\end{aligned}\]  
%
%
%
%
\end{proof}

Ahora ya podemos concluir nuestro resultado principal.

\begin{teorema}\label{te:result_principal}
Sea $1\le h\leq X$. Entonces, para $\epsilon>0$ arbitrariamente peque\~no,
\begin{equation}\label{eq:cota_var_princ}
 \mathbb E_{X<x\leq 2X} |\mathbb E_{(1-\frac{h}{X})x<n\leq x} \lambda(n)|^2
\ll_\epsilon \max\left(\frac{1}{(\log X)^{\frac{1-\epsilon}{3}}},\frac{1}{(\log h)^{1-\epsilon}}\right).
\end{equation}
\end{teorema}
\begin{proof}
  Podemos asumir que $h\leq \exp((\log X)^{2/3-\epsilon})$ (como en el
  Ejercicio \ref{ej:caso_h_grande}).
Aplicando la Proposici\'on \ref{pr:parseval} tenemos la cota
\[
 \mathbb E_{X<x\leq 2X} |\mathbb E_{(1-\frac{h}{X})x<n\leq x} \lambda(n)|^2
\ll \int_{1}^{1+i\frac{X}{h\delta^2}}  |Z_{\lambda 1_{[X,2X]}}(s)|^2 ds +O(\delta).
\]
Tomemos $\delta = h^{-1/4}$ (digamos).
El resultado (con un m\'ultiplo constante de $\epsilon$ en vez de $\epsilon$,
lo cual
es claramente inofensivo) se sigue del Teorema \ref{te:excepcionales} y el
Teorema \ref{te:tipicos}, aplicados con $h\delta^2$ en vez de $h$,
$\alpha = \delta = 1/(\log Q_0)^{\rho}$, $Q_0=\exp((\log X)^{1-\epsilon})$ y $\rho = 1 - \frac{2}{3 (1-2\epsilon)}$.
\end{proof}
Es posible obtener una cota 
con $(\log \log h)/\log h$ en vez de $1/(\log h)^{1-\epsilon}$
\cite{MR3488742}. Hemos optado por nuestra cota por simplicidad.

El Teorema \ref{te:result_principal} implica inmediatamente el Teorema
\ref{te:intervalos_cortos_m}, por un breve y conocido argumento (Chebyshev /
cota de varianza) que ahora mostraremos.
Escribamos la cota en (\ref{eq:cota_var_princ})
de la manera $\leq c_\epsilon/\max(\log h,(\log X)^{1/3})^{1-\epsilon}$.
Entonces, para todo $C$, la proporci\'on de valores de $X<x\leq 2X$ tales que
\[\left|\mathbb{E}_{\left(1-\frac{h}{X}\right) x <n\leq x} \lambda(n)\right|\geq
\frac{C \sqrt{c_\epsilon}}{\max(\log h,(\log X)^{1/3})^{(1-\epsilon)/2}}\]
es a lo m\'as $1/C^2$.

Si, alternativamente, sigui\'eramos \cite{MR3488742}, podr\'iamos mostrar
que, para $h$ en cierto rango, una cota m\'as fuerte vale: la proporci\'on de $x$
tales que
\begin{equation}\label{eq:excepc}
  |\mathbb{E}_{(1-h/X) x <n\leq x} \lambda(n)|\geq C/(\log h)^{1-\epsilon}
\end{equation} es
peque\~na. La diferencia
principal consiste en que \cite{MR3488742} separa desde un principio el
conjunto $\mathcal{S}$ de
enteros $X<n\leq 2 X$
con factores primos del tama\~no deseado, y prueba una cota de varianza
para ellos; luego reintroduce los enteros que quedaron fuera de $\mathcal{S}$
al \'ultimo momento. De esta manera consigue mostrar que pocos $x$ satisfacen (\ref{eq:excepc}), a\'un si la forma de la cota de varianza es en esencia la misma.

\subsubsection{Ejercicios}

\begin{enumerate}

\item Sea $2\le h=o(X)$, $h\in\mathbb N$, $h\to\infty$.
\begin{enumerate}
\item 
  Sea $f(n)=\cos(\frac{2\pi n}{100h})$.
  \begin{enumerate}
    \item  Demuestre que, si bien
  oscila en intervalos largos (es decir, $\sum_{X<n\le 2X} f(n)=o(X)$), no
  oscila en intervalos cortos:
  \[|\sum_{x<n\le x+2 h} f(n)|\gg h\] para todo $x\in [X,2X]$. Usando la Proposici\'on \ref{pr:parseval},
  deduzca que 
 \[
  \int\limits_{0}^{cX/h} |\sum_{X<n\le 2X} f(n) n^{-1-it} |^2 \, dt \gg 1
 \]
 para alguna constante $c>1$.
\item Muestre que
  \[\sum_{X<n\leq u} f(n) \leq 100 h\]
  para todo $u$. (Consejo: la funci\'on $f$ es peri\'odica, con per\'iodo
  $100$.) Deduzca, por sumaci\'on por partes, que
  \[\sum_{X<n\le 2X} f(n) n^{-1-it} \ll t h/X\]
  para todo $t$.
\item Deduzca que, para $h$ tal que $(\log h)^{150} \leq X/h$,
  \[
  \int\limits_{(\log h)^{100}}^{cX/h} |\sum_{X<n\le 2X} f(n) n^{-1-it} |^2 \, dt \gg 1
 \]
para alguna constante $c>1$.
\end{enumerate}
\item Sea $S$ el conjunto de n\'umeros que son producto de dos primos, uno en $[P,2P]$ y otro en $[X/P,2X/P]$. Sea $g(n)$ la funci\'on multiplicativa tal que $g(p)=\cos(\frac{2\pi p}{100h})$. Asumiendo que $h=X^{1/10}$ y $P=\sqrt{h}$, demuestre que
\[
 \int\limits_{(\log h)^{100}}^{cX/h} |\sum_{n\in S} g(n) n^{-1-it} |^2 \, dt \ll_c \frac{1}{(\log X)^{20}}.
\]
para cualquier constante $c>1$. Sugerencia: use el Teorema del valor medio (Lema \ref{lem:valmed}) y la Proposici\'on \ref{pr:ptgrande} (o m\'as
  bien su prueba, pues consideramos valores de $t$ por debajo de lo que la
  Proposici\'on \ref{pr:ptgrande} cubre).

\end{enumerate}

\item Sea $4\le h\le (\frac{X}2)^{1/5}$. Sea $S$ el conjunto de n\'umeros que son producto de tres primos, uno en $[h,2h]$, otro en $[h^2,2h^2]$ y otro en $[X/h^3,2X/h^3]$. Sea $f(n)$ una funci\'on multiplicativa tal que $Z(t)=\sum_{q\in [h^2,2h^2]} f(q)q^{-1-it}\ll (h^2)^{-1/9}$ para todo $t\in[(\log h)^{100},X/h]$.
\begin{enumerate}
\item Muestre que
\[
 \sum_{n\in S} f(n) n^{-1-it}=\sum_{p\in [h,2h]} f(p)p^{-1-it}\sum_{q\in [h^2,2h^2]} f(q) q^{-1-it} \sum_{r\in [X/h^3,2 X/h^3]} f(r) r^{-1-it}
\]
con $p,q,r$ primos.

\item Use la cota para $Z(t)$ y el Teorema del valor medio (Lema \ref{lem:valmed}) para probar que 
\[
  \int\limits_{\substack{(\log h)^{100} \\ |\sum_{p\in [h,2h]} f(p) p^{-1-it}|>h^{-1/18}}}^{X/h} 
  |\sum_{n\in S} f(n) n^{-1-it} |^2 \, dt \ll \frac{1}{(h^{2/9})^2}\cdot
  \frac{1}{(h^{-1/18})^4} =\frac{1}{h^{2/9}} .
\]
\item Use el Teorema del valor medio nuevamente para conclu\'ir que
\[
 \int\limits_{(\log h)^{100}}^{X/h} |\sum_{n\in S} f(n) n^{-1-it} |^2 \, dt \ll \frac{1}{h^{2/9}}.
\]
\end{enumerate}

\end{enumerate}

\section{Coeficientes de Fourier de $\lambda$ en intervalos cortos, en promedio}\label{sec:fourcort}

Nuestra tarea en esta secci\'on es acotar promedios de
coeficientes de Fourier de $\lambda\cdot 1_{(x,x+h\rbrack}$:
\[\int_0^X \left|\sum_{x<n\leq x + h} \lambda(n) e(\alpha n)\right| d x.\]
Al final de la secci\'on, veremos como tales cotas implica que el promedio
de expresiones como $\lambda(n) \lambda(n+c)$ tiende a $0$ para una proporci\'on
que tiende a $1$ de todos los $c$ en un peque\~no intervalo 
(``Chowla en promedio''). Las l\'ineas generales del argumento
siguen las que se segu\'ian para el mismo problema
cuando $c$ var\'ia en un intervalo grande, as\'i como para problemas an\'alogos. (El art\'iculo
\cite{MR3435814} menciona las fuentes \cite{MR675168}, \cite{MR836415},
\cite{MR0457371} y
\cite{MR2986954}.)

El tratamiento ser\'a distinto dependiendo de si $\alpha\in \mathbb{R}$
est\'a en un {\em arco mayor} o en un {\em arco menor}. (En verdad la definici\'on depende s\'olo de
$\alpha \mo \mathbb{Z}$.) Tomemos $Q = \sqrt{h}$.
Por el teorema de aproximaci\'on de Dirichlet (Ejercicio \ref{ej:dirichlet}),
existen $a,q\in \mathbb{Z}^+$, $1\leq q\leq Q$, $(a,q)=1$, tales que
$|\alpha-a/q|\leq 1/q Q$. Si tales $a$, $q$ existen con $q\leq R$ para cierto $R$ peque\~no (ser\'a una potencia de $\log h$), decimos que $\alpha$ est\'a en un
{\em arco mayor}; de lo contrario, existen tales $a$, $q$ con $q>R$, y decimos que $\alpha$
est\'a en un {\em arco menor}.

\subsubsection{Ejercicios}\label{sec:ejercicios_4_inicial}
\begin{enumerate}
\item\label{ej:dirichlet} (Teorema de aproximaci\'on de Dirichlet)
  \begin{enumerate}
    \item Sean dados $\alpha\in \mathbb{R}$
  y $M\in \mathbb{Z}^+$. Muestre que existen $0\leq m_1<m_2\leq M$ tales
  que $\alpha m_1$ y $\alpha m_2$ est\'an a distancia
  $\leq 1/(M+1)$ en  $\mathbb{R}/\mathbb{Z}$ el uno del otro.
  Deduzca que $|(m_2-m_1) \alpha - a|\leq  1/(M+1)$ para alg\'un
  $a\in \mathbb{Z}^+$.
\item Sean dados $\alpha\in \mathbb{R}$ y un real
  $Q\geq 1$. Concluya que existen $a,q\in \mathbb{Z}^+$, $1\leq q\leq Q$,
  $(a,q)=1$, tales que
  \[\left|\alpha - \frac{a}{q}\right|\leq \frac{1}{q Q}.\]
  (Aplique la primera parte del ejercicio con $M = \lfloor Q\rfloor$.)
  \end{enumerate}
\item\label{ej:goldbach} En este problema vamos a discutir, a vuelo de p\'ajaro,
  una variante relativamente
  sencilla del problema ternario de Goldbach, para ilustrar el uso de
  sumas exponenciales $\sum_n \lambda(n) e(\alpha n)$ similares a las que estudiaremos en esta
  secci\'on (si bien en verdad m\'as antiguas).
  
    Para $N\in\mathbb{Z}^+$, consideraremos la suma
\[
 S_N=\sum_{\substack{a,b,c\ge 1 \\ a+b+c=N}} \lambda(a)\lambda(b)\lambda(c).
\]
Si cambi\'aramos $\lambda$ por la funci\'on indicatriz de los primos estar\'iamos calculando el n\'umero de maneras de escribir $N$ como suma de tres primos. Nosotros queremos demostrar que $S_N=o(N^2)$, lo que nos dir\'ia que entre las tripletas de n\'umeros que suman $N$ hay la misma probabilidad de tener un n\'umero par que uno impar de primos.
\begin{enumerate}
\item Demuestre la identidad $\int_0^1 e(j\alpha)\,d\alpha=1_{j=0}$ para $j\in \mathbb Z$, con $e(x)=e^{2\pi i x}$. (Es la identidad sobre
  la cual se basa el an\'alisis de Fourier sobre $\mathbb{R}/\mathbb{Z}$.)
 \item Usando el apartado anterior, demuestre que 
 \[
  S_N=\int_0^1 (A_{\lambda}(\alpha))^3 e(-N\alpha) \, d\alpha,
 \]
con $A_{\lambda}(\alpha)=\sum_{1\le n\le N} \lambda(n) e(n\alpha)$.
\item\label{ej:goldbach_c} Es posible demostrar que $\max_{\alpha\in [0,1]} |A_{\lambda}(\alpha)|=o(N)$.
  (\'Esta es, claro est\'a, la parte dif\'icil, la cual omitimos.)
  Usando esta cota, demuestre que
\[
 S_N\le o(N) \int_0^1 |A_{\lambda}(\alpha)|^2 d\alpha= o(N)N=o(N^2).
\]
\item Sea $T_N=\sum_{a,b\ge 1: a+b=N} \lambda(a)\lambda(b)$ la suma correspondiente a Goldbach con dos primos. Observe que si intentamos demostrar que $T_N=o(N)$ con las t\'ecnicas de los apartados anteriores, no funciona. ?`Por qu\'e? Esto es parecido a lo que ocurre en la integral de la Proposici\'on \ref{pr:parseval}, y por eso all\'i fue necesario factorizarla.
\end{enumerate}
  
\end{enumerate}  

\subsection{Arcos menores}

\begin{proposicion}\label{pr:arcos_menores}
  Sean $X\geq 1$, $1\leq h\leq X$,
  $1\leq R\leq \log h$.
  Sea $\alpha\in \mathbb{R}$ tal que existen $a,q\in \mathbb{Z}^+$, $R\leq q\leq \sqrt h$, $(a,q)=1$, tales que
  $|\alpha-\frac aq|\leq \frac 1{q\sqrt h}$.
  Entonces
  \begin{equation}\label{eq:arcmincot}
    \frac{1}{h X} \int_X^{2 X} \left|\sum_{x<n\leq x + h} \lambda(n) e(\alpha n)\right| d x \ll \frac{(\log R)^3}{R^{1/4}}
    .\end{equation}
\end{proposicion}
\begin{proof}
  Probar (\ref{eq:arcmincot}) es equivalente a probar que, para todo
  $\theta:\mathbb{R}\to \mathbb{C}$ medible, con soporte en $\lbrack X,2X\rbrack$,
  y tal que $|\theta(x)|\leq 1$ para todo $x$,
  \begin{equation}\label{eq:arcmincot2}
    \int_{\mathbb{R}} \theta(x) \sum_{x<n\leq x + h} \lambda(n) e(\alpha n) d x \ll
   h X \frac{(\log R)^3}{R^{1/4}}.
  \end{equation}

  Sean $\beta = 1/2$, $\delta=R^{-1/4}$, $Q_0 = \exp(R^{1/4})$ y $P_0 = \exp((\log R)^3)$.
  Por el Lema \ref{le:factoriza}, el lado izquierdo de (\ref{eq:arcmincot2}) es igual a
  \begin{equation}\label{eq:convolv}
   O\left(\frac{(\log R)^3}{R^ {1/4}}\right) \cdot h X + \int_{\mathbb{R}} \theta(x) \sum_{x<n\leq x + h} \lambda(n) u_X(n) e(\alpha n) d x,
  \end{equation}
  donde $u_X$ es como en (\ref{eq:uXdef})
  (el error proviene de que cada valor de $n$ satisface $x<n\leq x+h$ s\'olo para $x$ dentro de un intervalo de largo $\leq h$).
  Entonces, por la definici\'on (\ref{eq:uXdef}) de $u_X$, es suficiente demostrar que
  \begin{equation}\label{eq:intogra}
   \frac{1}{\log(1+\delta)} \int_{P_0}^{Q_0} I(X,Q,\delta) \frac{d Q}{Q}\ll \frac{(\log R)^3}{R^ {1/4}} \cdot h X,
    \end{equation}
  donde
  \[\begin{aligned}
  I(X,Q,\delta) &=
  \int_{\mathbb{R}} \theta(x) \sum_{x<n\leq x + h} \lambda(n)
  \left(\upsilon_{1,Q,\delta}\ast \upsilon_{2,Q (1+\delta),X/Q}\right)(n) e(\alpha n) d x\\
  &= - \sum_m \upsilon_{2,Q (1+\delta),X/Q}(m) \lambda(m) \sum_{Q<p\leq (1+\delta) Q} e(\alpha m p)
  \int_{\mathbb{R}} \theta(x) 1_{(x,x+h\rbrack}(m p) d x,
  \end{aligned}\]
   y $\upsilon_{1,Q,\delta}$, $\upsilon_{2,r,Y}$ est\'an definidos como en (\ref{eq:bedrich}) y (\ref{eq:smetana}).
  Por Cauchy-Schwarz,
  \begin{equation}\label{eq:Icota}
  I(X,Q,\delta)^2\ll \frac{X}{Q}
  \sum_{X/Q<m\leq 2 X/Q} \left|
  \sum_{Q<p\leq (1+\delta) Q} e(\alpha m p) \int_{\mathbb{R}} \theta(x) 1_{(x,x+h\rbrack}(m p) d x \right|^2,
  \end{equation}
  ya que $\upsilon_{2,Q,(1+\delta),X/Q}$ tiene soporte en
  $[X/Q,2 X/Q]$ y 
  $|\upsilon_{2,Q,(1+\delta),X/Q}(m)|\le 1$ para todo $m$.

Ahora expandimos el cuadrado, y cambiamos el orden de la suma:\footnote{A partir de este momento,
  algunos lectores reconocer\'an
  ciertos paralelos con la versi\'on de Linnik de la prueba del trabajo
  de Vinogradov sobre los tres primos.}
\[\begin{aligned}
I(X,Q,\delta)^2 \ll \frac{X}{Q} 
\sum_{Q<p_1,p_2\leq (1+\delta) Q}
&\int_{\mathbb{R}} \int_{\mathbb{R}} 
\theta(x_1) \overline{\theta(x_2)} 
\mathop{\sum_{X/Q<m\leq 2X/Q}}_{\frac{x_i}{p_i} < m \leq \frac{x_i+h}{p_i}}
e(\alpha (p_1-p_2) m) d x_1 d x_2.
\end{aligned}\]
La suma sobre $m$ es una suma sobre un intervalo $I$ contenido en $(x_1/p_1,x_1/p_1+h/p_1\rbrack
\subset (x_1/p_1,x_1/p_1+h/Q)$.
Puede ser acotada trivialmente por $\leq |I|+1\leq h/Q + 1$;
tambi\'en vemos que se anula a menos que
$(x_1/p_1,(x_1+h)/p_1]\cap (x_2/p_2,(x_2+h)/p_2]\neq \emptyset$, lo cual implica
\begin{equation}\label{eq:xintervo}
  \frac{p_2}{p_1} x_1 - h < x_2 < \frac{p_2}{p_1} x_1 + (1+\delta) h.\end{equation}
Tambi\'en sabemos que, para cualquier $\beta\in \mathbb{R}$ no entero y cualquier intervalo
$I$ escrito en la forma $\lbrack m_0,m_1\rbrack$, $m_0,m_1\in \mathbb{Z}$,
\[\sum_{m\in I} e(\beta m) = \frac{e(\beta (m_1+1)) - e(\beta m_0)}{e(\beta) - 1}
\;\;\;\;\;\;\;\;\;\text{
(suma geom\'etrica)}\]
y, por lo tanto,
\begin{equation}\label{eq:cota_geometrica}\left|\sum_{m\in I} e(\beta m)\right|\leq \frac{2}{|e(\beta)-1|} = \frac{1}{\sin \pi \beta} \leq
\frac{2/\pi}{d(\beta,\mathbb{Z})},\end{equation}
donde $d(\beta,\mathbb{Z})$ es la distancia entre $\beta$ y el entero m\'as cercano.
Por lo tanto,
\[\mathop{\sum_{X/Q<m\leq 2X/Q}}_{\frac{x_i}{p_i} < m \leq \frac{x_i+h}{p_i}}
e(\alpha (p_1-p_2) m) \ll \min\left(\frac{h}{Q}, \frac{1}{d(\alpha (p_1-p_2),\mathbb{Z})}\right),\]
y as\'i, por la condici\'on (\ref{eq:xintervo}),
\[I(X,Q,\delta)^2 \ll \frac{X}{Q} X h \sum_{Q<p_1,p_2\leq (1+\delta) Q}
\min\left(\frac{h}{Q}, \frac{1}{d(\alpha (p_1-p_2),\mathbb{Z})}\right).\]
Ahora bien,  para todo entero $l$,
el n\'umero de parejas $Q\le p_1,p_2\leq Q+\delta Q$ tales que $p_1-p_2=l$ es
$O(\delta Q)$. 
Vemos entonces que
\[I(Q,X,\delta)^2\ll \delta X^2 h 
\mathop{\sum_{l\in \mathbb{Z}}}_{|l|\leq \delta Q}
\min\left(\frac{h}{Q}, \frac{1}{d(\alpha l,\mathbb{Z})}\right).\]
Ahora bien, por un lema de Vinogradov (ejercicio \ref{ej:vinogradov}),
  \[I(Q,X,\delta)^2\ll \delta X^2 h \left(\frac{\delta Q}{q} + 1 \right) \left(\frac{h}{Q} + q \log q\right).\]
Como $R\le q\le \sqrt h$, $\delta = R^{-1/4}$ y $\frac{R}{\delta}\ll P_0\le Q\le Q_0=\exp(R^{1/4})\le h^{1/4}$, tenemos que $q\log q\ll \frac hQ$ y 
\[
I(Q,X,\delta)^2\ll \delta X^2 h \left(\frac{\delta h}{q}+\frac{h}{Q}\right)
\ll \frac{\delta^2 X^2 h^2}{R}.
\]
Por lo tanto 
\[
\frac{1}{\log(1+\delta)} \int_{P_0}^{Q_0} I(X,Q,\delta) \frac{d Q}{Q}\ll (\log Q_0)\, \frac{Xh}{R^{1/2}}=\frac{Xh}{R^{1/4}} \]
lo que demuestra (\ref{eq:intogra}).
\end{proof}

\subsubsection{Ejercicios}\label{sec:ejercicios_4_menores}
\begin{enumerate}
\item\label{ej:vinogradov}  (Lema de Vinogradov)
  \begin{enumerate}
  \item Para $a,q\in \mathbb{Z}^+$ coprimos y $n = n_0, n_0+1, \dotsc, n_0+q-1$, $n_0$ arbitrario, muestre que las fracciones $n/q \mo \mathbb{Z}$ no son sino \[0, 1/q, 2/q,\dotsc, (q-1)/q\]  en alg\'un orden.
  \item Sean $\alpha\in \mathbb{R}$, $a,q\in \mathbb{Z}^+$, $1\leq q\leq Q$,
    $(a,q)=1$, tales que $|\alpha-a/q|\leq 1/q Q$. Muestre que, para
    $n = n_0, n_0+1, \dotsc, n_0+q-1$, $n_0$ arbitrario, las distancias $d(n \alpha,\mathbb{Z})$
    son a lo m\'as
    \[0,0,\frac{1/2}{q}, \frac{1}{q},\frac{3/2}{q},\dotsc,\frac{(q-2)/2}{q},\]
    en alg\'un orden.
  \item Sean $\alpha\in \mathbb{R}$, $a,q\in \mathbb{Z}^+$, $1\leq q\leq Q$,
    $(a,q)=1$, tales que $|\alpha-a/q|\leq 1/q Q$. Deduzca que, para $n_0$ y $X$ arbitrarios,
    \[\sum_{n=n_0}^{n_0+q-1} \min\left(X, \frac{1}{d(n \alpha,\mathbb{Z})}\right)
    \leq 2 X + 2 q (\log(q-2) + 1) \ll X + q \log q.\]
    Concluya que, para $N$ arbitrario,
    \[\sum_{|n|\leq N}  \min\left(X, \frac{1}{d(n \alpha,\mathbb{Z})}\right)
    \ll \left(\frac{N}{q} + 1 \right) (X + q \log q).\]
  \end{enumerate}

\item\label{ej:goldbach_facil}
  En este problema vamos a reproducir el esquema descrito en el ejercicio \ref{ej:goldbach} de \S \ref{sec:ejercicios_4_inicial} para estimar la suma 
\[
 S_N=\sum_{p+q+b=N} \log p \log q
\]
que cuenta el n\'umero de maneras (ponderadas) de escribir $N$ como suma de dos primos y un n\'umero natural cualquiera. Esta es una versi\'on muy sencilla del problema ternario de Goldbach, y podr\'iamos estimarla sin introducir exponenciales $e(n\alpha)$, pero por una cuesti\'on did\'actica veamos que es posible hacerlo con ellas. 
\begin{enumerate}
 \item Procediendo como en el ejercicio \ref{ej:goldbach} de \S \ref{sec:ejercicios_4_inicial}, muestre que 
 \[
  S_N=\int_{-1/2}^{1/2} F(\alpha)G(\alpha)^2 e(-N\alpha)\, d\alpha
 \]
con $F(\alpha)=\sum_{b=1}^N e(b\alpha)$ y $G(\alpha)=\sum_{p\le N} \log p \; e(p\alpha)$.

\item En este caso los arcos menores van a ser la zona donde $F$ es peque\~na.
  Tomemos como arcos menores $\mathfrak m=[-1/2,1/2]\setminus (-\frac{1}{N\delta},\frac{1}{N\delta}), $ con $0<\delta<1/2$. Demuestre usando
  \eqref{eq:cota_geometrica} que si $\alpha \in \mathfrak m$ entonces 
\[
 |F(\alpha)|\ll \delta N.
\]

\item Use el apartado anterior, el razonamiento expuesto en el ejercicio \ref{ej:goldbach_c} de \S \ref{sec:ejercicios_4_inicial} y el TNP en la forma del Teorema \ref{te:TNP} para demostrar que
\[
 \int_{\mathfrak m} F(\alpha)G(\alpha)^2 e(-N\alpha) \, d\alpha \ll \delta N^2 \log N.
\]
En los ejercicios de la siguiente secci\'on veremos que la integral sobre los arcos mayores $\mathfrak M=(-\frac{1}{N\delta},\frac{1}{N\delta})$ da la contribuci\'on principal para $S_N$ (para cierto $\delta$).

\end{enumerate}

\end{enumerate}

\subsection{Arcos mayores y conclusi\'on}
Examinemos ahora el caso de $\alpha=a/q$ con $q$ un entero peque\~no.
La idea es entonces que $e(\alpha n)$ es una funci\'on $q$-peri\'odica. As\'i, lo que nos gustar\'ia es sustituir $\lambda(n)e(\alpha n)$ por $\lambda(n) f(n)$ con $f(n)$ alguna funci\'on multiplicativa, para intentar usar las t\'ecnicas de la secci\'on anterior. Esto es posible, usando an\'alisis de Fourier para funciones $F:G\to \mathbb C$, con $G=(\mathbb Z/q\mathbb Z)^{\times}$ el grupo (abeliano) de residuos m\'odulo $q$ coprimos con $q$, con la operaci\'on de multiplicaci\'on. Dicho an\'alisis permite expresar $F$ como suma de varios \emph{car\'acteres m\'odulo $q$}. Un car\'acter m\'odulo $q$
no es sino un homomorfismo $\chi_*:G \to \mathbb C^{\times}$.
Podemos extender tal homomorfismo
a una funci\'on multiplicativa y $q$-peri\'odica $\chi:\mathbb Z\to \mathbb C$, simplemente poniendo $\chi(n)=0$ para $n$ no coprimo con $q$.

Nosotros s\'olo vamos a usar que $|\chi_*(g)|=1$
y que los car\'acteres $\chi_*$ m\'odulo $q$ forman una base ortonormal $\widehat{G}$
del espacio
de funciones de $G=(\mathbb{Z}/q\mathbb{Z})^{\times}$ a $\mathbb{C}$. En verdad,
$\widehat{G}$ es tambi\'en es un grupo.
Todos estos hechos son ciertos en general para grupos abelianos finitos $G$. Pueden, por
cierto, ser probados f\'acilmente (ejercicio \ref{ej:pentimento}).


Si $\alpha$ est\'a cerca de un racional con denominador peque\~no, vamos a ver c\'omo pasar de sumas cortas de $\lambda(n)e(\alpha n)$ a sumas cortas de $\lambda(n)\chi(n)$ para alg\'un car\'acter $\chi$ de m\'odulo peque\~no.

\begin{proposicion}\label{pr:suma_a_multip}
Sea $1<R^4<h\le X$, $|\alpha-\frac aq|\le \frac{1}{q \sqrt h}$ con $q\le R$. Entonces
\[
  \frac{1}{h X} \int_X^{2 X} \left|\sum_{x<n\leq x + h} \lambda(n) e(\alpha n)\right| d x \ll \frac 1R + R \max_{\chi,h', X'} \frac{1}{h' X'} \int_{X'}^{3 X'} \left|\sum_{x<n\leq x + h'} \lambda(n) \chi(n) \right| d x
\]
donde el m\'aximo se toma sobre los car\'acteres $\chi$ de m\'odulo menor o igual que $R$ y sobre $X'\in [X/R, X]$, $h'\in [\sqrt h/R^2, \sqrt h/R]$.
\end{proposicion}
\begin{proof}
Comenzamos con la desigualdad
\[
\frac{1}{h X} \int_X^{2 X} \left|\sum_{x<n\leq x + h} f(n)\right| d x \le
O\left(\frac{h_1}{h}\right) +\frac{1}{h_1 X} \int_X^{3 X} \left|\sum_{x<n\leq x + h_1} f(n)\right| d x
\]
para $1\le h_1\le h$, $|f|\le 1$, que simplemente proviene de dividir la suma interior en sumas de longitud $h_1$. La usamos con $f(n)=\lambda(n)e(\alpha n)$ y $h_1= \sqrt h/R$. Ahora bien,
\[
\left|\sum_{x<n\le x+h_1} \lambda(n)e(\alpha n)\right|
= \left| \sum_{x<n\le x+h_1} \lambda(n) e\left(\frac{an}{q}\right)
e\left(\left(\alpha-\frac aq\right)(n-x)\right)\right|
\]
y como $|(\alpha-a/q)(n-x)|\le \frac{1}{q\sqrt h} h_1\le R^{-1}$ tenemos que
\[
|\sum_{x<n\le x+h_1} \lambda(n)e(\alpha n)| = \left|\sum_{x<n\le x+h_1} \lambda(n)
e\left(\frac aq n\right)\right|+O(R^{-1}) h_1.
\]
As\'i
\[
  \frac{1}{h X} \int_X^{2 X} \left|\sum_{x<n\leq x + h} \lambda(n) e(\alpha n)\right| d x \ll R^{-1} +\frac{1}{h_1 X} \int_X^{3 X} \left|\sum_{x<n\leq x + h_1} \lambda(n) e\left(\frac aq n\right)\right| d x.
  \]

  Como $n\mapsto e(a n/q)$ es $q$-peri\'odica, podemos escribir
  (ejercicio \ref{ej:habakuk})
  \[e\left(\frac{a}{q} n\right) = \sum_{d|q} 1_{d|n} \sum_{\chi \in (\mathbb{Z}/((q/d) \mathbb{Z}))^\times}
  a_\chi \chi(n/d),\]
  donde $|a_\chi|\leq 1$. Podr\'iamos tener una mejor cota para
  $a_\chi$ ({\em sumas de Gauss}), pero no nos importa.
  Obtenemos que
  \[\begin{aligned}
  \frac{1}{X}
  \int_X^{3 X} \left|\sum_{x<n\leq x + h_1} \lambda(n) e\left(\frac aq n\right)\right| d x &\leq \frac{1}{X} \sum_{d|q}
  \sum_{\chi \in (\mathbb{Z}/((q/d) \mathbb{Z}))^\times} \int_X^{3 X}
  \left|\sum_{\frac{x}{d} <m\leq \frac{x}{d} +
    \frac{h_1}{d}} \lambda(m)\chi(m) \right| d x \\ &\leq
  q \mathop{\max_{d|q}}_{
    (\mathbb{Z}/((q/d) \mathbb{Z}))^\times} \frac{1}{X/d} \int_{X/d}^{3 X/d}
    \left|\sum_{x' <m\leq x' + \frac{h_1}{d}} \lambda(m)\chi(m) \right| d x'
  .\end{aligned}\]
\end{proof}

Con la proposici\'on anterior vemos que s\'olo nos queda controlar el promedio de sumas cortas de $\lambda(n)\chi(n)$. Para ello, vamos a ver que funcionar\'ian las t\'ecnicas de la secci\'on anterior, igual que para $\lambda(n)$. Lo \'unico que necesitar\'iamos usar en la prueba es cancelaci\'on para las sumas
\[
 \sum_{n<x} \lambda(n)\chi(n)    \qquad   \sum_{p<Q} \chi(p)p^{it_0},
\]
y el control de dicha sumas depende de las funciones $Z_{\lambda \chi}(s)=Z_{\chi^2}(2s)/Z_{\chi}(s)$ y $Z_{\chi}'(s)/Z_{\chi}(s)$ con $Z_{\chi}(s)=L(s,\chi)$, donde $s\mapsto L(s,\chi)$
son las as\'i llamadas funciones $L$ de Dirichlet:
\[L(s,\chi) = \sum_n \chi(n) n^{-s}.\]
Por lo tanto, necesitaremos control sobre los ceros y el tama\~no de dichas funciones zeta. Para el caso de $\lambda$ usamos \cite[Thm.~8.29]{MR2061214} (que es nuestro Teorema \ref{te:vinogradov_korobov}) y en el lema anterior a dicho teorema puede verse que dicho control proviene de cotas superiores para $|L(s,\chi)|$ en dicha zona. Como 
\[
 L(s,\chi)=q^{-s}\sum_{b(q)} \chi(b) \sum_{m=0}^{\infty} \left(m+\frac bq\right)^{-s}
\]
y $\sum_{m=0}^{\infty} (m+\alpha)^{-s}$ es muy similar a $\zeta(s)$, esencialmente podemos conseguir para $|L(s,\chi)|$ las mismas cotas que para $|\zeta(s)|$, y as\'i tenemos un equivalente al Teorema \ref{te:vinogradov_korobov} para este caso.

\begin{teorema}\label{te:vinogradov_korobov_siegel}
  Hay una constante $c>0$ tal que 
$L(s,\chi)\neq 0$ para $s=\sigma+it$ con $\sigma \ge 1- \frac{c}{\log q+(\log t)^{2/3}(\log\log t)^{1/3}}$, $|t|\ge 3$, y en dicha zona tambi\'en se cumplen las cotas
\[
\frac{1}{L(s,\chi)}\ll \log q+(\log t)^{2/3}(\log\log t)^{1/3},
\]
\[\frac{L'(s,\chi)}{L(s,\chi)}\ll \log q+ (\log t)^{2/3}(\log\log t)^{1/3}.\]
Adem\'as, en $|t|\le 3$ se cumple que $L(s,\chi)\neq 0$ en $\sigma>1-\frac{c}{\sqrt q (\log q)^2}$ y en esa zona $\frac{1}{L(s,\chi)}\ll \sqrt q (\log q)^2$, $\frac{L'(s,\chi)}{L(s,\chi)}\ll \sqrt q (\log q)^2$.
\end{teorema}

Las cotas para $|t|\le 3$ son cl\'asicas y provienen de otras t\'ecnicas \cite[Cap\'itulo 11]{MR2378655}. \'Estos resultados son todos efectivos; no esconden
constantes no especificables -- en particular, no lidiamos con los {\em ceros
  de Siegel}.

Para demostrar el Teorema  \ref{te:TNP_liouville} usamos las cotas del Teorema \ref{te:vinogradov_korobov} con $t$ escogido de manera tal que $(\log t)^{2/3}(\log\log t)^{1/3}$ sea igual a $(\log x)^{2/5}(\log\log x)^{1/5}$, por lo que, si $\sqrt{q}(\log q)^2 \ll (\log x)^{2/5}(\log\log x)^{1/5}$, tendremos las mismas cotas en el Teorema \ref{te:vinogradov_korobov_siegel}, y
luego ser\'an ciertos los resultados an\'alogos al Teorema \ref{te:TNP_liouville} y al Corolario \ref{cor:sumlambinv} para $\lambda(n)\chi(n)$. En particular

\begin{corolario}\label{cor:sumlambinvchi}
Sean $x\geq 1$, $t\leq e^{(\log x)^{3/5-\epsilon}}$, $\epsilon>0$, y  $\chi$  un car\'acter de m\'odulo $q\le (\log x)^{4/5-\epsilon}$. Entonces
\[\sum_{x<n\leq 2x} \frac{\lambda(n)\chi(n)}{n^{1+it}} \ll \exp(-(\log x)^{3/5+o_\epsilon(1)}).\]
\end{corolario}

Adem\'as, en la zona $\log t>(\log Q)^a$ con $a>0$, si $q\le (\log Q)^{2}$ tenemos que las cotas del Teorema \ref{te:vinogradov_korobov_siegel} ser\'ian las mismas que sin $\log q$, y
luego el resultado an\'alogo a la Proposici\'on \ref{pr:ptgrande} tambi\'en ser\'a cierto.

\begin{proposicion}\label{pr:ptgrandechi}
Sea $\chi$ un car\'acter de m\'odulo $q\le (\log Q)^{2}$.  Para $\exp((\log Q)^a)\leq t_0\leq \exp((\log Q)^{(3/2) (1-a)})$,
  $a>0$, $0<\delta\leq 1$,
\[
\sum_{Q<p\leq (1+\delta) Q} \chi(p)p^{-1-it_0} \ll \exp(- (\log Q)^{a+o(1)}).
\]
\end{proposicion}

Como en la demostraci\'on del Teorema \ref{te:result_principal} se usa el equivalente al Corolario \ref{cor:sumlambinvchi} con $x\approx X$ y el equivalente a la Proposici\'on \ref{pr:ptgrandechi} con $\log Q=(\log X)^{1-\epsilon}$, tenemos que 

\begin{teorema}\label{te:result_principal_chi}
Sea $1< h\le X$, $\epsilon>0$. Para $q\le (\log X)^{4/5-\epsilon}$,
\[
 \mathbb E_{X<x\le 2X} |\mathbb E_{(1-\frac hX)x<n\le x} \lambda(n)\chi(n)|^2\ll_{\epsilon} \frac{1}{(\log h)^{1-\epsilon}}+\frac{1}{(\log X)^{1/3-\epsilon}}.
\]
\end{teorema}
\begin{corolario}\label{cor:l1normacar}
  Sea $1< h\le X$, $\epsilon>0$. Para $q\le (\log X)^{4/5-\epsilon}$,
\[
\mathbb E_{X<x\le 2X} \left|\mathbb E_{x<n\le x+h} \lambda(n)\chi(n)\right|
\ll_{\epsilon} \frac{1}{(\log h)^{1/3-\epsilon}}+\frac{1}{(\log X)^{1/9-\epsilon}}.
\]
\end{corolario}
\begin{proof}
Notemos que es suficiente demostrar la desigualdad
\begin{equation}\label{eq:desplazada}
\mathbb E_{X<x\le 2X} \left|\mathbb E_{x-h<n\le x} \lambda(n)\chi(n)\right|
\ll_{\epsilon} \frac{1}{(\log h)^{1/3-\epsilon}}+\frac{1}{(\log X)^{1/9-\epsilon}}.
\end{equation}
Ahora, para $0<\delta \leq 1$ tenemos que 
\[
 \mathbb E_{X<x\le 2X} \left|\mathbb E_{x-h<n\le x} \lambda(n)\chi(n)\right|\ll\mathbb E_{j: 1< (1+\delta)^j\le 2} \mathop{\mathbb E_{Y<x\le Y+\delta Y}}_{Y=(1+\delta)^j X} \left|\mathbb E_{x-h<n\le x} \lambda(n)\chi(n)\right|
\]
y como para cualquier $x\in (Y,Y+\delta Y]$ se cumple que
\[
 \mathbb E_{x-h<n\le x} \lambda(n)\chi(n)=O(\delta)+\mathbb E_{x-h\frac{x}{Y}<n\le x} \lambda(n)\chi(n)
\]
deducimos que 
\[
\mathbb E_{X<x\le 2X} \left|\mathbb E_{x-h<n\le x} \lambda(n)\chi(n)\right|\ll \delta + \max_{Y\in [X,2X]} \mathbb E_{Y<x\le Y+\delta Y} |\mathbb E_{x-h\frac{x}{Y}<n\le x} \lambda(n)\chi(n) |.
\]
Ahora aplicamos Cauchy-Schwarz:
\[
 \mathbb E_{Y<x\le Y+\delta Y} |\mathbb E_{x-h\frac{x}{Y}<n\le x} \lambda(n)\chi(n) |\ll \sqrt{\mathbb E_{Y<x\le Y+\delta Y} |\mathbb E_{x-h\frac{x}{Y}<n\le x} \lambda(n)\chi(n) |^2 },
\]
y luego,
por el Teorema \ref{te:result_principal_chi} (acotando brutalmente
la integral de una cantidad no negativa sobre $(Y,Y+\delta Y\rbrack$ por la
integral de la misma cantidad sobre $(Y,2 Y\rbrack$) vemos que 
\[
\mathbb E_{X<x\le 2X} \left|\mathbb E_{x-h<n\le x} \lambda(n)\chi(n)\right|\ll \delta + \max_{Y\in [X,2X]} \sqrt{\delta^{-1} \left(\frac{1}{(\log h)^{1-\epsilon}}+\frac{1}{(\log Y)^{1/3-\epsilon}}\right)}.
\]
Finalmente, tomando 
\[\delta = \min\left(1, \max\left(\frac{1}{(\log h)^{1/3-\epsilon}},\frac{1}{(\log X)^{1/9-\epsilon}}\right)\right)
\]
demostramos \eqref{eq:desplazada}.

\end{proof}

Usando los resultados anteriores, obtenemos lo siguiente.
\begin{proposicion}\label{pr:arcos_mayores}
 Sea $1<R^5\leq h\le X$, $|\alpha-\frac aq|\le \frac{1}{q\sqrt h}$ con $q\le R\le (\log X)^{4/5-\epsilon}$. Entonces, para $\epsilon>0$,
\[
\frac{1}{h X} \int_X^{2 X} \left|\sum_{x<n\leq x + h} \lambda(n) e(\alpha n)\right| d x \ll_\epsilon \frac{1}{R}+
\frac{R}{(\log h)^{\frac{1}{3}-\epsilon}}+\frac{R}{(\log X)^{\frac 19-\epsilon}}.
\]
\end{proposicion}
\begin{proof}
  Por la Proposici\'on \ref{pr:suma_a_multip}
  y el Corolario \ref{cor:l1normacar}.
\end{proof}

Finalmente, usando el Teorema de aproximaci\'on de Dirichlet y tomando $R$ igual a $\min((\log h)^{1/3},(\log X)^{1/9})^{4/5}$ en la Proposiciones \ref{pr:arcos_menores} y \ref{pr:arcos_mayores}, conclu\'imos que

\begin{teorema}\label{te:lambda_fourier}
  Sean $X>1$, $1< h\leq X$.
  Entonces, para todo $\alpha\in \mathbb R$ y $\epsilon>0$, 
\[
 \frac{1}{h X} \int_X^{2 X} \left|\sum_{x<n\leq x + h} \lambda(n) e(\alpha n)\right| d x \ll_\epsilon \frac{1}{(\log h)^{\frac{1}{15}-\epsilon}}+\frac{1}{(\log X)^{\frac{1}{45}-\epsilon}}.
\]
\end{teorema}
Los exponentes $1/15$, $1/45$ no son de ninguna manera \'optimos; el lector
interesado puede divertirse mejor\'andolos.

Vamos a ver como el teorema anterior implica la conjetura de Chowla en promedio (Teorema \ref{te:mrt1}). Para ello, veamos la relaci\'on entre las sumas cortas de $f(n)$ y sumas largas de $\overline{f(n)}f(n+h)$. Supongamos que $f(n)$ tiene soporte finito. Entonces, expandiendo el cuadrado y cambiando el orden de sumaci\'on obtenemos
\[
 \sum_x |\sum_{x<n\le x+h} f(n)|^2=\sum_{|n-m|<h} f(m)\overline{f(n)} (h-|n-m|)
\]
Por tanto, escribiendo $m=n+j$, tenemos que
\[
 \sum_x |\sum_{x<n\le x+h} f(n)|^2=\sum_{|j|< h} (h-|j|) \sum_n f(n+j)\overline{f(n)}.
\]
Ah\'i vemos directamente que las sumas cortas de $f(n)$ controlan un promedio de sumas largas de $\overline{f(n)}f(n+j)$. El problema es que podr\'ia ser que hubiera cancelaci\'on en la suma en $j$ en vez de en la suma en $n$. Para evitarlo, usamos la misma identidad con $f(n)=f_0(n)e(n\alpha)$, que da
\[
 \sum_x |\sum_{x<n\le x+h} f_0(n)e(n\alpha)|^2=\sum_{|j|< h} \left[(h-|j|) \sum_n f_0(n+j)\overline{f_0(n)} \right] e(j\alpha).
\]
Ahora, teniendo en cuenta la identidad de Parseval
\begin{equation}\label{eq:parseval_toro}
 \int_0^1 |\sum_j a_j e(j\alpha)|^2 d\alpha= \sum_j |a_j|^2,
\end{equation}
la cual simplemente proviene de que $\int_0^1 e(m\alpha)\, d\alpha=1_{m=0}$, tenemos que
\[
 \sum_{|j|<h} |(h-|j|)|^2 \left|\sum_n f_0(n+j)\overline{f_0(n)} \right|^2 =\int_0^1 \left| \sum_x |\sum_{x<n\le x+h} f_0(n)e(n\alpha)|^2 \right|^2 d \alpha.
\]
As\'i, si
\[
 M=\max_{\alpha}  \sum_x |\sum_{x<n\le x+h} f_0(n)e(n\alpha)|^2 
\]
entonces sac\'andolo fuera de la integral tenemos
\[
 \sum_{|j|<h} |(h-|j|)|^2 \left|\sum_n f_0(n+j)\overline{f_0(n)} \right|^2\le M \sum_x \int_0^1|\sum_{x<n\le x+h} f_0(n)e(n\alpha)|^2  \, d\alpha,
\]
y usando de nuevo (\ref{eq:parseval_toro}) sobre la parte derecha conclu\'imos que
\[
 \sum_{|j|<h} |(h-|j|)|^2 \left|\sum_n f_0(n+j)\overline{f_0(n)} \right|^2\le M \sum_x \sum_{x<n\le x+h} |f_0(n)|^2.
\]
Finalmente, aplicando esta desigualdad con $f_0(n)=\lambda(n)1_{(X,2X]}(n)$ y usando el Teorema \ref{te:lambda_fourier}, obtenemos 
\begin{corolario}\label{cor:chowla_promedio_2}
Para $1\le h\le X$ y para todo $\epsilon>0$
\[
 \frac{1}{hX^2}\sum_{j\le h/2} |\sum_{X<n,n+j\le 2X} \lambda(n+j)\lambda(n)|^2  \ll_\epsilon \frac{1}{(\log h)^{\frac{1}{15}-\epsilon}}+\frac{1}{(\log X)^{\frac{1}{45}-\epsilon}}.
\]
\end{corolario}
A partir de ah\'i es muy sencillo eliminar la condici\'on $X<n+j\le 2X$, y
as\'i obtener el Teorema \ref{te:mrt1} en el caso de dos factores ($k=2$). El caso de m\'as factores se sigue del siguiente lema (ejercicio \ref{ej:cauchy_corto})

\begin{lema}\label{le:cauchy_corto}
  Sea $1\le H\le X$. Para funciones $f,g:\mathbb{Z}^+\to \mathbb{C}$
  cualesquiera con $|g(n)|\le 1$, $|f(n)|\le 1$ para todo $n$ y soporte
  en $[1,X]$, se cumple la desigualdad
\[
 \frac{1}{HX^2}\sum_{h\le H}  |\sum_n f(n+h) g(n)|^2 \le \sqrt{ \frac{1}{HX^2} \sum_{|h|\le H}|\sum_n f(n+h)\overline{f(n)}|^2}.
\]
\end{lema}

{\em Reflexiones finales.} Como hemos visto, una vez que se tiene
una cierta cota sobre sumas exponenciales en promedio (Teorema \ref{te:lambda_fourier}) podemos obtener el resultado de tipo ``Chowla en promedio'' que
dese\'abamos (Teorema \ref{te:mrt1}) con bastante facilidad.

Es bueno reflexionar sobre que tipo de cotas sobre sumas exponenciales
necesitamos para obtener otros resultados sobre las autocorrelaciones de
$\lambda$. ?`Qu\'e pasa si deseamos saber el promedio de $\lambda(n) \lambda(n+h) \lambda(n+ 2 h)$ para la mayor parte de valores de $h$ en un intervalo peque\~no, digamos?

Muchas tales preguntas se reducen a acotar {\em seminormas de Gowers}
$|\lambda|_{U^k}$. Ya ser\'{\i}a un muy buen comienzo, por ejemplo,
poder acotar la {\em segunda norma de Gowers} $|\lambda|_{U^2}$, definida por
\begin{equation}\label{eq:gowersu2}
|\lambda|_{U^2}^2 = \frac{1}{H^2 X} \sum_{1\leq h_1, h_2\leq H} \sum_{1\leq n\leq X}
\lambda(n) \lambda(n+h_1) \lambda(n+h_2) \lambda(n+h_1+h_2)\end{equation}
para todo $H=H(X)\to \infty$. En efecto, si pudi\'eramos mostrar
que $|\lambda|_{U^2}=o(1)$, tendr\'{\i}amos que
\[\frac{1}{H X} \sum_{h\leq H} \sum_{n\leq X} \lambda(n) \lambda(n+h) \lambda(n+2 h) = o(1)\]
despues de algunas aplicaciones de Cauchy-Schwarz. M\'as a\'un,
si pudi\'eramos mostrar que la tercera norma de Gowers $|\lambda|_{U^3}$ es $o(1)$, podr\'iamos deducir, de manera similar, que
\[\frac{1}{H X^2} \sum_{h\leq H} \left|\sum_{n\leq X} \lambda(n) \lambda(n+h) \lambda(n+2 h)\right|^2 = o(1).\]

Acotar (\ref{eq:gowersu2}) por $o(1)$ resulta ser equivalente a probar
la {\em conjetura de uniformidad de Fourier}
\[ \frac{1}{h X} \int_0^X \sup_{\alpha\in \mathbb{R}/\mathbb{Z}}
\left|\sum_{x<n\leq x+h} \lambda(n) e(\alpha n)\right| dx = o(1)\]
para $h=h(N)\to \infty$. 
Esta conjetura, planteada por primera vez en \cite{MR3435814}, parece 
m\'as dif\'icil que el Teorema \ref{te:lambda_fourier}, en el cual
el orden de la integral y del supremo $\sup_{\alpha\in \mathbb{R}/\mathbb{Z}}$
  (impl\'{\i}cito en el Teorema
 \ref{te:lambda_fourier}) es el inverso.

\subsubsection{Ejercicios}
\begin{enumerate}
\item\label{ej:pentimento}
      Sea $G$ un grupo abeliano finito. 
  \begin{enumerate}
  \item Pruebe que, para todo car\'acter $\chi$ de $G$ y todo
    $g\in G$, $|\chi(g)|=1$.
    (Sugerencia: muestre que $\chi(g)^{|G|} =1$.)
  \item\label{it:permut} Sea $\chi$ un car\'acter no trivial de $G$, es decir, un car\'acter tal
    que existe un $g\in G$ para el cual $\chi(g)\ne 1$. Muestre que
    \[\psi(g) \sum_{h\in G} \psi(h) = \sum_{h\in G} \psi(g h)
    = \sum_{h\in G} \psi(h).\]
    Concluya que $\sum_{h\in G} \psi(h) = 0$.
  \item Sean $\chi$, $\chi'$ dos car\'acteres distintos de $G$. Muestre que
    $\psi = \overline{\chi} \cdot \chi'$ es un car\'acter no trivial. Entonces,
    por (\ref{it:permut}), $\sum_{g\in G} \overline{\chi}(g) \chi'(g) = 0$.
  \item\label{ej:pentifin}
    Muestre que hay $|G|$ car\'acteres distintos $\chi:G\to \mathbb{C}$.
    Concluya que los car\'acteres de $G$ forman una base ortonormal del espacio
    de funciones de $G$ a $\mathbb{C}$ con el producto escalar
    \[\langle f_1,f_2\rangle = \frac{1}{|G|} \sum_{g\in G} \overline{f_1}(g) f_2(g)
    .\]
  \end{enumerate}
\item\label{ej:habakuk}
  \begin{enumerate}
    \item Muestre que, para
  toda funci\'on $q$-peri\'odica
  $f:\mathbb{Z}\to \mathbb{C}$ con soporte en los enteros coprimos con $q$,
  \[f(n) =  \sum_{\chi\in \widehat{G}_q} \widehat{f}(\chi) \chi(n),\]
  donde $G_q = (\mathbb{Z}/q\mathbb{Z})^\times$,
  $\widehat{G}_q$ es el grupo de car\'acteres de $G$ (ejercicio
  \ref{ej:pentimento}) y 
  \[\widehat{f}(\chi) =
  \frac{1}{|G|} \sum_{m\in (\mathbb{Z}/q\mathbb{Z})^\times} 
  \overline{\chi(m)} f(m).\]
  Para este prop\'osito,
  utilice el hecho que $\widehat{G}_q$ es una base ortonormal
  (ejercicio \ref{ej:pentifin}).
    Est\'a claro que, si $|f(m)|\leq 1$ para todo $m$, $|\widehat{f}(\chi)|\leq 1$.
  \item  Para toda funci\'on $q$-peri\'odica $f:\mathbb{Z}\to \mathbb{C}$,
  \[f(n) = \sum_{d|q} f(n) 1_{(n,q)=d} = \sum_{d|q} 1_{d|n} f_d(n/d),\]
  donde $f_d:\mathbb{Z}\to \mathbb{C}$ es la funci\'on tal que
  $f_d(m) = f(m d)$ si $(m,q/d)=1$ y $f_d(m)=0$ de otra manera. Concluya que para ciertos $|a_{\chi}|\le 1$
  \[f(n) = \sum_{d|q} 
  \sum_{\chi \in \widehat{G}_{q/d}} a_{\chi} 1_{d|n} \chi(n/d).\]
  \end{enumerate}
  
\item\label{ej:cauchy_corto} Demuestre el Lema \ref{le:cauchy_corto}.
  Para hacerlo: expanda el cuadrado, meta adentro la suma en $h$, aplique Cauchy-Schwarz, expanda los cuadrados y reagr\'upelos de forma que queden cuadrados de sumas en $n$. Aplique dicho lema para demostrar el Teorema \ref{te:mrt1} a partir del Corolario \ref{cor:chowla_promedio_2}.

\item En este problema vamos a concluir la estimaci\'on de
  $S_N=\sum_{p+q+b=N} \log p \log q$ que comenzamos en el ejercicio \ref{ej:goldbach_facil} de \S \ref{sec:ejercicios_4_menores}
  .  All\'i vimos que para cualquier $0<\delta<1/2$
 \[
  S_N=O(\delta N^2 \log N) + \int_{\mathfrak M} F(\alpha) G^2(\alpha) e(-N\alpha) \, d\alpha
 \]
con $\mathfrak M=(-\frac{1}{N\delta},\frac{1}{N\delta})$ los arcos mayores, $F(\alpha)=\sum_{b=1}^{N} e(b\alpha)$, $G(\alpha)=\sum_{p\le N} \log p \, e(p\alpha)$.
\begin{enumerate}
 \item Use la regla del rect\'angulo \eqref{eq:regla_rectangulo} para obtener la estimaci\'on
\begin{equation}\label{eq:Fformu}
 F(\alpha)=\frac{e(N\alpha)-1}{2\pi i\alpha} + O(\delta^{-1})
\end{equation}
para $\alpha\in \mathfrak M$.
\item Escriba $G(\alpha)=\sum_{n\le N} ( 1_{n \text{ primo}} \log n) e(n\alpha)$.
  Use sumaci\'on por partes (ejercicio \ref{ej:sumapart}
  de \S \ref{sec:lobasico}), el teorema \ref{te:TNP} y
  la estimaci\'on (\ref{eq:Fformu})
  para obtener que 
\[
G(\alpha)=\frac{e(N\alpha)-1}{2\pi i\alpha} +
O_A\left(\frac{\delta^{-1} N}{(\log N)^A}\right)
\]
para $\alpha\in \mathfrak{M}$ y $A>0$ arbitrario.
\item Por los apartados anteriores y el cambio de variable $\alpha=t/N$, demuestre que 
\[
S_N= (C+O(\delta^2)) N^2 + O(\delta N^2\log N)
+
O_A\left(\frac{N^2}{\delta^{2 }(\log N)^{A}}\right)
\]
con $C$ definida como la constante $\int_{-\infty}^{\infty} \frac{(e(t)-1)^3 e(-t)}{(2\pi i t)^3} \, dt$. Tomando $\delta=(\log N)^{-2}$, $A=5$, obtenemos que
\[
 S_N =  (C+o(1)) N^2.
 \]
\item Para demostrar que $C=1/2$ sin dolor, verifique que el mismo
  razonamiento muestra que $T_N = \sum_{a+b+c=N} 1$ satisface
  $T_N = (C+o(1)) N^2$, y luego muestre que $T_N = (1/2 + o(1)) N^2$ de otra
  manera. Alternativamente, muestre que $C=1/2$ como prefiera.
\end{enumerate}
\end{enumerate}

\section{La autocorrelaci\'on de $\lambda$ en escala logar\'{\i}tmica}
\subsection{Inicio y esbozo del argumento}
Quisi\'eramos ahora probar el Teorema \ref{teo:taochowla}.
Para simplificar la notaci\'on, nos concentraremos en el caso
$a_1=a_2=b_2=1$, $b_1=0$; es decir, probaremos que, para $w=w(x)$ tal que $w\to \infty$ cuando $x\to \infty$,
\begin{equation}\label{eq:chowla01}
  \sum_{\frac{x}{w} <n\leq x} \frac{\lambda(n) \lambda(n+1)}{n} = o(\log w).
  \end{equation}
  El tratamiento del caso general ($a_i$, $b_i$ arbitrarios) es pr\'acticamente
  id\'entico.
  
 De hecho, probaremos \eqref{eq:chowla01} en la siguiente forma cuantitativa.
\begin{teorema}\label{te:chowla_log_cuant}
Sean $w>e^e$, $x>e^{e^e}$. Entonces 
 \[
  \sum_{x/w<n\le x} \frac{\lambda(n)\lambda(n+1)}{n}\ll \frac{\log w}{\min(\log_3 w,\log_4 x)^{1/5}}.\]
\end{teorema} 
Cuando escribimos $\log_k$, queremos decir el logaritmo iterado $k$ veces:
$\log_2 x = \log \log x$, $\log_3 x = \log \log \log x$,
$\log_4 x = \log \log \log \log x$.

El primer paso hacia el Teorema \ref{te:chowla_log_cuant} consiste en usar la multiplicatividad de $\lambda$ para escribir la suma como una suma de sumas con una condici\'on de divisibilidad.

\begin{lema}\label{le:truco_divisibilidad}
    Sea $1\le w\le x$.
    Sean  $1\le K_0\leq K_1<x/w$. Entonces
    \begin{equation}\label{eq:promp}
      \sum_{\frac{x}{w} <n\leq x} \frac{\lambda(n) \lambda(n+1)}{n} =
      \frac{1}{\ell} \left(\sum_{K_0<p\leq K_1}
      \mathop{\sum_{\frac{x}{w}< n \leq x}}_{p|n}
      \frac{\lambda(n) \lambda(n+p)}{n} + O(\log K_1)\right),\end{equation}
donde $\ell = \sum_{K_0<p\le K_1} p^{-1}$.
  \end{lema}
 
La idea principal de la prueba es que el intervalo $x/w<n\leq x$ con el peso $1/n$ es casi invariante bajo desplazamientos multiplicativos
$p\cdot$, $p$ peque\~no. 
  
  \begin{proof}
    Est\'a claro que
    \[\sum_{K_0<p\leq K_1} \frac{1}{p}
    \sum_{\frac{x}{w} <n\leq x} \frac{\lambda(n) \lambda(n+1)}{n} =
    \sum_{K_0<p\leq K_1} 
    \sum_{\frac{x}{w} <n\leq x} \frac{\lambda(p n) \lambda(p n+ p)}{p n}
    \]
    Ahora bien, para $p\leq K_1$,
    \[\begin{aligned}
    \sum_{\frac{x}{w} <n\leq x} \frac{\lambda(p n) \lambda(p n+ p)}{p n}
    &=
    \sum_{\frac{x}{p w} <n\leq \frac{x}{p}}
    \frac{\lambda(p n) \lambda(p n+ p)}{p n} +
    O\left(\frac{1}{p} \sum_{\frac{x}{p w} <n\leq \frac{x}{w}} \frac{1}{n}
    + \frac{1}{p} \sum_{\frac{x}{p} <n\leq x} \frac{1}{n}
    \right)\\
    &=   \mathop{\sum_{\frac{x}{w} <n\leq x}}_{p|n}
    \frac{\lambda(n) \lambda(n+p)}{n} +
   \frac{O(\log p)}{p}.
    \end{aligned}\]
    Aplicando el Corolario \ref{cor:corTNP} y 
    dividiendo por $\sum_{K_0<p\le K_1} p^{-1}$, obtenemos el resultado.
  \end{proof}
 Tras este resultado, para demostrar \eqref{eq:chowla01},  ser\'a suficiente mostrar que  para alg\'un par $(K_0, K_1)$ con $\log K_1 = o(\log w \sum_{K_0<p\le K_1} \frac 1p)$ se cumple
 \begin{equation}\label{eq:dobest}
  \sum_{K_0<p\leq K_1}
      \mathop{\sum_{\frac{x}{w}< n \leq x}}_{p|n}
      \frac{\lambda(n) \lambda(n+p)}{n}=o\left((\log w)\cdot
      \sum_{K_0<p\leq K_1} \frac{1}{p}\right).
 \end{equation}
La idea principal es la siguiente. Los m\'etodos de Matom\"aki y
  Radziwi{\l}{\l} (en particular, los que acabamos de ver en \S \ref{sec:fourcort}) nos bastar\'an para probar que
  \begin{equation}\label{eq:tristar}
    \sum_{K_0<p\leq K_1} \frac{1}{p} \sum_{\frac{x}{w}< n \leq x}
      \frac{\lambda(n) \lambda(n+p)}{n} = o\left((\log w)\cdot
      \sum_{K_0<p\leq K_1} \frac{1}{p}\right).
  \end{equation}
  Ahora bien, si vemos al hecho de ser divisible por $p$ como un evento
  aleatorio de probabilidad $1/p$, tiene sentido que los lados izquierdos
  de (\ref{eq:dobest}) y (\ref{eq:tristar}) sean aproximadamente iguales.

  En verdad, que $n$ sea divisible por $p$ {\em es}
  un evento aleatorio, si tomamos $n$ al azar entre $x/w$ y $x$. El
  problema reside en que se trata de un evento no independiente de
  $\lambda(n) \lambda(n+p)$.

  Ahora bien, resulta ser que -- para hablar de manera aproximada -- si la dependencia entre los eventos $p|n$ ($K_0<p\leq K_1)$ y $(\lambda(n),\lambda(n+1),\dotsc)$ ($x/w<n\leq x$) es fuerte para muchos valores de $(K_0,K_1)$, entonces existe un
  valor de $(K_0,K_1)$ en la cual no lo es tanto. Existe una medida de
  dependencia -- la {\em informaci\'on mutua}, definida en t\'erminos de la
  {\em entrop\'ia} -- la cual, si bien es un tanto burda, goza de una propiedad
  de aditividad. Esto conlleva que se pueda tratar a la entrop\'ia
  como un recurso agotable; podemos pensar en ella como una sopa con una
  cantidad finita de lentejas, de tal manera que, si la gente se va sirviendo,
  eventualmente a alguien
  le tendr\'a que tocar pocas lentejas (es decir,
  poca informaci\'on mutua).\footnote{En
    la versi\'on oral de estas charlas, se mencion\'o a
    un oll\'on de locro, pero se hizo aparente que parte de la audiencia 
    no sab\'ia qu\'e era el locro.}
    En el momento que nos toca pocas lentejas, la dependencia es d\'ebil,
  y sabremos proceder.

\subsubsection{Ejercicios}

\begin{enumerate}

\item Queremos ver que el resultado  \eqref{eq:chowla01} no es tan fuerte como la conjetura de Chowla.
Sea $(a_n)_{n\in\mathbb N}$ una sucesi\'on con $|a_n|\le 1$.

\begin{enumerate}
\item  Demuestre  que $\sum_{n\le x} a_n=o(x)$ implica $\sum_{x/w\le n\le x} \frac{a_n}{n}=o(\log w)$ cuando $x\to \infty$, con $x\ge w=w(x)\to\infty$. \emph{Sugerencia: use sumaci\'on por partes}.
 
\item Observe, usando la regla del rect\'angulo \eqref{eq:regla_rectangulo}, que la sucesi\'on $a_n=n^i$ satisface $\sum_{x/w\le n\le x} a_n/n=O(1)=o(\log w)$ cuando $x\to \infty$  para cualquier $x\ge w=w(x)\to\infty$, pero  que $\sum_{n\le x} a_n\neq o(x)$. Muestre que lo mismo ocurre para $a_n = \Re n^i = \cos(\log n)$.

\item Demuestre que si tuvi\'eramos cancelaci\'on para $w=2$, es decir $\sum_{x/2<n\le x} \frac{a_n}{n}=o(1)$, entonces s\'i podr\'iamos deducir que $\sum_{n\le x} a_n=o(x)$. \emph{Sugerencia: use sumaci\'on por partes}.
 
\end{enumerate}

\item El Teorema \ref{te:chowla_log_cuant} vale en el rango $2\le w\le x^{1/8}$. Usando ese resultado, deduzca que el teorema tambi\'en es cierto en el rango $x^{1/8}<w\le x$.
  
\end{enumerate}

\subsection{Sumas y esperanzas}\label{subs:sumyesp}
  
En esta secci\'on vamos a formalizar la relaci\'on entre las sumas en \eqref{eq:dobest} y los conceptos de probabilidad e independencia que hemos comentado. Para ello, es conveniente primero partir la suma en $n$ en segmentos cortos (de longitud $H$). Esto es lo que hacemos en el siguiente resultado.

\begin{lema}\label{le:trozos_long_H}
Sea $H$ un n\'umero natural tal que $K_1< H\le x/w$. Entonces para cualquier $p\le K_1$ tenemos
\[
      \mathop{\sum_{\frac{x}{w}< n \leq x}}_{p|n}
      \frac{\lambda(n) \lambda(n+p)}{n}=\frac{1}{H}
      \sum_{\frac{x}{w}< n \leq x} \frac 1n\sum_{j\le H-p}\lambda(n+j) \lambda(n+j+p) 1_{p\mid n+j}+ O\left(\frac{\log w}{H} + \frac{1}{p}\right).
\]
\end{lema}
La idea es simplemente utilizar el hecho de que el peso $1/n$ y el intervalo $(x/w,x\rbrack$ son aproximadamente invariantes bajo 
peque\~nos desplazamientos aditivos.
\begin{proof}
Para cualquier $j\le H$, por cambio de variable
\[
      \mathop{\sum_{\frac{x}{w}< n \leq x}}_{p|n}
      \frac{\lambda(n) \lambda(n+p)}{n}=
      \mathop{\sum_{\frac{x}{w}-j< n \leq x-j}}_{p|n+j}
      \frac{\lambda(n+j) \lambda(n+j+p)}{n+j}.
\] 
Como
\[
 \mathop{\sum_{\frac{x}{w}-j< n \leq \frac{x}{w}}}_{p|n+j}
      \frac{1}{n+j}+\mathop{\sum_{x-j< n \leq x}}_{p|n+j}
      \frac{1}{n+j}
      =
      \mathop{\sum_{\frac{x}{w}< n \leq \frac{x}{w}+ j}}_{p|n}
      \frac{1}{n}+\mathop{\sum_{x< n \leq x+j}}_{p|n}
      \frac{1}{n}
      \ll \frac 1p
      \left(1+\log \frac{x/w+H}{x/w}\right)
    \ll \frac 1p
\]
y
\[\begin{aligned}
\mathop{\sum_{\frac{x}{w}< n \leq x}}_{p|n+j}
\left(\frac{1}{n}-\frac{1}{n+j}\right)
&\le \mathop{\sum_{n>x/w}}_{p|n+j} \frac{j}{n (n+j)}
\le 2 H \mathop{\sum_{n>x/w}}_{p|n+j} \frac{1}{(n+j)^2}\\
&\le 2 H \mathop{\sum_{n>x/w}}_{p|n} \frac{1}{n^2}
\ll \frac{H}{p^2}\cdot \frac{1}{x/w p}\le \frac 1p,
\end{aligned}\]
vemos que
\[
  (H-p) \mathop{\sum_{\frac{x}{w}< n \leq x}}_{p|n}
  \frac{\lambda(n) \lambda(n+p)}{n} =
  \sum_{j\le H-p}\sum_{\frac xw<n\le x} \frac{\lambda(n+j)\lambda(n+j+p)1_{p\mid n+j}}{n} + O\left(\frac 1p\right) \cdot (H-p).
\]
Dividamos todo por $H$. Para conclu\'ir, notemos que
\[
\frac{p}{H} \left|\mathop{\sum_{\frac{x}{w}< n \leq x}}_{p|n}
\frac{\lambda(n) \lambda(n+p)}{n}\right| \leq
\frac{p}{H} \mathop{\sum_{\frac{x}{w}< n \leq x}}_{p|n} \frac{1}{n}
\ll \frac{p}{H} \cdot \frac{\log w + 1}{p} \leq \frac{\log w}{H} + \frac{1}{p}.
\]
\end{proof}

Tomemos $K_0 = \epsilon H/2$, $K_1 = \epsilon H$, con $0<\epsilon<1$ peque\~no. Nuestro objetivo -- el cual nos permitir\'a demostrar \eqref{eq:dobest},
gracias al Lema \ref{le:trozos_long_H} -- ser\'a acotar de forma no trivial la suma
\begin{equation}\label{eq:definS}
 S=\frac 1H\sum_{\frac{x}{w}< n \leq x} \frac 1n \sum_{K_0<p\leq K_1}\sum_{j\le H-p}\lambda(n+j) \lambda(n+j+p) 1_{p\mid n+j}.
\end{equation}
Para ver la conexi\'on con las probabilidades, observemos que si $N$ es la variable aleatoria que toma valores en el conjunto de enteros en el intervalo $(x/w,x]$ con probabilidad
\begin{equation}\label{eq:definN}
 \mathbb P(N=n)=\frac{1/n}{L}       \quad \text{si} \quad n\in (x/w,x],
\end{equation}
donde $L=\sum_{x/w<n\le x} \frac 1n$,
entonces podemos escribir
\[
 S=\frac{L}{H}\cdot \mathbb E \left(\sum_{K_0<p\leq K_1}\sum_{j\le H-p}\lambda(N+j) \lambda(N+j+p) 1_{N\equiv - j (p)}\right),
\]
es decir, $S$ es la esperanza de una variable aleatoria que es una suma doble de variables aleatorias. Para examinar la dependencia entre la condici\'on de divisibilidad y los t\'erminos con $\lambda$, definimos la variables aleatorias
\begin{equation}\label{eq:defXY}
    X_H = (\lambda(N+1),\lambda(N+2),\dotsc,\lambda(N+H)),\;\;\;\;\;\;
  Y_H = (N \mo p)_{K_0<p\leq K_1},\end{equation}
con $X_H$ tomando valores en $\{-1,1\}^H$ e $Y_H$ en $\Omega=\prod_{K_0<p\le K_1} \frac{\mathbb Z}{p \mathbb Z}$.  As\'i podemos escribir
\begin{equation}\label{eq:Snova}
 S=\frac{L}{H} \cdot \mathbb E (F(X_H,Y_H))
\end{equation}
con $F:\{-1,1\}^H\times \Omega\to \mathbb R$  la funci\'on
\begin{equation}\label{eq:defF}
 F(\vec x, \vec y)=\sum_{K_0<p\le K_1}\sum_{j\le H-p}  x_j x_{j+p} 1_{y_p\equiv -j (p)}.
\end{equation}

Obtenemos el siguiente resultado. El inter\'es en estimar la expresi\'on
en el lado izquierdo de (\ref{eq:acotEsp}) viene, claro est\'a, del hecho
que aparece en el lado derecho de (\ref{eq:promp}).
\begin{lema}\label{le:suma_como_esperanza}
  Sea $1\leq w<x$.
  Sean $K_0=\epsilon H/2$ y $K_1=\epsilon H$, con $H\leq x/w$ y
  $\max\left(\frac{1}{\log w},\frac{1}{\sqrt H}\right)\leq \epsilon<1$.
  Entonces
  \begin{equation}\label{eq:acotEsp} \sum_{K_0<p\leq K_1} \mathop{\sum_{\frac{x}{w}< n \leq x}}_{p|n}
\frac{\lambda(n) \lambda(n+p)}{n} 
 =  \frac{L}{H} \mathbb E(F(X_H,Y_H)) +
 O\left(\epsilon \frac{\log w}{\log H}\right),
 \end{equation}
 donde $L=\sum_{x/w<n\le x} n^{-1}$ y $F$ es como en \eqref{eq:defF}.
\end{lema}
\begin{proof}
Por el Lema
\ref{le:trozos_long_H} y las estimaciones del Corolario
\ref{cor:corTNP},
\begin{equation}\label{eq:dandelio}\begin{aligned}
  \sum_{K_0<p\leq K_1}
&\mathop{\sum_{\frac{x}{w}< n \leq x}}_{p|n}
\frac{\lambda(n) \lambda(n+p)}{n} = 
S + 
\sum_{K_0<p\leq K_1} O\left(\frac{1}{p} + \frac{\log w}{H}\right)\\
&= S + \left(\log \log K_1 - \log \log \frac{K_1}{2} + O\left(\frac{1}{\log K_1}\right)\right) + O\left(\frac{\log w}{H}\cdot \frac{K_1}{\log K_1}\right)\\
&= S + O\left(\frac{\epsilon \log w + 1}{\log K_1}\right),
\end{aligned}\end{equation}
donde $S$ es como en (\ref{eq:definS}).
Gracias a $\epsilon\geq \max(1/\log w,1/\sqrt{H})$, vemos que
$\epsilon \log w + 1\le 2 \epsilon \log w$ y $\log K_1 \le (\log H)/2$.
Usamos la expresi\'on (\ref{eq:Snova}) para $S$.
\end{proof}

El plan es mostrar que, para alg\'un $H$,
las variables $X_H$ e $Y_H$ son m\'as o menos independientes,
y usar este hecho para obtener $\mathbb E(F(X_H,Y_H))=o(H/\log H)$,
lo cual es exactamente lo necesario para poder conclu\'ir, por
el Lema \ref{le:truco_divisibilidad}, que \eqref{eq:chowla01} se
cumple. Para ello aprovecharemos que la funci\'on $F$ puede
escribirse como una suma de variables aleatorias:
\begin{equation}\label{eq:Fcomosuma}
 F(\vec x,\vec y)=\sum_{\frac{\epsilon H}2<p\le \epsilon H} F_p(\vec x, y_p)
\end{equation}
con $\vec y=(y_p)_p$ y $F_p(\vec x,\cdot ):\mathbb Z/p\mathbb Z \to \mathbb R$ definida por 
$
 F_p(\vec x, t)= \sum_{j\le H-p,\; j\equiv -t (p)} x_j x_{j+p}.
$
Como
\begin{equation}\label{eq:max_Fp}
 \|F_p\|_{\infty}\le \frac{H}{p}\le \frac {2}{\epsilon},
\end{equation}
por el teorema de los n\'umeros primos vemos que 
\begin{equation}\label{eq:max_F}
 \|F\|_{\infty} \ll \frac{H}{\log H}
\end{equation}
por lo que s\'olo necesitamos mejorar un poco esa cota para obtener nuestro objetivo para $\mathbb E(F(X_H,Y_H))$.

Ahora bien, tenemos el siguiente resultado.
\begin{lema}\label{le:Hoeffding_corolario}
Sean $C,n\ge 1$ y $\Omega=\Omega_1\times \ldots \times \Omega_n$ con $\Omega_m$ conjuntos finitos, y $G_m$ una funci\'on real sobre $\Omega_m$ para $1\le m\le n$, con $\|G_m\|_{\infty}\le C$. Definamos $G:\Omega\to\R$ por 
\[
 G(\vec t)=G(t_1,t_2,\ldots t_n)=G_1(t_1)+G_2(t_2)+\ldots + G_n(t_n),
\]
y sea $\overline G$ su promedio sobre $\Omega$.
Entonces, para cualquier $0<\mu<1$, se cumple que 
\[
 |G(\vec t)-\overline G|\le \mu \cdot C n
\]
para todo $\vec t\in\Omega$ excepto para un conjunto de tama\~no a lo m\'as $2|\Omega|^{1-\frac{\mu^2/2}{\log R}}$, con $R=|\Omega|^{1/n}$.
\end{lema}
\begin{proof}
Si consideramos una variable aleatoria $T=(T_1,\ldots,T_n)$, $T_m$ con distribuci\'on uniforme en $\Omega_m$ e independientes, el enunciado del lema equivale a decir que 
\[
 \mathbb P(|G(T)-\mathbb E(G(T))|\ge \mu \cdot C n) \le 2e^{-\frac{\mu^2 n}2}.
\]
Pero esto se deduce directamente de la desigualdad de Hoeffding, que nos dice que si $X_1,X_2,\ldots, X_n$ son variables aleatorias independientes  tomando valores en el intervalo $[-C,C]$ entonces para $S=X_1+\ldots +X_n$ se cumple que
\begin{equation}\label{eq:hoeffding}
 \mathbb P(|S-\mathbb E(S)|\ge s)\le 2 e^{-\frac{s^2}{2C^2 n}}.
 \end{equation}
Un breve comentario: si bien (\ref{eq:hoeffding}) es de forma claramente
similar al teorema central del l\'imite, se trata de un resultado de
{\em grandes desviaciones}, pues es v\'alido para $s$ arbitrario, mientras
que el teorema central del l\'imite se ocupa del caso en el cual $s$
est\'a acotado por un m\'ultiplo constante de la desviaci\'on est\'andar,
es decir, el caso $s\ll C \sqrt{n}$.
\end{proof}

As\'i, teniendo en cuenta \eqref{eq:Fcomosuma} y \eqref{eq:max_Fp}, podemos aplicar el Lema \ref{le:Hoeffding_corolario} con $F(\vec x,\cdot)$ para obtener (usando el teorema de los n\'umeros primos) que,
para $H$ m\'as grande que una constante y $\epsilon\geq 2/\sqrt{H}$,
\begin{equation}\label{eq:Faverage_behaviour}
 |F(\vec x,\vec y)-\overline{F(\vec x, \cdot)}|\le \mu \frac{4 H}{\log H} 
\end{equation}
para todo $\vec y\in \Omega$ excepto en un conjunto $E_{\vec x} \subset \Omega$ que satisface
\begin{equation}\label{eq:Fexceptional_sets}
 \left|E_{\vec{x}}\right|\le 2 |\Omega|^{1-\frac{\mu^2}{2\log H}}.
\end{equation}

Esta cota, junto con (\ref{eq:max_Fp}),
nos ayudar\'a a estimar $\mathbb E(F(X_H,Y_H))$, ya que s\'olo
tendremos que preocuparnos en mostrar que, la mayor parte del tiempo,
$Y_H$ tiende a evitar un conjunto relativamente peque\~no de valores
$E_{X_H}$, dado por $X_H$. Claro est\'a, como $Y_H$ est\'a casi equidistribu\'ida,
tal aseveraci\'on
se deducir\'ia de inmediato si $X_H$ y $Y_H$ fueran variables independientes.
Como veremos, bastar\'a probar una forma muy d\'ebil de independencia, para
alg\'un $H$.

\subsubsection{Ejercicios}\label{se:ejercicios_prob}
Los siguientes ejercicios dan un ejemplo muy b\'asico de como deducir un
enunciado sobre los enteros de un enunciado probabil\'istico general
sobre sumas de variables aleatorias.

\begin{enumerate}
 \item Pruebe la \emph{desigualdad de Chebyshev} para una variable aleatoria $X$:
 \[
  \mathbb P(|X-\mu|\ge \lambda\sigma^2)\le \frac{1}{\lambda^2}
 \]
para todo $\lambda\in \mathbb R$, con $\mu=\mathbb E[X]$ la esperanza de $X$ y $\sigma^2=\mathbb E[(X-\mu)^2]$ su varianza.

\item \label{ej:divisores_primos} En este ejercicio queremos demostrar que en $[1,x]$ casi todo entero (es decir, todos, salvo $o(x)$ de ellos) tiene  $(1+o(1))\log\log x$ divisores primos distintos.

  Podemos asumir que $x$ es entero.
  Sea $N$ una variable aleatoria en el espacio $\Omega=\{n\le x\}$ tal que $\mathbb P(N=n)=1/x$ para todo $n\in\Omega$.
Para todo primo $p\leq D=x^{1/4}$, considere la variable aleatoria $X_p=1_{N\equiv 0 (p)}$. 
\begin{enumerate}
 \item Demuestre que $X_p$ es una variable de Bernoulli con media $\mu_p= \frac 1p+O(x^{-1})$ y por lo tanto con varianza $\sigma_p^2=\mu_p(1-\mu_p)=\frac{1}{p}(1-\frac 1p)+O(x^{-1})$.
 \item Pruebe que para $p\neq q$ las variables $X_p$ y $X_q$ tienen covarianza casi nula: para $Y_p=X_p-\mu_p$ tenemos $\mathbb E[Y_p Y_q]\ll x^{-1}$.
 \item Muestre que la variable $w=\sum_{p\leq D} X_p$ tiene media
   $\mu=\sum_{p\leq D} \mu_p$ y que su varianza
   $\sigma^2=\mathbb E[(\sum_{p<D} Y_p)^2]$ satisface $\sigma^2=\sum_{p\leq D} \sigma_p^2+O(x^{-1/2})$. (En otras palabras, las variables $X_p$ se comportan
   como si fueran independientes.)
 
 \item Usando el teorema de los n\'umeros primos (o, para ser precisos,
   (\ref{eq:suminvp})), concluya que $\sigma^2=\mu+O(1)=\log\log x+O(1)$.
 
 \item Use la desigualdad de Chebyshev para demostrar el enunciado del ejercicio, observando que un entero en $[1,x]$ tiene a lo sumo 3 divisores primos mayores
   que $D$.
 
\end{enumerate}

\end{enumerate}

\subsection{Entrop\'ia e informaci\'on mutua}

Ahora comenzamos la labor de mostrar que, para alg\'un $H$, las variables aleatorias $X_H$ e $Y_H$ no son muy dependientes.  Esto lo mediremos mediante el concepto de  <<informaci\'on mutua>>, que pasamos a definir.
  
  \subsubsection{Definiciones}
  
  Sea $X$ una variable aleatoria con un n\'umero finito de valores posibles $x$.
  La {\em entrop\'ia} $\mathbb{H}(X)$ de $X$ es
  \[\mathbb{H}(X) = - \sum_x p_x \log p_x,\]
  donde $p_x$ es la probabilidad $\mathbb{P}(X=x)$ de que $X$ tome el valor $x$.
  La {\em entrop\'ia condicional} de $X$ con respecto a una variable
  aleatoria $Y$ (que suponemos tener tambi\'en un n\'umero finito de valores, o
  por lo menos ser discreta) es
  \begin{equation}\label{eq:defentcond}
    \mathbb{H}(X|Y) = \sum_y \mathbb{H}(X|Y=y) \mathbb{P}(Y=y),\end{equation}
  donde, para $E$ un evento probabil\'istico (como $Y=y$),
  $\mathbb{H}(X|E)$ se define por
  $\mathbb{H}(X|E) = -\sum_x p_{x,E} \log p_{x,E}$, donde
  $p_{x,E} = \mathbb{P}(X=x|E)$.

  Las siguientes propiedades b\'asicas son f\'aciles de probar (ver los ejercicios). Escribimos $\mathbb{H}(X,Y)$ para denotar la entrop\'ia
  $\mathbb{H}((X,Y))$ de la variable aleatoria $(X,Y)$, donde $X$ e $Y$ son
  variables aleatorias.
  \begin{enumerate}
  \item\label{it:fernet0} La entrop\'ia $\mathbb{H}(X)$ y la entrop\'ia
    condicional $\mathbb{H}(X|Y)$ son no negativas.
  \item\label{it:ferneta} $
    \mathbb{H}(X,Y) = \mathbb{H}(X|Y) + \mathbb{H}(Y)
    = \mathbb{H}(Y|X) + \mathbb{H}(X)$,
  \item\label{it:fernetb} $\mathbb{H}(X|Y) \leq \mathbb{H}(X)$, 
  \item\label{it:fernetc} $\mathbb{H}(X,Y)\leq \mathbb{H}(X) + \mathbb{H}(Y)$ (subaditividad  de la entrop\'ia).
  \item\label{it:fernetd} Si $X$ toma $\leq N$ valores distintos, $\mathbb{H}(X)\leq \log N$.
  \end{enumerate}
  Tambi\'en podemos acotar con facilidad la diferencia entre las entrop\'ias
  de dos variables aleatorias $X$, $Y$ que toman los mismos valores con
  probabilidades distintas (ejercicio \ref{ej:diffentr}).
  
  Definimos la {\em informaci\'on mutua} $\mathbb{I}(X,Y)$:
  \begin{equation}\label{eq:defI}
    \mathbb{I}(X,Y) = \mathbb{H}(X) + \mathbb{H}(Y) - \mathbb{H}(X,Y).
    \end{equation}
  Por la subaditividad de la entrop\'ia, $\mathbb{I}(X,Y)\geq 0$.
  Est\'a claro por \ref{it:ferneta}.\ que
  \begin{equation}\label{eq:defI2}
    \mathbb{H}(X|Y)
    = \mathbb{H}(X) -  \mathbb{I}(X,Y),
    \;\;\;\;\;\;\;\;  \mathbb{H}(Y|X)
    = \mathbb{H}(Y) -  \mathbb{I}(X,Y).
  \end{equation}
  \subsubsection{El argumento por agotamiento de informaci\'on mutua}

  Consideramos ahora las variables aleatorias $X_H, Y_H$ definidas en \eqref{eq:defXY} en t\'erminos de la variable aleatoria $N$ cuya distribuci\'on
  fue dada en (\ref{eq:definN}). La meta es mostrar que existe un $H \in \lbrack H_-,H_+\rbrack$
  (donde $H_-$, $H_+$ ser\'an especificados m\'as tarde) tal que
  $\mathbb{I}(X_H,Y_H)$ es peque\~na en comparaci\'on con $H$.
  
  Podemos definir para $H_1$, $H_2$ arbitrarios,
  \begin{equation}\label{eq:defXdisp}X_{H_1,H_1+H_2} = (\lambda(N+j))_{H_1<j\leq H_1+H_2}.\end{equation}
  Podemos asumir sin p\'erdida de generalidad que $w\leq x^{1/8}$.
  Tambi\'en asumiremos que $H_1,H_2\leq x^{1/8}$.
   Entonces, por el ejercicio \ref{ej:despla},
  \[\mathbb{H}(X_{H_1,H_1+H_2}) = \mathbb{H}(X_{H_2}) + O\left(1/x^{5/8}\right).\]
  
  Por la subaditividad de la entrop\'ia, deducimos que
  \begin{equation}\label{eq:subadcon}
    \mathbb{H}(X_{H_1+H_2}) \leq \mathbb{H}(X_{H_1}) + 
  \mathbb{H}(X_{H_1,H_1+H_2}) \leq \mathbb{H}(X_{H_1}) +  \mathbb{H}(X_{H_2}) +
  O(1/x^{5/8}).\end{equation}

  Vemos por el ejercicio \ref{ej:despcond} (suponiendo que $\delta\asymp (\prod_{K_0<p\leq K_1} p)^{-1}$ satisface $\delta^{-1}\leq x^{1/8}$) que 
  el mismo razonamiento vale para la entrop\'ia condicional:
  \[\mathbb{H}(X_{H_1,H_1+H_2}|(N+H_1 \mo p)_{K_0<p\leq K_1}) =
  \mathbb{H}(X_{H_2}|(N \mo p)_{K_0<p\leq K_1}) + O(1/\sqrt{x}).\]
Por el teorema de los
  n\'umeros primos (en la forma \cite[Thm.~9]{MR0137689}),
  \begin{equation}\label{eq:sizprod}
    \prod_{K_0<p\leq K_1} p \leq e^{\sum_{p\leq K_1} \log p}
  \leq e^{1.02 K_1}.\end{equation}
  Supondremos entonces que $K_1\leq (\log x)/9$.
  
  Ahora bien, $N+H_1 \mo p$ codifica la misma informaci\'on que
  $N\mo p$. Por lo tanto,
  \[\mathbb{H}(X_{H_1,H_1+H_2}|(N \mo p)_{K_0<p\leq K_1}) =
  \mathbb{H}(X_{H_1,H_1+H_2}|(N+H_1 \mo p)_{K_0<p\leq K_1})\]
  Recordamos la definici\'on (\ref{eq:defXY}) de $Y_H$. De nuevo por
  subaditividad, como en (\ref{eq:subadcon}), conclu\'imos que
  \[\mathbb{H}(X_{H_1+H_2}|Y_H) \leq \mathbb{H}(X_{H_1}|Y_H) +
  \mathbb{H}(X_{H_2}|Y_H) + O(1/\sqrt{x}).\]
  Iterando con $H_1 = H, 2 H, 3 H,\dotsc$ y $H_2=H$, vemos que
  \[
  \mathbb{H}(X_{k H}|Y_H) \leq k \mathbb{H}(X_{H}|Y_H) + O(k/\sqrt{x})\]
  para $k H\leq x^{1/8}$, y en consecuencia, otra vez por subaditividad,
  \[\mathbb{H}(X_{k H}) \leq \mathbb{H}(X_{k H}|Y_H) + \mathbb{H}(Y_H)
  \leq k \mathbb{H}(X_{H}|Y_H) + \mathbb{H}(Y_H) + O(k/\sqrt{x}),\]
  as\'i que, por (\ref{eq:defI2}),
  \[\frac{\mathbb{H}(X_{k H})}{k H} \leq
  \frac{\mathbb{H}(X_H)}{H} - \frac{\mathbb{I}(X_H,Y_H)}{H}
  + \frac{\mathbb{H}(Y_H)}{k H} + O\left(\frac{1}{\sqrt{x}}\right).\]
  Por la propiedad \ref{it:fernetd} de la entrop\'ia y (\ref{eq:sizprod}),
  \[\mathbb{H}(Y_H) \leq \log\left( \prod_{K_0<p\leq K_1} p\right) \leq  1.02 \epsilon H.\]
  Conclu\'imos que,
  \begin{equation}\label{eq:crucmill}
    \frac{\mathbb{H}(X_{k H})}{k H} \leq
    \frac{\mathbb{H}(X_H)}{H} - \frac{\mathbb{I}(X_H,Y_H)}{H} + 1.02 \frac{\epsilon}{k} + O\left(\frac{1}{\sqrt{x}}\right)
      .\end{equation}

He aqu\'i el sujeto de nuestra met\'afora: la tasa de entrop\'ia
$\mathbb{H}(X_{H'})/H'$ (con $H'=H, kH,\dotsc$) es ``las lentejas'', y la
tasa $\mathbb{I}(X_H,Y_H)/H$ es la cucharada que nos servimos.
Para cuantificar que tan poca lenteja nos terminaremos sirviendo en alg\'un momento,
basta con un sencillo lema.

\begin{lema}\label{lem:divergir}
  Sea $h_1\geq 15$ arbitrario y, para $j\geq 1$,
  $h_{j+1} = \lfloor 4 \log h_j \log_3 h_j\rfloor \cdot h_j$.
  Entonces
  \[
  \sum_{j= 1}^{J} \frac{1}{\log h_j \log_3 h_j}\ge 100
  \]
 para alg\'un $J$ cumpliendo $\log J\ll (\log_2 h_1)^2$. 
\end{lema}
\begin{proof}
  Ejercicio \ref{ej:divergir}.
\end{proof}

\begin{corolario}\label{cor:pocamutua}
  Sean $X_H$, $Y_H$ y $N$ como en (\ref{eq:definN}) y (\ref{eq:defXY}), con
  $K_0 = \epsilon H/2$, $K_1 = \epsilon H$, $0<\epsilon\leq 1$ y
  $w\leq x^{1/8}$.
  Sea $H_-\geq 3$.
  Entonces hay un $H_+>H_-$, dependiendo s\'olo de $H_-$ y cumpliendo $\log_3 H_+\le 2 \log_3 H_-+O(1)$, tal que
  \begin{equation}\label{eq:despocamutua}
    \mathbb{I}(X_H,Y_H) \leq \frac{H}{\log H \log_3 H} \end{equation}
  para alg\'un entero $H\in \lbrack H_-,H_+\rbrack$, con tal que
  $x\geq \exp(H_+^9)$. 
\end{corolario}
Es f\'acil ver que la condici\'on $x\geq \exp(H_+^9)$ se deduce de
$\log_3 H_+ \leq 2 \log_3 H_- + O(1)$ para
$x$ m\'as grande que una constante y
$H_-\leq \exp(\exp(\sqrt{\log_3 x}/C))$, donde $C>0$ es otra constante.
\begin{proof}
Podemos asumir sin p\'erdida de generalidad que
$H_-$ es un entero m\'as grande que una constante apropiada.
  Sea $h_1 = H_-$; definamos $h_2,h_3,\dotsc$ como en el Lema
  \ref{lem:divergir}, y sea $k_j = \lfloor 4 \log h_j \log_3 h_j\rfloor$.
  Utilizando la suposici\'on que $h_j\geq H_-$ es m\'as grande que una
  constante, as\'i como la desigualdad $k_j\leq k_j h_j\leq x^{1/8}$, la cual
  debemos asumir de todas maneras,
  simplificamos  (\ref{eq:crucmill}), obteniendo
  \begin{equation}\label{eq:scum}
    \frac{\mathbb{H}(X_{h_{j+1}})}{h_{j+1}} \leq
    \frac{\mathbb{H}(X_{h_j})}{h_j} - \frac{\mathbb{I}(X_{h_j},Y_{h_j})}{h_j} +
    \frac{1}{2 \log h_j \log_3 h_j}
      .\end{equation}

Deducimos entonces de las propiedades
(\ref{it:fernet0}) y (\ref{it:fernetd}) de la entrop\'ia que
\[\sum_{j\leq J} \left(\frac{\mathbb{I}(X_{h_j},Y_{h_j})}{h_j}
- \frac{1}{2 \log h_j \log_3 h_j}\right)
\leq
\frac{\mathbb{H}(X_{H_-})}{H_-}\leq \log 2.\]

Por el Lema \ref{lem:divergir}, vemos  que hay alg\'un $J$ dependiendo s\'olo de $H_-=h_1$, tal que
\[\frac{\mathbb{I}(X_{h_j},Y_{h_j})}{h_j}
- \frac{1}{2 \log h_j \log_3 h_j} \leq
\frac{1}{2 \log h_j \log_3 h_j}\]
para alg\'un $1\leq j\leq J$. (?`Por qu\'e?) Conclu\'imos que
\[\frac{\mathbb{I}(X_{h_j},Y_{h_j})}{h_j} \leq
\frac{1}{\log h_j \log_3 h_j}.\]
Definimos $H_+ = h_j$ y obtenemos el resultado. La condici\'on
$H_+ \leq (\log x)/9$ implica $H_+\leq x^{1/8}$, y as\'i tambi\'en la suposici\'on
$k_j h_j = h_{j+1}\leq x^{1/8}$ para todo $j\leq J-1$.
\end{proof}

    \subsubsection{Ejercicios}
  \begin{enumerate}
  \item Sean $X$ e $Y$ variables aleatorias que toman un n\'umero finito
    de valores. Pruebe las propiedades (\ref{it:fernet0})--(\ref{it:fernetd})
    de la entrop\'ia. Sugerencias para cada propiedad:
    \begin{enumerate}[(1)]
    \item Use simplemente las definiciones de entrop\'ia y entrop\'ia
      condicional.
    \item De nuevo por las definiciones.
    \item Por la concavidad de $x\mapsto - x \log x$ (primera desigualdad)
      y por el hecho que $\log x$ es creciente (segunda desigualdad; tambi\'en
      se deduce inmediatamente de las propiedades (\ref{it:fernet0})
      y (\ref{it:ferneta})).
    \item Use las propiedades (\ref{it:ferneta}) y (\ref{it:fernetb}).
    \item Por la concavidad de $x\mapsto \log x$.
    \end{enumerate}
  \item\label{ej:diffentr}
    Sean $X$ e $Y$ variables aleatorias con valores en un conjunto $S$ de
    $N$ elementos. Denotemos por $\nu_X$, $\nu_Y$ sus funciones de
    distribuci\'on.
    La {\em distancia de variaci\'on total} $|\nu_X-\nu_Y|_{TV}$ se
    define como $\max_{S'\subset S} |\nu_X(S') - \nu_Y(S')|$.
    (Es f\'acil ver que es igual a $\frac{1}{2} |\nu_X-\nu_Y|_1$.)
    \begin{enumerate}
    \item Muestre que $p=\sum_{x\in S} \min(\nu_X(x),\nu_Y(x))$ es igual a
      $1-|\nu_X-\nu_Y|_{TV}$.
      Si $p=1$, entonces $X$ e $Y$ tienen la misma distribuci\'on, y
      estamos en el caso trivial. Si $p=0$, $X$ e $Y$ tienen soportes
      disjuntos, y la cota que probaremos al final es muy sencilla (mu\'estrelo llegado el momento).  Asumamos de ahora en adelante que $0<p<1$.
    \item Sea $Z$ una variable aleatoria con funci\'on de distribuci\'on
      \[\nu_Z(x) = \frac{1}{p} \min(\nu_X(x),\nu_Y(x)).\] Sean $X'$ e $Y'$ variables
      aleatorias independientes con distribuciones
      \[\nu_{X'}(x) = \begin{cases} \frac{\nu_X(x)-\nu_Y(x)}{1-p} &
        \text{si $\nu_X(x)>\nu_Y(x)$,}\\ 0 & \text{de otra manera,}\end{cases}
      \;\;\;\;\;\;\;
      \nu_{Y'}(x) = \begin{cases} \frac{\nu_Y(x)-\nu_X(x)}{1-p} &
        \text{si $\nu_Y(x)>\nu_X(x)$,}\\ 0 & \text{de otra manera.}\end{cases}\]
      Sea $B$ una variable que toma el valor $0$ con probabilidad $p$
      y el valor $1$ con probabilidad $1-p$.
      Construyamos la variable aleatoria $C$ de la siguiente forma:
      si $B=0$, $C$ toma el valor $(Z,Z)$; si $B=1$, $C$ toma el valor
      $(X',Y')$. Muestre que la primera coordenada $C_1=\pi_1(C)$ de $C$ es una
      variable con distribuci\'on $\nu_X$, mientras que la segunda
      coordenada $C_2=\pi_2(C)$ tiene distribuci\'on $\nu_Y$. Muestre tambi\'en
      que $\mathbb{P}(C_1\ne C_2) = |\nu_X-\nu_Y|_{TV}$. Se dice que la
      variable $C$ es un {\em acoplamiento \'optimo} de $X$ e $Y$.
    \item Por las propiedades (\ref{it:fernetb}) de la entrop\'ia, 
      \[\mathbb{H}(C_1|B) \leq \mathbb{H}(C_1) \leq
      \mathbb{H}(C_1|B) + \mathbb{H}(B).
      \]
      Sabemos que
      \[\mathbb{H}(C_1|B) = p \mathbb{H}(C_1|B=0) + (1-p)
      \mathbb{H}(C_1|B=1) = p \mathbb{H}(Z) + (1-p) \mathbb{H}(X').\]
      De la misma manera,
      \[\mathbb{H}(C_2|B) = p \mathbb{H}(Z) + (1-p) \mathbb{H}(Y').\]
      Por lo tanto,
      \[\begin{aligned}\left|\mathbb{H}(X)- \mathbb{H}(Y)\right|
     &= \left|\mathbb{H}(C_1)- \mathbb{H}(C_2)\right|\\
      &\leq \mathbb{H}(B) + \left|\mathbb{H}(C_1|B)- \mathbb{H}(C_2|B)\right|\\
      &\leq \mathbb{H}(B) + (1-p) \left|\mathbb{H}(X') - \mathbb{H}(Y')\right|
      \\ &\leq \mathbb{H}(B) + (1-p) \min(\mathbb{H}(X'),\mathbb{H}(Y')).
      \end{aligned}\]
    \item
      Concluya, por la propiedad (\ref{it:fernetd}) de la entrop\'ia, que
      \[\left|\mathbb{H}(X)- \mathbb{H}(Y)\right| \leq
       |\nu_X-\nu_Y|_{TV} \cdot \log N + \mathbb{H}(B).\]
      En verdad, podemos dar una cota ligeramente menor: la variable $X'$
      tiene soporte en un subconjunto estricto de $S$ (?`por qu\'e?) y lo
      mismo es cierto de $Y'$; concluya que
      \begin{equation}\label{eq:diffentr}
        \left|\mathbb{H}(X)- \mathbb{H}(Y)\right| \leq
                |\nu_X-\nu_Y|_{TV}
        \cdot \log(N-1) + \mathbb{H}(B).
        \end{equation}
      Aqu\'i, por supuesto, \[\mathbb{H}(B) = - |\nu_X-\nu_Y|_{TV} \log
      |\nu_X-\nu_Y|_{TV} - (1- |\nu_X-\nu_Y|_{TV}) \log
      (1-|\nu_X-\nu_Y|_{TV}).\]
    \end{enumerate}
  \item\label{ej:desplazar}
    Sea $I=(x_0,x_1\rbrack$ un intervalo en $\mathbb{R}^+$. Sea
    $N$ una variable aleatoria que toma el valor entero
      $n\in I$ con probabilidad $(1/n)/L$, donde
      $L = \sum_{m\in I} 1/m$. Para $h\in \mathbb{Z}^+$, sea $N+h$
      la variable que toma el valor $N+h$ con probabilidad
      $(1/n)/L$. Denote $\nu_X$ la distribuci\'on de probabilidad de una
      variable $X$.
    \begin{enumerate}
    \item Muestre que $|\nu_{N+h}-\nu_N|_{TV} \leq h/x_0 L$.
    \item Sea $f$ una funci\'on sobre $\mathbb{Z}^+$. Deduzca que
      \begin{equation}\label{eq:difftvfP}
        \left|\nu_{f(N+h)}-\nu_{f(N)}\right|_{TV} \leq h/x_0 L.
      \end{equation}
    \item\label{ej:despla}
      Sean $X_H$ y $X_{H_1,H_1+H_2}$ como en (\ref{eq:defXY}) y
      (\ref{eq:defXdisp}).
      Concluya, usando (\ref{eq:diffentr}) y (\ref{eq:difftvfP}), que, para
      $L = \sum_{x/w<n\leq x} 1/n \geq 1$ y $1\leq H_1\leq x_0/2$, $x_0=x/w$,
      \[\begin{aligned}
      \left|\mathbb{H}(X_{H_1,H_1+H_2}) - \mathbb{H}(X_{H_2})\right|&\leq
      \frac{H_1}{x_0 L} \log\left(2^{2 H_2}-1\right) + h\left(\frac{H_1}{x_0 L}\right) 
      \\
      &\leq \frac{H_1}{x_0 L} H_2 \log 4 + h\left(\frac{H_1}{x_0}\right) \leq
      \frac{H_1}{x_0} \left(H_2 \log 4 + 2 \log x_0\right),\end{aligned}\]
      donde $h(\epsilon) = -\epsilon \log \epsilon
      -(1-\epsilon) \log(1-\epsilon) \leq -2 \epsilon \log \epsilon$ para
      $0<\epsilon\leq 1/2$. En particular, para $1\leq H_1, H_2\leq x_0^{\alpha}$,
      $\alpha>0$,
      \[\left|\mathbb{H}(X_{H_1,H_1+H_2}) - \mathbb{H}(X_{H_2})\right|
      \ll_{\alpha} \frac{1}{x_0^{1-2\alpha}}.
      \]
    \item\label{ej:despcond} Sea $S\subset (x_0,x_1\rbrack$ y
      $\delta = \mathbb P(N\in S) = (\sum_{n\in S} 1/n)/L$. Muestre que
      la distancia de variaci\'on total entre las distribuciones
      de probabilidad condicional de $N+h$ y $N$ con la condici\'on $N\in S$
      es $\leq h/\delta x_0 L$. Concluya que, para $1\leq H_1\leq \delta x_0/2$,
      \[\left|\mathbb{H}(X_{H_1,H_1+H_2}| N+H_1\in S) - \mathbb{H}(X_{H_2}|N\in S)\right|\leq
      \frac{H_1}{\delta x_0} \left(H_2 \log 4 + 2 \log x_0\right),\] 
      con $x_0=x/w$ y $x_1=x$.
      En particular, si $1\leq H_1,H_2,\delta^{-1}\leq x_0^\alpha$, 
    $0<\alpha<1$, 
      \[\left|\mathbb{H}(X_{H_1,H_1+H_2}| N+H_1\in S) - \mathbb{H}(X_{H_2}| N\in S)\right|
      \ll_{\alpha} \frac{1}{x_0^{1-3\alpha}}.\]
        \end{enumerate}
  \item\label{ej:divergir}
    Sea $h_j$ como en el Lema \ref{lem:divergir}.
    \begin{enumerate}
    \item Muestre que $\log h_j\leq 2 (j+ \log h_1)\log(j+\log h_1) $ para todo $j\geq  1$.
    \item\label{eq:compdiver} Pruebe que, para $j\geq \log h_1$,
      \[\frac{1}{\log h_j  \log_3 h_j} \ge
       \frac{1/8}{j \log j \log_3 j}.\]
    \item Demuestre que \[\sum_{ \log h_1<j<(\log h_1)^{C\log_2 h_1}} \frac{1}{j \log j \log_3 j}>800\]
     para $C$ suficientemente grande, independiente de $h_1$. (Utilice, por ejemplo, un test de integrales.) Concluya que el
      Lema \ref{lem:divergir} es cierto.
    \end{enumerate}
  \end{enumerate}

\subsection{Informaci\'on mutua y dependencia}

En la secci\'on anterior hemos visto que la informaci\'on mutua entre $X_H$ e $Y_H$ es peque\~na para alg\'un $H$. 
En esta secci\'on vamos a ver c\'omo usar ese hecho para deducir que $X_H$ e $Y_H$ son m\'as o menos independientes. Todo va a sustentarse sobre
nuestra cota para la informaci\'on mutua.


El esquema ser\'ia el siguiente. La variable $Y_H$ para $H$ peque\~no es esencialmente uniforme. Por lo tanto, su entrop\'ia est\'a muy cerca del m\'aximo posible, y, como la informaci\'on mutua con $X_H$ es peque\~na, deducimos que la entrop\'ia $\mathbb H(Y_H | X_H)$ tambi\'en va a estar muy cerca del m\'aximo posible. Esto va a implicar que, con probabilidad casi 1,  $X_H$ toma un valor $\vec x$ para el cual la entrop\'ia de la variable $Y_H$ condicionada a $X_H=\vec x$ est\'a cerca del m\'aximo posible. A su vez, eso deber\'ia implicar que la distribuci\'on de la variable $Y_H$ condicionada a $X_H=\vec x$ est\'a cerca de ser uniforme, por lo que habr\'iamos demostrado esencialmente la independencia de $X_H$ e $Y_H$.

En realidad, este esquema s\'olo demostrar\'a la independencia ``a efectos de esperanza'', es decir,
\[
 \mathbb E (F(X_H,Y_H)) \sim \mathbb E(F(X_H, Y_H^*))
\]
donde $Y_H^*$ es una variable con la misma distribuci\'on que $Y_H$ pero independiente de $X_H$. Empero, esta igualdad aproximada entre esperanzas es justo lo que necesitamos para acotar nuestras sumas.

Como ya hemos indicado (final de \S \ref{subs:sumyesp}),
la equidistribuci\'on de $Y_H$
condicionada a $X_H=\vec x$ que requerimos es bastante d\'ebil: s\'olo
necesitamos mostrar que $Y_H$ tiende a evitar un peque\~no conjunto,
de tama\~{n}o acotado por la desigualdad (\ref{eq:Fexceptional_sets}).
Debido a la forma de (\ref{eq:Fexceptional_sets}), el hecho que
nuestra cota para la informaci\'on (\ref{eq:despocamutua})
es mejor que la cota trivial por 
un factor de algo m\'as que $\log H$ ser\'a crucial.

El primer paso es muy simple:
ver que $Y_H$ es esencialmente uniforme.
Sabemos por \eqref{eq:defXY} que $Y_H$ toma valores en el conjunto
\begin{equation}\label{eq:defOmega}
  \Omega=\prod_{\epsilon H/2<p\le \epsilon H} \mathbb Z/p\mathbb Z .
\end{equation}
Por el teorema chino de los restos, $\Omega$ es isomorfo a $\mathbb Z/M\mathbb Z$, donde
$M = |\Omega|=\prod_{\epsilon H/2<p\le \epsilon H} p$. Por lo tanto,
si $|\Omega|<x/w$ entonces
\[\begin{aligned}
\mathbb P(Y_H=\vec y) &= \mathbb P\left(N\equiv y_p \mo p \;\; \forall p\in
\left(\frac{\epsilon H}2, \epsilon H\right]\right)
  = \frac{1}{L} \mathop{\sum_{\frac{x}{w}<n \le x}}_{n\equiv y_p (p)\forall p\in (\frac{\epsilon H}2, \epsilon H]} \frac{1}{n}
  \\ &= \frac{1}{L}
    \mathop{\sum_{\frac xw<n\le x}}_{n\equiv y_* (|\Omega|)} \frac{1}{n}= \frac{1}{|\Omega|}\frac{\log w+O(\frac 1{x/w|\Omega|})}{\log w+O(\frac 1{x/w})}=\frac{1}{|\Omega|} +O\left(\frac{1}{x/w}\right),
    \end{aligned}\]
donde $L = \sum_{x/w<n\leq n} 1/n$ e $y_*$ es cualquier entero congruente a $y_p$ para cada
$p\in (\epsilon H/2, \epsilon H\rbrack$.
Por otra parte, por el teorema de los n\'umeros primos,
\begin{equation}\label{eq:tamomega}
|\Omega|= e^{\frac{\epsilon H}2} \left(1+ O\left((\log\epsilon H)^{-50}\right)\right) \ll e^{\frac{\epsilon H}{2}}.\end{equation}
Luego, tomando $\epsilon H<\log(x/w)$ vemos que
\begin{equation}\label{eq:Yuniforme}
 \mathbb{P}(Y_H=\vec y)=\frac{1}{|\Omega|}\left(1+O\left(\frac{1}{|\Omega|}\right)\right),
\end{equation}
es decir, la distribuci\'on de $Y_H$ es casi uniforme.

Deducimos mediante la cota general (\ref{eq:diffentr}) que
\begin{equation}\label{eq:entropiaY}
  \mathbb H(Y_H)= 
  \left(1+O\left(\frac{1}{|\Omega|}\right)\right) \log |\Omega|,
\end{equation}
(El m\'aximo posible ser\'ia $\log |\Omega|$.)

Apliquemos ahora el Corolario \ref{cor:pocamutua}, con
$H_{-}$ de forma que $H_{+}\le \frac{1}{9}\log x$. Obtenemos que la
informaci\'on mutua con $X_H$ es peque\~na para cierto $H\in [H_{-},H_{+}]$, y, por \eqref{eq:defI2}, conclu\'imos que
\[
 \mathbb H(Y_H|X_H)\ge \mathbb{H}(Y_H)-\frac{H}{\log H\log_3 H}
\]
para dicho $H$.
As\'i, usando \eqref{eq:entropiaY} y la cota \eqref{eq:tamomega}, vemos que
\begin{equation}\label{eq:entropiaYXgrande}
  \mathbb H(Y_H | X_H)\ge 
  \left(1- \frac{4}{\epsilon \log H \log_3 H}\right) \log |\Omega|
\end{equation}
para $H_-$ m\'as grande que una constante que depende s\'olo de
$\epsilon$. En otras palabras, para alg\'un
$H\in [H_{-},H_{+}]$, la entrop\'ia condicional $\mathbb{H}(Y_H|X_H)$
est\'a cerca del m\'aximo posible.

Digamos que un elemento $\vec x\in\{-1,1\}^H$ es \emph{bueno} si
\begin{equation}\label{eq:entropia_bueno}
  \mathbb H(Y_H | X_H=\vec x) \ge 
  \left(1- \frac{4}{(\epsilon \log_3 H)^{3/4} \log H}\right)
  \log |\Omega| ,
\end{equation}
y en otro caso decimos que $\vec x$ es \emph{malo}. Por la definici\'on \eqref{eq:defentcond} de $\mathbb H(Y_H | X_H)$ en t\'erminos de
$\mathbb H(Y_H |X_H=\vec x)$,  y por la desigualdad \eqref{eq:entropiaYXgrande}, vemos que 
\[\begin{aligned}
\left(1-\frac{4}{\epsilon\log_3 H\log H} \right) \log |\Omega|
&\le \sum_{\vec x \text{ malo}} 
\left(1- \frac{4}{(\epsilon \log_3 H)^{3/4} \log H}\right) \log |\Omega| \cdot
\mathbb P(X=\vec x) \\ &+ \sum_{\vec x \text{ bueno}} \log |\Omega| \cdot \mathbb P(X=\vec x)
\end{aligned}\]
ya que $\mathbb H(Y_H | X_H= \vec x)$ es como m\'aximo $\log |\Omega|$. De esta desigualdad deducimos la cota
\begin{equation}\label{eq:prob_malo}
 \mathbb P(X_H \text{ malo})\le \frac{1}{(\epsilon\log_3 H)^{1/4}}.
\end{equation}
Luego, es muy probable que $X_H$ sea igual a un valor bueno de $\vec x$, i.e.,
un valor de $\vec{x}$ tal que la entrop\'ia de la variable $Y_H$ condicionada a $X_H=\vec x$ estar\'a cerca del m\'aximo posible.

Para continuar, querr\'iamos ver que la \'unica distribuci\'on que est\'a cerca de alcanzar el m\'aximo de entrop\'ia es la uniforme. Esto ser\'ia cierto si su entrop\'ia est\'a extremadamente cerca de dicho m\'aximo, pero cuando hay un poco m\'as de diferencia dicha unicidad no es cierta (ejercicio
\ref{ej:notanuniforme}).
Aun as\'i, se puede decir que la variable no concentra mucho su masa, en el siguiente sentido.

\begin{lema}\label{le:no_concentracion}
Sea $Y$ una variable aleatoria tomando valores es un conjunto $\Omega$ finito tal que $\mathbb H(Y)\ge (1-\delta)\log |\Omega|$, con $\frac{1}{\log|\Omega|}\le \delta<1$. Si un subconjunto $E\subset \Omega$ tiene tama\~no
\begin{equation}\label{eq:cottamE}
 |E|\leq |\Omega|^{1-M\delta}
\end{equation}
para alg\'un $M>0$, entonces
\[
 \mathbb P(Y\in E)\le \frac{2}{M}.
\]
\end{lema}
\begin{proof}
  Por la concavidad de la funci\'on $x\to -x\log x$, el ejercicio
  \ref{it:concavi} nos da
 \begin{equation}\label{eq:concavidad_dos}
  \mathbb H(Y)\le  \mathbb P(E)\log\frac{|E|}{\mathbb P(E)} + \mathbb P(E^c)\log\frac{|E^c|}{\mathbb P(E^c)}
 \end{equation}
 donde $\mathbb P(A)=\mathbb P(Y\in A)$. Escribiendo $p=\mathbb P(E)$ y usando
 $\mathbb H(Y)\ge (1-\delta)\log |\Omega|$, vemos que
\[
 (1-\delta)\log |\Omega| \le p\log |E| +(1-p)\log |\Omega| -p\log p-(1-p)\log (1-p).
\]
Ahora, usando de nuevo la concavidad con los dos \'ultimos sumandos,
as\'i como la cota \eqref{eq:cottamE} sobre el tama\~no de  $E$, llegamos a 
\[
 (1-\delta)\log|\Omega|\le p(1-M\delta) \log |\Omega|+(1-p)\log|\Omega|+\log 2,
\]
lo cual da
\[
 \delta(pM-1)\log |\Omega|\le \log 2.
 \]
 Luego, si se cumpliera $p>2/M$, tendr\'iamos $\delta\le \log 2/\log |\Omega|$,
 en contradicci\'on a una de las hip\'otesis del lema.
\end{proof}

El lema anterior indica que, para un $\vec x$ bueno, la variable $Y_H$ condicionada a $X_H=\vec x$ tiene una distribuci\'on que no concentra mucho su masa.
\'Esta es una forma muy d\'ebil de cuasiuniformidad, 
pero va a ser suficiente para estimar la esperanza de $F(X_H,Y_H)$, debido a \eqref{eq:Faverage_behaviour} y \eqref{eq:Fexceptional_sets}.

\begin{teorema}\label{te:esperanza_estimacion}
  Existe un $c>0$ tal que lo siguiente se cumple.
  Sean $x>e^e$, $1\le w\leq x^{1/8}$ y $0<\epsilon\leq 1$.
  Sea $F$ como en (\ref{eq:defF}).
  Para todo
  $3\leq H_-\leq \exp(\exp(c \sqrt{\log_3 x}))$,
  hay un $H>H_-$, dependiendo s\'olo de $H_-$ y cumpliendo
  $\log_3 H\le 2 \log_3 H_-+O(1)$,
  tal que, para
$X_H$, $Y_H$, $N$ y $\Omega$ definidos como en (\ref{eq:definN}), (\ref{eq:defXY}) y (\ref{eq:defOmega}) con
  $K_0 = \epsilon H/2$ y $K_1 = \epsilon H$,
\[
\mathbb E(F(X_H,Y_H)) =\mathbb E(F(X_H,Y_H^*))+
O\left( \frac{H/\log H}{(\epsilon\log_3 H)^{1/4}}\right)
\]
para $Y_H^*$ una variable aleatoria con distribuci\'on uniforme en $\Omega$ e independiente de $X_H$.
\end{teorema}
\begin{proof}
Podemos asumir que $x$ y $\epsilon \log_3 H_-$ son m\'as grandes que una constante, ya que si no el resultado es trivial.
Aplicamos el Corolario \ref{cor:pocamutua} para obtener un $H_+$ con
$\log_3 H_+\le 2 \log_3 H_-+O(1)$ y 
cierto $H\in [H_-,H_+]$ tal que (\ref{eq:despocamutua}) se cumple, i.e.,
tal que la informaci\'on mutua entre $X_H$ e $Y_H$ es peque\~na.
Por la definici\'on de esperanza,
\[
 \mathbb E(F(X_H,Y_H))=\mathop{\sum_{\vec x\in \{-1,1\}^H}}_{\vec y\in \Omega} F(\vec x,\vec y) \, \mathbb P(X_H=\vec x, Y_H=\vec y)
\]
\[
 =\sum_{\vec x\in \{-1,1\}^H}    \mathbb P(X_H=\vec x) \sum_{\vec y\in \Omega} F(\vec x,\vec y)\,  \mathbb P(Y_H=\vec y | X_H=\vec x).
\]
Ahora dividimos el rango de $\vec x$ entre buenos y malos, acorde a la definici\'on previa a \eqref{eq:entropia_bueno}.
Acotamos la suma sobre los $\vec{x}$ malos f\'acilmente:
\[
 \sum_{\vec x \text{ malo}}    \mathbb P(X_H=\vec x) \sum_{\vec y\in \Omega} F(\vec x,\vec y)\,  \mathbb P(Y_H=\vec y | X_H=\vec x)\ll \|F\|_{\infty} \mathbb P(X_H \text{ malo})\ll \frac{H/\log H}{(\epsilon \log_3 H)^{1/4}}
\]
por \eqref{eq:max_F} y \eqref{eq:prob_malo}. Ahora, para cada $\vec x$ bueno, dividimos el rango de $\vec y$ en $E_{\vec x}$ y $E_{\vec x}^c$, con $E_{\vec x}$ el conjunto de excepciones en \eqref{eq:Fexceptional_sets} con $\mu^2=(\epsilon \log_3 H)^{-1/2}$. Teniendo en cuenta \eqref{eq:entropia_bueno}, aplicamos el Lema \ref{le:no_concentracion} para obtener
\[
 \mathbb P(Y\in E_{\vec x} |X_H=\vec x)\ll \frac{1}{(\epsilon\log_3 H)^{1/4}},
\]
de donde
\[
  \sum_{\vec x \text{ bueno}}    \mathbb P(X_H=\vec x) \sum_{\vec y\in E_{\vec x}} F(\vec x,\vec y)\,  \mathbb P(Y_H=\vec y | X_H=\vec x)\ll \frac{\|F\|_{\infty} }{(\epsilon\log_3 H)^{1/4}}\ll \frac{H/\log H }{(\epsilon\log_3 H)^{1/4}}.
\]
As\'i, s\'olo nos quedan los $\vec y$ no excepcionales y los $\vec x$ buenos. En este caso usamos la estimaci\'on \eqref{eq:Faverage_behaviour}, y teniendo en cuenta las cotas ya obtenidas vemos que
\[
 \mathbb E(F(X_H,Y_H))= O\left(\frac{H/\log H}{(\epsilon \log_3 H)^{1/4}}\right)+\sum_{\vec x \text{ bueno}}    \mathbb P(X_H=\vec x) \sum_{\vec y\not\in E_{\vec x}} \overline{F(\vec x,\cdot)}\,  \mathbb P(Y_H=\vec y | X_H=\vec x).
\]
Ahora usamos de nuevo las cotas para $\mathbb P(Y\in E_{\vec x}| X_H=\vec x)$, $\mathbb P(X_H \text{ malo})$ y $\|F\|_{\infty}$ para suplir los $(\vec x,\vec y)$ que faltan y as\'i obtener 
\[
\mathbb E(F(X_H,Y_H))= O\left(\frac{H/\log H}{(\epsilon \log_3 H)^{1/4}}\right)
+\sum_{\vec x \in \{-1,1\}^H}    \mathbb P(X_H=\vec x) \sum_{\vec y\in\Omega} \overline{F(\vec x,\cdot)}\,  \mathbb P(Y_H=\vec y | X_H=\vec x).
\]
Pero
\[\begin{aligned}
 \sum_{\vec x \in \{-1,1\}^H}    &\mathbb P(X_H=\vec x) \sum_{\vec y\in\Omega} \overline{F(\vec x,\cdot)}\,  \mathbb P(Y_H=\vec y | X_H=\vec x)=\sum_{\vec x \in \{-1,1\}^H}      \mathbb P(X_H=\vec x)  \overline{F(\vec x,\cdot)}\\
 = &\sum_{\vec x \in \{-1,1\}^H}      \mathbb P(X_H=\vec x)  \sum_{\vec y\in \Omega}F(\vec x,\vec y) \frac{1}{|\Omega|}=\mathbb E(F(X_H,Y_H^*)).
\end{aligned}\]

\end{proof}

Ahora, por \eqref{eq:Fcomosuma},  tenemos que la esperanza del Teorema \ref{te:esperanza_estimacion} se puede escribir como
\[
 \mathbb E(F(X_H,Y_H^*))=\sum_{\epsilon H/2<p\le \epsilon H} \mathbb E(F_p(X_H, (Y_H^*)_p))
\]
donde $Y_H^*=((Y_H^*)_p)_{\epsilon H/2<p\le \epsilon H}$. Como $(Y_H^*)_p$ es independiente de $X_H$ y uniforme en $\mathbb Z/p\mathbb Z$ tenemos
\[\begin{aligned}
 &\mathbb E(F_p(X_H, (Y_H^*)_p))=\sum_{t\in \mathbb Z/p\mathbb Z} \mathbb E(F_p(X_H, t)) \mathbb P((Y_H^*)_p=t)=\frac 1p \sum_{t\in \mathbb Z/p\mathbb Z} \mathbb E(F_p(X_H, t)) \\
 = &\frac 1p \sum_{t\in \mathbb Z/p\mathbb Z} \sum_{\vec x\in \{-1,1\}^H} F_p(\vec x,t) \mathbb P(X_H=\vec x)=\frac 1p \sum_{\vec x\in\{-1,1\}^H} \mathbb P(X_H=\vec x)   \sum_{j\le H-p} x_j x_{j+p} \sum_{t\in \mathbb Z/p\mathbb Z} 1_{j\equiv -t (p)}.
\end{aligned}\]
Luego 
\[
 \mathbb E(F(X_H,Y_H^*))=\mathbb E(G(X_H))
\]
con 
\[
 G(\vec x)= \sum_{\epsilon H/2<p\le\epsilon H} \frac 1p \sum_{j\le H-p} x_j x_{j+p}
\]
y por la definici\'on de $X_H$ 
\[
 \mathbb E(G(X_H))=\sum_{\frac xw<n\le x} \frac{1}{nL}  \sum_{\epsilon H/2<p\le\epsilon H} \frac 1p \sum_{j\le H-p} \lambda(n+j)\lambda(n+j+p).
\]
Ahora, procediendo como en la prueba del lema \ref{le:trozos_long_H}, podemos reescribir esta suma como
\[
\frac{1}{L}\sum_{\epsilon H/2<p\le \epsilon H} \frac 1p \left( H
\sum_{x/w<n\le x}    \frac{\lambda(n)\lambda(n+p)}{n}
+ O\left(H+p \log w\right)\right)
\]
por lo que finalmente, por el teorema \ref{te:esperanza_estimacion}
(con $H_- = \exp(\exp(\min(\log_3 x,\log_2 w)^{1/3}))$, digamos)
y el lema \ref{le:suma_como_esperanza}, obtenemos el resultado
siguiente, el cual es lo que quer\'iamos.
\begin{corolario}\label{co:ind_lambda_div}
  Sea $w\le x^{1/8}$, $w$ m\'as grande que una constante. Sea
  $(\log_3 w)^{-1}\leq \epsilon\leq 1$.
  Entonces existe $H$ con
  $\log_3 H\gg \min(\log_4 x,\log_3 w)$ y $\log_3 H\leq (3/4) \log_3 w$
  tal que
\begin{equation}\label{eq:cotaindlamd}\begin{aligned}
\sum_{\epsilon H/2<p\le \epsilon H} \mathop{\sum_{x/w<n\le x}}_{p\mid n} \frac{\lambda(n)\lambda(n+p)}{n} &- \sum_{\epsilon H/2<p\le \epsilon H} \frac 1p \sum_{x/w<n\le x} \frac{\lambda(n)\lambda(n+p)}{n}\\
&= O\left(\epsilon+\frac{1}{(\epsilon\log_3 H)^{1/4}}\right) \cdot \frac{\log w}{\log H}.
\end{aligned}
\end{equation}
\end{corolario}
La cota superior sobre $\log_3 H$ (la cual, por cierto, es de lejos
m\'as fuerte de lo que necesitamos) nos ser\'a \'util cuando tengamos
que aplicar el Lema \ref{le:truco_divisibilidad} y el Corolario
\ref{cor:promedio_sobre_primos2}.
\subsubsection{Ejercicios}

\begin{enumerate}

\item\label{ej:notanuniforme} Sea $\Omega$ un conjunto finito y $E\subset \Omega$ con $|E|=|\Omega|^{1-\delta}\le |\Omega|/2$. Sea $Y$ la variable aleatoria en $\Omega$ que satisface $\mathbb P(Y\in E)=\mathbb P(Y \in E^c)=1/2$ y tal que $Y$ es uniforme en $E$ y tambi\'en en $E^c$. Observe que $Y$ est\'a muy lejos de ser uniforme en $\Omega$ pero que sin embargo su entrop\'ia es grande:
\[
 \mathbb H(Y)=(1-\delta)\log|\Omega|+O(1).
\]

\item\label{it:concavi}
  Una funci\'on c\'oncava $f$ cumple $f((1-t)a+tb)\ge (1-t)f(a)+tf(b)$. 
\begin{enumerate}
\item Demuestre la desigualdad
 \begin{equation}\label{eq:difprom}
 \frac{f(a_1)+f(a_2)+\ldots +f(a_n)}{n} \le
 f\left(\frac{a_1+a_2+\ldots+a_n}{n}\right)
 \end{equation}
 para cualquier funci\'on c\'oncava y $n\in\mathbb N$.

 \item Sea $f:I\to \mathbb{R}$ una funci\'on doblemente diferenciable
  en un intervalo $I\subset \mathbb{R}$. Asuma que $f''(t)\leq 0$ para todo
  $t\in I$. Muestre que $f$ es c\'oncava.

\item Muestre que la funci\'on
  $x\mapsto x \log(1/x)$ es c\'oncava en $(0,\infty)$, y por
  lo tanto en $\lbrack 0,\infty)$ (definiendo que el valor de
  $x \log(1/x)$ para $x=0$ es $0$).
  Usando la desigualdad (\ref{eq:difprom}), pruebe que  
 \[
  \sum_{y\in A} \mathbb P(Y=y) \log \frac{1}{\mathbb P(Y=y)} \le   \mathbb   P(Y\in A)\log\frac{|A|}{\mathbb P(Y\in A)} 
 \]
para cualquier variable aleatoria $Y$ sobre $\Omega$ y subconjunto $A\subset \Omega$.

\item Demuestre la desigualdad \eqref{eq:concavidad_dos}.

\end{enumerate}

\item Sea $Y$ una variable aleatoria en $\Omega$ tal que $\mathbb H(Y)=\log|\Omega|-o(1/|\Omega|)$.
\begin{enumerate}
 \item Sea $y\in Y$. Muestre usando \eqref{eq:concavidad_dos} con $E=\{y\}$ que 
 \[
  \log|\Omega|-o\left(\frac{1}{|\Omega|}\right)\le -p\log p +(1-p)\log \frac{|\Omega|-1}{1-p}
 \]
con $p=\mathbb P(Y=y)$, y que, por lo tanto,
\begin{equation}\label{eq:eraparteb}
 p\log|\Omega|+\frac{1-p}{|\Omega|}-o\left(\frac{1}{|\Omega|}\right)\le -p\log p-(1-p)\log(1-p).
\end{equation}
\item Utilizando que la parte derecha de (\ref{eq:eraparteb}) est\'a acotada por $\log 2$, pruebe que $p=o(1)$, y adem\'as, usando tambi\'en que $p=o(1)$, pruebe que
\[
 t\log t+1-o(1)\le t (1+o(1))
\]
con $t=p|\Omega|$.
\item Demuestre que para $t\neq 1$ tenemos $t\log t +1 -t>0$. Concluya que
\[
 p= \frac{1}{|\Omega|}(1+o(1)),
\]
es decir, $Y$ se comporta asint\'oticamente como una variable uniforme.

\end{enumerate}

\item Sea $w(n)$ el n\'umero de divisores primos distintos de $n$. Demuestre que 
 \[
 S=\frac 1x\sum_{n\le x} w(n) \lambda(n+1) = o(\log\log x) = 
 o\left(\frac 1x\sum_{n\le x} w(n)\right)
 \]
 usando el Teorema \ref{te:TNP_liouville} y
 el ejercicio \ref{ej:divisores_primos} de \ref{se:ejercicios_prob}. Para ponerlo de otra manera,
 \[
  S=\mathbb E(w\lambda_+)+O(1)
 \]
 con $w$ la variable $w=\sum_{p\le x^{1/4}} 1_{N\equiv 0(p)}$ y $\lambda_+=\lambda(N+1)$, donde $\mathbb P(N=n)=1/x$ para todo $n\in \Omega=[1,x]$.

\item Sean $w,\lambda_{+}$ las variables aleatorias del problema anterior.
\begin{enumerate}
 \item $\lambda_{+}$ toma valores en el conjunto $\{-1,1\}$. Use el Teorema \ref{te:TNP_liouville} para demostrar que su entrop\'ia est\'a cerca de la m\'axima posible, en el sentido que
 \[
  \mathbb H(\lambda_+)=(1+o(1))\log 2.
 \]
\item Demuestre usando el teorema de los n\'umeros primos que $w$ toma valores enteros en un intervalo $\Omega=[0,z)$, con $z\sim \frac{\log x}{\log\log x}$.  Observe que la m\'axima entrop\'ia que $w$ podr\'ia tener es $\log|\Omega|\sim \log\log x$.
\item Use \eqref{eq:concavidad_dos} con $E=\{w<2\log\log x\}$ y el ejercicio
  \ref{ej:divisores_primos} de \ref{se:ejercicios_prob} para demostrar 
 \[
  \mathbb H(w)=o(\log|\Omega|)=o(\log\log x)
 \]
 y que por lo tanto su entrop\'ia est\'a lejos de la m\'axima posible.

\end{enumerate}
\end{enumerate}

  \subsection{Sumas de $\lambda(n) \lambda(n+p)$, en promedio sobre $p$. Conclusi\'on}

Por lo visto en la secci\'on anterior, para obtener \eqref{eq:chowla01} s\'olo nos queda controlar sumas del tipo
\[
 \sum_{p\le h} \sum_{X<n\le 2X} \lambda(n)\lambda(n+p).
\]
Estas son m\'as complicadas que las del Corolario \ref{cor:chowla_promedio_2}
\[
 \sum_{j\le h} |\sum_{X<n\le 2X} \lambda(n)\lambda(n+j)|^2
\]
en el sentido de ahora s\'olo sumamos sobre primos, pero m\'as sencillas ya que no tenemos el cuadrado. Vamos a ver que de nuevo es posible comprenderlas en t\'erminos de los coeficiente de Fourier de $\lambda(n)$ en intervalos cortos. La manera de hacerlo va a ser usar el m\'etodo del c\'irculo, el cual, en breve,
consiste en usar la identidad
\begin{equation}\label{eq:circulo}
 1_{k=0}=\int_0^1 e(k\alpha) \, d\alpha  \qquad  k\in \mathbb Z
\end{equation}
para reescribir una ecuaci\'on aditiva (como $m=n+p$) de forma anal\'itica,
usando los arm\'onicos $e(k\alpha)$.

\begin{lema}\label{le:circulo}
Si $|w_j|\le 1$ tenemos que
\[
\sum_{p\le h} \sum_{j\le h}  w_{j+p} \overline{w_j}\ll \epsilon \frac{h^2}{\log h}+\frac{h^2}{\log h}\int_{\mathfrak M_{\epsilon}} \left|\sum_{b\le h} w_b e(b\alpha)\right| \, d\alpha  
\]
para cualquier $\epsilon\in (0,1)$, con $\mathfrak{M}_{\epsilon}=\{\alpha\in [0,1]: |\sum_{p\le h}e(p\alpha)|>\epsilon \frac{h}{\log h} \}$.
\end{lema}
\begin{proof}
Usando la identidad (\ref{eq:circulo}) para $k=m-j-p$ tenemos
\[
\sum_{p\le h} \sum_{j\le h} \overline{w_j} w_{j+p}=\sum_{m\le 2h} \sum_{p\le h}
\sum_{j\le h} w_m \overline{w_j} \int_0^1 e((m-j-p)\alpha) \, d\alpha,
\]
y sacando fuera la integral y factorizando obtenemos
\[
 \sum_{p\le h} \sum_{j\le h} \overline{w_j} w_{j+p}=\int_0^1 W_{2h}(\alpha) \overline{W_h(\alpha)} \overline{ P_h(\alpha)} \, d\alpha,
\]
con $W_{d}(\alpha)=\sum_{b\le d} w_b e(b\alpha)$ y $P_d(\alpha)=\sum_{p\le d} e(p\alpha)$.
Ahora dividimos la integral entre los <<arcos mayores>> $\mathfrak M_{\epsilon}$ y los <<arcos menores>> $\mathfrak m_{\epsilon}=[0,1)\setminus \mathfrak M_{\epsilon}$. Para la parte de los arcos mayores tenemos
\[
 \int_{\mathfrak M_{\epsilon}} W_{2h}(\alpha) \overline{W_h(\alpha)} P_h(-\alpha) \, d\alpha \ll \frac{h^2}{\log h} \int_{\mathfrak M_{\epsilon}} |W_h(\alpha)|\, d\alpha = \frac{h^2}{\log h}\int_{\mathfrak M_{\epsilon}} |\sum_{b\le h} w_b e(b\alpha)| \, d\alpha.
\]
Para la parte de los arcos menores tenemos que
\[
 \left|\int_{\mathfrak m_{\epsilon}} W_{2h}(\alpha) \overline{W_h(\alpha)} P_h(-\alpha) \, d\alpha \right|\le \frac{\epsilon h}{\log h} \int_0^1 |W_{2h}(\alpha)| |W_h(\alpha)| \, d\alpha.
\]
Usando la desigualdad $|ab|\le |a|^2+|b|^2$ vemos que
\[
 \int_0^1 |W_{2h}(\alpha)| |W_h(\alpha)| \, d\alpha\le \int_0^1 |W_{2h}(\alpha)|^2 \, d\alpha +\int_0^1 |W_h(\alpha)|^2\, d\alpha
\]
y por Parseval (ecuaci\'on (\ref{eq:parseval_toro}))
\[
 \int_0^1 |W_h(\alpha)|^2\, d\alpha =\sum_{j\le h} |w_j|^2 \le h
\]
y lo mismo para $W_{2h}$, luego obtenemos la cota $O(\epsilon h^2/\log h)$ para la parte de los arcos menores.

\end{proof}

Ahora vamos a ver que la parte que queda (arcos menores) es muy peque\~na, por lo que la integral sobre esa parte tambi\'en va a ser peque\~na.

\begin{lema}\label{le:medida_arcos_mayores}
Sea $0<\epsilon<1$, $h>1$. Con las definiciones del lema anterior tenemos que
\[
 |\mathfrak M_{\epsilon}|\ll \frac{1}{\epsilon^4} \frac{1}{h}.
\]
\end{lema}
\begin{proof}
La idea, igual que en la demostraci\'on del Lema \ref{le:excepcionales}, es controlar alg\'un promedio de $P_h(\alpha)=\sum_{p\le h} e(p\alpha)$ para concluir que $\mathfrak M_{\epsilon}$ es un conjunto peque\~no. La media cuadr\'atica no ser\'a suficiente (ver problema \ref{ej:cotacutre}), pero s\'i la potencia cuarta. Tenemos que
\[
 \int_0^1 |P_h(\alpha)|^4 \, d\alpha =\int_0^1 |P_h^2 (\alpha)|^2\, d\alpha =\int_0^1 |\sum_{|j|\le h} (\sum_{\substack{q-p=j \\ p,q\le h}} 1 ) e(j\alpha) |^2 \, d\alpha,
\]
donde $p,q$ se mueven sobre los primos (hemos agrupado las frecuencias $e(q\alpha)\overline{e(p\alpha)}=e((q-p)\alpha)$). Ahora aplicamos Parseval (ecuaci\'on \ref{eq:parseval_toro}) para obtener
\[
 \int_0^1 |P_h(\alpha)|^4 \,d\alpha=\sum_{|j|\le h} |\sum_{\substack{q-p=j\\ p,q\le h}} 1|^2.
\]
Para $j=0$ tenemos que la suma interior es $O(h/\log h)$ por el teorema de los n\'umeros primos. Para el resto de $j$s podemos usar la cota de criba (\ref{eq:primgemcot}) y as\'i obtener
\[
 \sum_{|j|\le h} |\sum_{\substack{q-p=j\\ p,q\le h}} 1|^2\ll \frac{h^2}{(\log h)^2} +\frac{h^2}{(\log h)^4}\sum_{j=1}^h \prod_{p\mid j} \left(1+\frac 1p\right)^C
\]
para alguna constante $C>1$ fija.
Como $\sum_{j=1}^h \prod(1+1/p)^C\ll h$ (ejercicio \ref{ej:produtron}),
obtenemos 
\[
 \int_0^1 |P_h(\alpha)|^4 \,d\alpha\ll \frac{h^3}{(\log h)^4},
\]
de donde se deduce la cota para $|\mathfrak M_{\epsilon}|$.

\end{proof}

Juntando los dos \'ultimos lemas podemos concluir que los promedios de $\lambda(n+p)\lambda(n)$ est\'an controlados por los coeficientes de Fourier de $\lambda(n)$ en intervalos cortos.

\begin{proposicion}\label{pr:primos_y_fourier}
Sea $0<\epsilon<1$ y $1\le h\le \epsilon X$. Entonces
\[
\sum_{p\le h} \sum_{X<n\leq 2X} \lambda(n+p)\lambda(n) \ll \frac{\epsilon hX}{\log h}+\frac{1}{\epsilon^4  \log h} \max_{0\le \alpha\le 1} \int_X^{2X} 
\left|\sum_{x<m\le x+h} \lambda(m)e(m\alpha) \right| \, dx.
 \]
\end{proposicion}
\begin{proof}
Comenzamos por la observaci\'on de que si desplazamos el intervalo de sumaci\'on en $n$ un poco la suma total casi no var\'ia:
\[
 \sum_{X<n\le 2X} \lambda(n+p)\lambda(n) =O(j)+\sum_{X+j<n\le 2X+j} \lambda(n+p)\lambda(n)
\]
\[
 =O(j)+\sum_{X<n\le 2X} \lambda(n+j+p)\lambda(n+j).
\]
Ahora usamos esa ecuaci\'on para todo $j\le h$, obteniendo
\[
  \sum_{X<n\le 2 X } \lambda(n+p)\lambda(n) = O(h) +\frac{1}{h} \sum_{j\le h} \sum_{X<n\le 2 X } \lambda(n+j+p)\lambda(n+j).
\]
As\'i, sumando en $p$ e intercambiando el orden de sumaci\'on, por TNP tenemos
\begin{equation}\label{eq:corput_shift}
  \sum_{p\le h} \sum_{X<n\le 2X} \lambda(n+p)\lambda(n) =
  O\left(\frac{h^2}{\log h}\right) +
  \frac{1}{h} \sum_{X<n\le 2X} \sum_{p\le h}\sum_{j\le h} \lambda(n+j+p)\lambda(n+j).
\end{equation}
Ahora usamos el Lema \ref{le:circulo} con $w_j=\lambda(n+j)$ y llegamos a
\begin{equation}\label{eq:circulo_lambda}
  \sum_{p\le h} \sum_{X<n\le 2X} \lambda(n+p)\lambda(n)\ll \frac{\epsilon h X}{\log h}+\frac{h}{\log h} \int_{\mathfrak M_{\epsilon}} \sum_{X<n\le 2X} |\sum_{b\le h} \lambda(n+b) e(b\alpha)| \, d\alpha.   
\end{equation}
Teniendo en cuenta que 
\[
 \left|\sum_{b\le h} \lambda(n+b) e(b\alpha)\right|=\left|\sum_{n<m\le n+h} \lambda(m) e(m\alpha)\right|
\]
y el Lema \ref{le:medida_arcos_mayores} obtenemos el resultado buscado.
\end{proof}

Usando las cotas para los coeficientes de Fourier de $\lambda(n)$ en intervalos cortos de la secci\'on 4, vemos ahora que hay cancelaci\'on en las sumas de $\lambda(n+p)\lambda(n)$.

\begin{corolario}\label{cor:promedio_sobre_primos}
Sea $\log h\le (\log X)^{1/3}$. Entonces
\[
\sum_{p\leq h}\sum_{X<n\le 2X} \lambda(n)\lambda(n+p) \ll \frac{1}{(\log h)^{\frac{1}{75}-o(1)}} \cdot \frac{h X}{\log h}.
\]
\end{corolario}
\begin{proof}
Por la Proposici\'on \ref{pr:primos_y_fourier} y el Teorema \ref{te:lambda_fourier} tenemos que
\[
 \sum_{p\le h} \sum_{X<n<2X} \lambda(n)\lambda(n+p)\ll \frac{\epsilon h X}{\log h} +\frac{hX}{\epsilon^4 \log h}\cdot \frac{1}{(\log h)^{\frac{1}{15}-o(1)}}
\]
para cualquier $0<\epsilon<1$ y cualquier $1\le h\le \epsilon X$ con
$\log h \le (\log X)^{1/3}$.
Tomando $\epsilon=(\log h)^{-\frac{1}{75}}$ obtenemos el resultado.
\end{proof}

El resultado que necesitamos se deduce f\'acilmente del Corolario
\ref{cor:promedio_sobre_primos}.
\begin{corolario}\label{cor:promedio_sobre_primos2}
Sean $1\leq w\leq \sqrt{x}$ y $1< h\le w$ tal que $\log h\le (\log \frac xw)^{1/3}$. Entonces
\[
\sum_{h/2 < p\leq h} \frac{1}{p} \sum_{x/w<n\leq x} \frac{\lambda(n)\lambda(n+p)}{n}
\ll \frac{1}{(\log h)^{\frac{1}{75}-o(1)}} \cdot \frac{\log w}{\log h}.
\]
\end{corolario}
\begin{proof}
Ejercicio \ref{ej:promedio_correcto}.
\end{proof}

Llegamos a la prueba del resultado principal.
 
\begin{proof}[Prueba del teorema \ref{te:chowla_log_cuant}]
  Podemos suponer que $w\leq x^{1/8}$ (pues podemos reducir a este
  caso subdividiendo el rango $x/w<n\leq x$) y tambi\'en que $w$ es
  m\'as grande que una constante (pues, de lo contrario, la cota
  que debemos probar es trivial).
  Apliquemos el Corolario \ref{co:ind_lambda_div} con
  $\epsilon=(\log_3 H)^{-1/5}$ y el Corolario \ref{cor:promedio_sobre_primos2}
  con $h = \epsilon H$ para obtener que
  \[
  \sum_{\epsilon H/2<p\le \epsilon H} \mathop{\sum_{x/w<n\le x}}_{p\mid n} \frac{\lambda(n)\lambda(n+p)}{n} \ll \frac{1}{(\log_3 H)^{1/5}} \cdot \frac{\log w}{\log H}.\]
  Luego aplicamos el Lema
  \ref{le:truco_divisibilidad} con $K_0 = \epsilon H/2$, $K_1 = \epsilon H$
  para conclu\'ir que
 \[\sum_{\frac{x}{w} <n\leq x} \frac{\lambda(n) \lambda(n+1)}{n} \ll
 \log \epsilon H \cdot \frac{1}{(\log_3 H)^{1/5}} \cdot \frac{\log w}{\log H}
 + \log \epsilon H \ll
 \frac{\log w}{(\log_3 H)^{1/5}}.\]
\end{proof}

\subsubsection{Ejercicios}

\begin{enumerate}
 \item\label{ej:cotacutre} Demuestre que
 \[
  \int_0^1 |\sum_{p<h} e(p\alpha)|^2\, d\alpha=\frac{h}{\log h}(1+o(1)).
 \]
 Deduzca que $|\mathfrak M_{\epsilon}|\ll \frac{\log h}{\epsilon^2 h}$, con $\mathfrak M_{\epsilon}$ definido como en el Lema \ref{le:circulo}.
 Observe que esta cota es peor que la obtenida en el Lema \ref{le:medida_arcos_mayores} para $1/\sqrt{\log h}<\epsilon<1$.
 
\item\label{ej:produtron} Vamos a demostrar que 
\[
 \sum_{j\le h} \prod_{p\mid j} \left(1+\frac 1p\right)^C \ll_C h.
\]
Para ello, pruebe las siguientes desigualdades e identidades,
con $w(d)=\sum_{p:p\mid d} 1$:
\[\begin{aligned}
\sum_{j\le h} \prod_{p\mid j} \left(1+\frac 1p\right)^C&\ll_C \sum_{j\le h} \prod_{p\mid j} \left(1 + \frac{2 C}{p}\right)\\
    &=  \sum_{d\le h}\frac{(2 C)^{w(d)}}{d} \mathop{\sum_{j\le h}}_{d\mid j} 1
    \leq  h \sum_{d=1}^{\infty} \frac{(2 C)^{w(d)}}{d^2} \ll_C h.
\end{aligned}\]

\item Tambi\'en es factible demostrar el Lema \ref{le:medida_arcos_mayores} estudiando la suma $P(\alpha)=\sum_{p\le h}e(p\alpha)$ para todo $\alpha$ y viendo cuando es grande. Es posible ver que el conjunto $\mathfrak M_{\epsilon}$ de los
  valores para los cuales $P(\alpha)$ es grande consiste
 en los $\alpha$ cercanos a racionales con denominador peque\~no.
 En una direcci\'on (mostrar que los $\alpha$ que no estan cerca a tales
 racionales no est\'an en $\mathfrak M_{\epsilon}$), esto no es nada f\'acil;
 se trata de la parte principal de la estrategia de Vinogradov para el
 problema ternario de Goldbach. Veamos como demostrar la otra direcci\'on,
 por lo menos para algunos racionales de denominador peque\~no.

\begin{enumerate}
 \item Demuestre que si $\alpha=\delta/h$, $|\delta|<1$, entonces $|P(\alpha)|=\frac{h}{\log h}(1+o_h(1)+O(\delta))$. Observe que esto demuestra que $|\mathfrak M_{\epsilon}|\gg h$, lo cual coincide con la cota superior que se obtiene en el Lema \ref{le:medida_arcos_mayores} para $\epsilon\gg 1$.
 
 \item Demuestre que, de todos los car\'acteres $\chi$ m\'odulo 5, el \'unico cuya funci\'on $L(s,\chi)$ tiene un polo en $s=1$ es el car\'acter trivial $\chi_0$ que vale 1 para todo n\'umero no divisible por 5. (Sugerencia: para los otros
   car\'acteres $\chi \mo 5$, utilice sumaci\'on por partes para estimar
   $\sum_n \chi(n) n^{-\sigma}$, $\sigma = 1 + \epsilon$, recordando que
   $\sum_{n\leq u} \chi(n) < 5$ para todo $u$ (?`por qu\'e?).)
 
\item Usando el apartado anterior y el Teorema \ref{te:vinogradov_korobov_siegel}, muestre que 
\[
 \sum_{\substack{p\le h, p\equiv b (5)}} 1 =\frac{1}{4}\frac{h}{\log h} (1+o_h(1))=\frac 14 (1+o_h(1)) \sum_{p\le h} 1 
\]
para $b=1,2,3,4$.

 \item Pruebe usando los apartados anteriores que si $\alpha=\frac{1}{5}+\frac{\delta}{h}$ con $|\delta|<1$, entonces
\[
 P(\alpha)=-\frac{1}{4}\frac{h}{\log h} (1+O(\delta)+o_h(1)).
\]

\end{enumerate}

\item\label{ej:promedio_correcto}
  Deseamos mostrar que el Corolario \ref{cor:promedio_sobre_primos2}
  se deduce del Corolario \ref{cor:promedio_sobre_primos}. El procedimiento
  es sencillo y muy general.
  \begin{enumerate}
   \item\label{ej:prom_f} Sea $f:\mathbb{Z}^+\to \mathbb{C}$ arbitrario. Sea
    $F(t) = \sum_{p\leq t} f(p)$. Muestre que
    \[\sum_{\frac{h}{2}<p\leq h} \frac{f(p)}{p} = \int_{h/2}^h \frac{F(t)}{t^2} dt
    +\frac{1}{h} F(h) - \frac{2}{h} F(h/2).\]
  \item\label{ej:prom_g} Sea $g:\mathbb{Z}^+\to \mathbb{C}$ tal que $|g(n)|\leq 1$ para
    todo $n$. Sea $G(t) = \sum_{t<n\leq 2 t} g(n)$. Muestre que, para
    $1\leq x_0\leq x_1$,
    \[\sum_{x_0<n\leq x_1} \frac{g(n)}{n} = \int_{x_0}^{x_1} \frac{G(t)}{t^2} dt
    + O(1) .\]
  \item Usando \eqref{ej:prom_g} y el Corolario
    \ref{cor:promedio_sobre_primos}, pruebe que
    \begin{equation}\label{eq:medioprod}
      \begin{aligned}\sum_{p\leq h} \sum_{x/w<n\leq x} \frac{\lambda(n)
      \lambda(n+p)}{n} &\ll
    \frac{1}{(\log h)^{\frac{1}{75}-o(1)}} \frac{h \log w}{\log h}
    + \frac{h}{\log h}
    \end{aligned}\end{equation}
    para $1\leq w\leq x$ y $\log h\geq (\log \frac xw)^{1/3}$.
    Est\'a claro que el t\'ermino $h/\log h$ puede omitirse para
    $h\leq w$.
  \item Use (\ref{eq:medioprod}) y el apartado \eqref{ej:prom_f} para deducir
    el Corolario \ref{cor:promedio_sobre_primos2}.
    (Podemos, por cierto,
    suponer que $h$ es m\'as grande que una constante, pues
    de lo contrario lo que queremos probar es trivial.)
  \end{enumerate}
\end{enumerate}

\bibliographystyle{alpha}
\bibliography{helfubis}

\begin{thebibliography}{{Mon}71}

\bibitem[{Bon}36]{zbMATH03026527}
C.~E. {Bonferroni}.
\newblock {Teoria statistica delle classi e calcolo delle probabilita.}
\newblock {(Pubbl. d. R. Ist. Super. di Sci. Econom. e Commerciali di Firenze.
  8) Firenze: Libr. Internaz. Seeber. 62 S. (1936).}, 1936.

\bibitem[{Bru}15]{brun1915uber}
V.~{Brun}.
\newblock {\"Uber das {\it Goldbach}sche Gesetz und die Anzahl der
  Primzahlpaare.}
\newblock {\em {Arch. Math. Naturvid.}}, 34(8):3--19, 1915.

\bibitem[BSZ13]{MR2986954}
J.~Bourgain, P.~Sarnak, and T.~Ziegler.
\newblock Disjointness of {M}oebius from horocycle flows.
\newblock In {\em From {F}ourier analysis and number theory to {R}adon
  transforms and geometry}, volume~28 of {\em Dev. Math.}, pages 67--83.
  Springer, New York, 2013.

\bibitem[Che73]{MR0434997}
J.~R. Chen.
\newblock On the representation of a larger even integer as the sum of a prime
  and the product of at most two primes.
\newblock {\em Sci. Sinica}, 16:157--176, 1973.

\bibitem[DD82]{MR675168}
H.~Daboussi and H.~Delange.
\newblock On multiplicative arithmetical functions whose modulus does not
  exceed one.
\newblock {\em J. London Math. Soc. (2)}, 26(2):245--264, 1982.

\bibitem[FI10]{MR2647984}
J.~Friedlander and H.~Iwaniec.
\newblock {\em Opera de cribro}, volume~57 of {\em American Mathematical
  Society Colloquium Publications}.
\newblock American Mathematical Society, Providence, RI, 2010.

\bibitem[HR74]{MR0424730}
H.~Halberstam and H.-E. Richert.
\newblock {\em Sieve methods}.
\newblock Academic Press [A subsidiary of Harcourt Brace Jovanovich,
  Publishers], London-New York, 1974.
\newblock London Mathematical Society Monographs, No. 4.

\bibitem[Hux72]{MR0292774}
M.~N. Huxley.
\newblock On the difference between consecutive primes.
\newblock {\em Invent. Math.}, 15:164--170, 1972.

\bibitem[IK04]{MR2061214}
H.~Iwaniec and E.~Kowalski.
\newblock {\em Analytic number theory}, volume~53 of {\em American Mathematical
  Society Colloquium Publications}.
\newblock American Mathematical Society, Providence, RI, 2004.

\bibitem[K{\'a}t86]{MR836415}
I.~K{\'a}tai.
\newblock A remark on a theorem of {H}. {D}aboussi.
\newblock {\em Acta Math. Hungar.}, 47(1-2):223--225, 1986.

\bibitem[{Mon}71]{zbMATH03342902}
H.~L. {Montgomery}.
\newblock {\em {Topics in multiplicative number theory.}}, volume 227.
\newblock Springer, 1971.

\bibitem[Mon94]{MR1297543}
H.~L. Montgomery.
\newblock {\em Ten lectures on the interface between analytic number theory and
  harmonic analysis}, volume~84 of {\em CBMS Regional Conference Series in
  Mathematics}.
\newblock Published for the Conference Board of the Mathematical Sciences,
  Washington, DC; by the American Mathematical Society, Providence, RI, 1994.

\bibitem[Mot76]{MR0424726}
Y.~Motohashi.
\newblock On the sum of the {M}\"obius function in a short segment.
\newblock {\em Proc. Japan Acad.}, 52(9):477--479, 1976.

\bibitem[MR16]{MR3488742}
K.~Matom\"aki and M.~Radziwi\l\l.
\newblock Multiplicative functions in short intervals.
\newblock {\em Ann. of Math. (2)}, 183(3):1015--1056, 2016.

\bibitem[MRT15]{MR3435814}
K.~Matom\"aki, M.~Radziwi\l\l, and T.~Tao.
\newblock An averaged form of {C}howla's conjecture.
\newblock {\em Algebra Number Theory}, 9(9):2167--2196, 2015.

\bibitem[MRT16]{MR3513734}
K.~Matom\"aki, M.~Radziwi\l\l, and T.~Tao.
\newblock Sign patterns of the {L}iouville and {M}\"obius functions.
\newblock {\em Forum Math. Sigma}, 4:e14, 44, 2016.

\bibitem[MRT17]{mrt2017correlations}
K.~Matom{\"a}ki, M.~Radziwi{\l}{\l}, and T.~Tao.
\newblock {Correlations of the von {M}angoldt and higher divisor functions
  {II.} Divisor correlations in short ranges}.
\newblock {\em Mathematische Annalen}, pages 1--48, 2017.

\bibitem[MRT19]{zbMATH07035922}
K.~{Matom\"aki}, M.~Radziwi{\l}{\l}, and T.~{Tao}.
\newblock {Correlations of the von {M}angoldt and higher divisor functions {I.}
  {L}ong shift ranges.}
\newblock {\em {Proc. Lond. Math. Soc. (3)}}, 118(2):284--350, 2019.

\bibitem[MV77]{MR0457371}
H.~L. Montgomery and R.~C. Vaughan.
\newblock Exponential sums with multiplicative coefficients.
\newblock {\em Invent. Math.}, 43(1):69--82, 1977.

\bibitem[MV07]{MR2378655}
H.~L. Montgomery and R.~C. Vaughan.
\newblock {\em Multiplicative number theory. {I}. {C}lassical theory},
  volume~97 of {\em Cambridge Studies in Advanced Mathematics}.
\newblock Cambridge University Press, Cambridge, 2007.

\bibitem[Ram76]{MR0424723}
K.~Ramachandra.
\newblock Some problems of analytic number theory.
\newblock {\em Acta Arith.}, 31(4):313--324, 1976.

\bibitem[R{\'e}n47]{MR0021958}
A.~A. R{\'e}nyi.
\newblock On the representation of an even number as the sum of a single prime
  and a single almost-prime number.
\newblock {\em Doklady Akad. Nauk SSSR (N.S.)}, 56:455--458, 1947.

\bibitem[RS62]{MR0137689}
J.~Barkley Rosser and Lowell Schoenfeld.
\newblock Approximate formulas for some functions of prime numbers.
\newblock {\em Illinois J. Math.}, 6:64--94, 1962.

\bibitem[Sou17]{MR3666035}
K.~Soundararajan.
\newblock The {L}iouville function in short intervals.
\newblock {\em Ast\'erisque}, (390):Exp. No. 1119, 453--479, 2017.
\newblock S\'eminaire Bourbaki. Vol. 2015/2016. Expos\'es 1104--1119.

\bibitem[Tao16a]{MR3533300}
T.~Tao.
\newblock The {{E}rd\H{o}s} discrepancy problem.
\newblock {\em Discrete Anal.}, pages Paper No. 1, 29, 2016.

\bibitem[Tao16b]{MR3569059}
T.~Tao.
\newblock The logarithmically averaged {C}howla and {E}lliott conjectures for
  two-point correlations.
\newblock {\em Forum Math. Pi}, 4:e8, 36, 2016.

\end{thebibliography}




\end{document}